\documentclass[reqno]{amsart} 

\usepackage[utf8]{inputenc} 

\usepackage{amsfonts}
\usepackage{}

\usepackage[margin=1.2in]{geometry} 
\usepackage{amsmath,amssymb,amsthm,amsfonts}
\usepackage{mathrsfs}
\usepackage{bbm}

\usepackage[usenames,dvipsnames]{xcolor}
\usepackage[colorlinks=true, pdfstartview=FitV, linkcolor=blue, citecolor=blue, urlcolor=blue]{hyperref}

\usepackage{esint}
\usepackage{calrsfs}

\usepackage{doi}
\usepackage[english]{babel}

\usepackage{color}

\theoremstyle{plain}
\newtheorem{theorem}{Theorem}
\newtheorem{lemma}[theorem]{Lemma}

\newtheorem{proposition}[theorem]{Proposition}
\newtheorem{remark}[theorem]{Remark}

\numberwithin{equation}{section}
\numberwithin{theorem}{section}

\newcommand{\be}{\begin{equation}}
\newcommand{\bm}{\begin{multline}}
\newcommand{\ee}{\end{equation}}

\newcommand{\dis}{\displaystyle}

\newcommand{\C}{\mathbb{C}}

\newcommand{\R}{\mathbb{R}}

\newcommand{\Z}{\mathbb{Z}}

\renewcommand{\S}{\mathbb{S}}
\newcommand{\T}{\mathbb{T}}
\newcommand{\W}{\mathbb{W}}

\newcommand{\highG}{\Theta}
\newcommand{\highB}{\Lambda}

\newcommand{\FP}{\mathbf{P}}

\newcommand{\FI}{\mathbf{I}}

\newcommand{\CA}{\mathcal{A}}

\newcommand{\CD}{\mathcal {D}}
\newcommand{\CE}{\mathcal{E}}
\newcommand{\CF}{\mathcal{F}}
\newcommand{\CG}{\mathcal{G}}

\newcommand{\CL}{\mathcal{L}}
\newcommand{\CN}{\mathcal{N}}

\newcommand{\CR}{\mathcal{R}}
\newcommand{\CS}{\mathcal{S}}
\newcommand{\CM}{\mathcal{M}}

\newcommand{\bh}{\mathbbm{h}}

\newcommand{\na}{\nabla}

\newcommand{\al}{\alpha}

\newcommand{\ga}{\gamma}

\newcommand{\Om}{\Omega}
\newcommand{\la}{\lambda}
\newcommand{\de}{\delta}
\newcommand{\si}{\sigma}
\newcommand{\pa}{\partial}

\newcommand{\eps}{\epsilon}

\newcommand{\vth}{\vartheta}

\newcommand{\vps}{\varepsilon}

\newcommand{\Si}{\Sigma}

\newcommand{\Ga}{\Gamma}

\newcommand{\eqdef}{\overset{\mbox{\tiny{def}}}{=}}

\begin{document}
\title[Global mild solutions of  Landau and non-cutoff Boltzmann]{Global mild solutions of the Landau and non-cutoff Boltzmann equations}

\author[R.-J. Duan]{Renjun Duan}
\address[RJD]{Department of Mathematics, The Chinese University of Hong Kong,
Shatin, Hong Kong, P.R.~China}
\email{rjduan@math.cuhk.edu.hk}

\author[S.-Q. Liu]{Shuangqian Liu}
\address[SQL]{Department of Mathematics, Jinan University, Guangzhou 510632, P.R.~China}
\email{tsqliu@jnu.edu.cn}

\author[S. Sakamoto]{Shota Sakamoto}
\address[SS]{Mathematical Institute,
Tohoku University,
Sendai, 980-0856, Japan}
\email{shota.sakamoto.e1@tohoku.ac.jp}

\author[R. M. Strain]{Robert M. Strain}
\address[RMS]{Department of Mathematics, University of Pennsylvania,  David Rittenhouse Lab, 209 South 33rd Street, Philadelphia, PA 19104-6395, USA}
\email{strain@math.upenn.edu}


\begin{abstract}
This paper proves the existence of small-amplitude global-in-time unique mild solutions to both the Landau equation including the Coulomb potential and the Boltzmann equation without angular cutoff.  Since the well-known works \cite{Guo-L} and \cite{AMUXY-2012-JFA,GS} on the construction of classical solutions in smooth Sobolev spaces which in particular are regular in the spatial variables, it still remains an open problem to obtain global solutions in an $L^\infty_{x,v}$ framework, similar to that in \cite{Guo-2010}, for the Boltzmann equation with cutoff in general bounded domains.  One main difficulty arises from the interaction between the transport operator and the velocity-diffusion-type collision operator in the non-cutoff Boltzmann and Landau equations; another major difficulty is the potential formation of singularities for solutions to the boundary value problem. In the present work we introduce a new function space with low regularity in the spatial variable to treat the problem in cases when the spatial domain is either a torus, or a finite channel with boundary. For the latter case, either the inflow boundary condition or the specular reflection boundary condition is considered. An important property of the function space is that the $L^\infty_T L^2_v$ norm, in velocity and time, of the distribution function is in the Wiener algebra  $A(\Omega)$ in the spatial variables.  Besides the construction of global solutions in these function spaces, we additionally study the large-time behavior of solutions for both hard and soft potentials, and we further justify the property of propagation of regularity of solutions in the spatial variables.
\end{abstract}
\maketitle



\thispagestyle{empty}
\setcounter{tocdepth}{1}
\tableofcontents

\section{Introduction}
The Landau equation with Coulomb potential and the non-cutoff Boltzmann equation for the long range interaction potentials are two fundamental mathematical models in collisional kinetic theory which describe the dynamics of a non-equilibrium rarefied gas.  These two equations are connected in several ways; for example the Landau equation can be formally derived from the Boltzmann equation by taking the grazing limit.  Further the collision operators in both equations each feature velocity diffusion which can induce a gain of spatial regularity for spatially inhomogeneous solutions as a result of  the interplay with the free transport operator.  The existence theory for both equations has a long history.  A well-established framework in which to study global well-posedness is to look for solutions that are close to the global Maxwellian equilibria in different function spaces. However, for these two equations it is a big open problem to characterize the  optimal mathematical space of initial data with lower regularity in space and velocity variables such that unique solutions may exist globally in time. The main goal of this work is to prove the global existence of unique solutions in a new function space with mild regularity for both equations in the perturbative framework.

\subsection{Equation}
In this article we are concerned with the collisional kinetic equation of the form
\begin{equation}\label{Leq}
\pa_t F+v\cdot \na_x F=Q(F,F).
\end{equation}
Here the unknown $F=F(t,x,v)\geq 0$ stands for the density function of particles with position $x=(x_1,x_2,x_3)\in \Omega\subset \R^3$ and velocity $v=(v_1,v_2,v_3)\in \R^3$ at time $t>0$, and it is supplemented with initial data
\begin{equation*}
F(0,x,v)=F_0(x,v).
\end{equation*}

\subsection{Collision term}
The collision term on the right-hand side of \eqref{Leq} that we will consider in this paper is described by either the Landau operator or the non-cutoff Boltzmann operator which are given as follows.

\subsubsection{Landau collision operator}

For the Landau collision operator,
$Q(\cdot,\cdot)$ is given by
\begin{equation}\label{bLop}
Q(G,F)(v)=\nabla_v\cdot\left\{\int_{{\R}^{3}}\psi(v-u)\left[G(u)\nabla_vF(v)-F(v)\nabla_uG(u)\right]du\right\}.
\end{equation}
Then \eqref{Leq} with \eqref{bLop} is called the Landau equation.
Equivalently the Landau collision operator can be written as
\begin{equation}\label{bLop.e}
Q(G,F)(v)
=\sum\limits_{j,m=1}^{3}\partial_{v_j}\int_{{\R}^{3}}\psi^{jm}(v-u)\left[G(u)\partial_{v_m}F(v)-F(v)\partial_{u_m}G(u)\right]du.
\end{equation}
The Landau collision kernel $\psi$ in \eqref{bLop} or \eqref{bLop.e} is a non-negative symmetric matrix-valued function defined for $0\neq z=(z_1, z_2, z_3)\in \R^3$ as
\begin{equation}\label{def.Lker}
\psi^{jm}(z)=\left\{\delta_{jm}-\frac{z_jz_m}{|z|^2}\right\}|z|^{\gamma+2}, \quad j,m=1,2,3,
\end{equation}
where $\de_{jm}$ is the Kronecker delta and $-3\leq \gamma\leq 1$ is a parameter determined by the interaction potential between particles. It is customary to use the term hard potentials when $0<\ga\leq 1$, Maxwell molecules when $\ga=0$, moderately soft potentials when $-2\leq \ga<0$, and very soft potentials when $-3 \le \gamma < -2$ cf.~\cite{Vil02}. The case $\ga=-3$ corresponds to the classical Coulomb potential, cf.~\cite{Guo-L,Vil96}.  However see the discussion below \eqref{q} for the terminology that we will use in the rest of this paper.

\subsubsection{Non-cutoff Boltzmann collision operator}
For the Boltzmann collision operator without angular cutoff, $Q(\cdot,\cdot)$ takes the form of
\begin{equation}\label{bBop}
Q(G,F)(v)=\int_{\R^3}\int_{\S^2}B(v-u,\si)
\left[G(u')F(v')-G(u)F(v)\right]\,d\si
d u,
\end{equation}
where the velocity pairs $(v,u)$ and
$(v',u')$  satisfy
\begin{equation*}
\left\{\begin{aligned}
v'=\frac{v+u}{2}+\frac{|v-u|}{2}\si,\\
u'=\frac{v+u}{2}-\frac{|v-u|}{2}\si,
\end{aligned}\right.\quad \si\in\S^2.
\end{equation*}
Then \eqref{Leq} with \eqref{bBop} is called the Boltzmann equation.
The Boltzmann collision kernel $B(v-u,\si)$ is a non-negative function, depending only on the relative velocity
$|v-u|$ and on the deviation angle $\theta$ given by
\begin{equation*}
\cos \theta=\Big\langle\si,\frac{v-u}{|v-u|}\Big\rangle.
\end{equation*}
Without loss of generality we may assume
that $B(v-u,\si)$ is supported on $\cos\theta\geq 0$, i.e.~$0\leq \theta\leq \frac{\pi}{2}$. Otherwise we can
reduce to this situation with the following standard symmetrization, {cf.~\cite{Gla,Vil98}}:
\begin{equation*}
B(v-u,\si) = \left[B(v-u, \si) + B(v-u,-\si)\right]{\bf 1}_{\cos\theta\geq 0}.
\end{equation*}
Here and in the rest of the paper, ${\bf 1}_A$ denotes the usual indicator function of the set $A$. Throughout this paper we further assume that  $B(v-u,\si)$ takes the product form as follows:
\begin{equation*}
B(v-u,\si)=C_B |v-u|^{\ga}b(\cos\theta),
\end{equation*}
for a constant $C_B>0$, where $ |v-u|^{\ga}$ is called the kinetic part with $\gamma>-3$,
and $b(\cos\theta)$ is called the angular part satisfying that there are $C_b>0$, $0<s<1$ such that
\begin{equation*}
\frac{1}{C_b\theta^{1+2s}}\leq \sin\theta b(\cos\theta)\leq \frac{C_b}{\theta^{1
+2s}},\quad \forall\,\theta\in (0,\frac{\pi}{2}].
\end{equation*}
For the angular non-cutoff Boltzmann operator, it is also customary to use the term hard potentials when $\ga \ge 0$, moderately soft potentials when  $0>\ga\geq -2s$, and very soft potentials when $-3<\ga<-2s$.  However see the discussion below \eqref{q} for the terminology that we will use in the rest of this paper.

One physical example is given when the collision kernel is derived from a spherical intermolecular repulsive potential of the inverse power law form $\phi(r)=r^{-(\ell-1)}$ with $2<\ell<\infty$ corresponding to which $B$ satisfies the assumptions above with $\ga=(\ell-5)/(\ell-1)$ and $s=1/(\ell-1)$, cf.~\cite{CIP}.

Through the paper, in the Boltzmann case we further require
\begin{equation*}
\ga>\max\left\{-3,-\frac 32-2s\right\},
\end{equation*}
due to the mild regularity setting of the results in this article.
Indeed, the above condition is satisfied for the full range of the inverse power law model.

\subsection{Spatial domain and boundary condition} In this paper we focus on two kinds of specific bounded domains $\Omega\subset\R^3$.  We consider either a torus, or a finite channel with prescribed boundary conditions.

\subsubsection{Case of the Torus} In this case, we set
\begin{equation}
\label{d.pb}
\Omega=\T^3:=[0,2\pi]^{3}.
\end{equation}
Correspondingly, $F(t,x,v)$ is assumed to be spatially periodic in $x\in \T^3$.

\subsubsection{Case of the Finite channel} In this case we set
\begin{equation}
\label{d.fc}
\Omega=I\times \T^2=\left\{x=(x_1,\bar{x}),x_1\in I:=(-1,1),\bar{x}:=(x_2,x_3)\in \T^2=[0,2\pi]^2\right\}.
\end{equation}
Correspondingly, $F(t,x,v)$ is assumed to be spatially periodic for $\bar{x}\in \T^2$ and satisfy the following two kinds of boundary conditions at $x_1=\pm 1$:

\begin{itemize}
\item {\it Inflow} boundary condition:
\begin{equation}\label{bcF.i}
F(t,-1,\bar{x},v)|_{v_1>0}=G_-(t,\bar{x},v),\ \ F(t,1,\bar{x},v)|_{v_1<0}=G_+(t,\bar{x},v).
\end{equation}

\item {\it Specular reflection} boundary condition:
\begin{equation}\label{bcF.sr}
\begin{aligned}
F(t,-1,\bar{x},v_1,\bar{v})|_{v_1>0}&=F(t,-1,\bar{x},-v_1,\bar{v}),\\
F(t,1,\bar{x},v_1,\bar{v})|_{v_1<0}&=F(t,1,\bar{x},-v_1,\bar{v}),
\end{aligned}
\end{equation}
where $\bar{v}=(v_2,v_3)$.

\end{itemize}

\subsection{Reformulation}

Consider the following global Maxwellian equilibrium state:
\begin{equation*}
\mu=\mu(v)=(2\pi)^{-3/2}e^{-|v|^2/2}.
\end{equation*}
Note that $\mu$ is a spatially homogeneous steady solution to  \eqref{Leq} for either the Landau case or the Boltzmann case. We look for a solution to \eqref{Leq} of the form
\begin{equation*}
F(t,x,v)=\mu+{\mu}^{\frac{1}{2}}f(t,x,v).
\end{equation*}
Then, the new unknown $f=f(t,x,v)$ satisfies
\begin{equation}\label{LLeq}
\pa_tf+v\cdot\na_xf
+Lf=\Ga(f,f),
\end{equation}
with initial data
\begin{equation}\label{idf}
f(0,x,v)=f_0(x,v):=\mu^{-1/2}[F_0(x,v)-\mu].
\end{equation}
Here, the linearized collision operator $L$ and nonlinear collision operator $\Ga$ are given by
\begin{equation}\label{Ldef}
Lf=-\mu^{-\frac{1}{2}}\left\{Q(\mu,\mu^{\frac{1}{2}}f)+Q(\mu^{\frac{1}{2}}f,\mu)\right\},
\end{equation}
and
\begin{equation}\label{nopdef}
\Gamma(f,f)=\mu^{-\frac{1}{2}}Q(\mu^{\frac{1}{2}}f,\mu^{\frac{1}{2}}f),
\end{equation}
respectively.  In the case of a finite channel, corresponding to \eqref{bcF.i} or \eqref{bcF.sr},
 the boundary conditions along $x_1$ are given by
\begin{equation}\label{ifb}
f(t,-1,\bar{x},v)|_{v_1>0}=g_-(t,\bar{x},v),\ \ f(t,1,\bar{x},v)|_{v_1<0}=g_+(t,\bar{x},v),
\end{equation}
for the inflow boundary, and
\begin{equation}\label{srb}
\begin{split}
f(t,-1,\bar{x},v_1,\bar{v})|_{v_1>0}&=f(t,-1,\bar{x},-v_1,\bar{v}),
\\
f(t,1,\bar{x},v_1,\bar{v})|_{v_1<0}&=f(t,1,\bar{x},-v_1,\bar{v}),
\end{split}
\end{equation}
for the specular reflection boundary.

\subsection{Problem} In this paper we will study the following two problems {\bf (PT)} and {\bf (PC)} as described below.  For each problem, both the Landau operator \eqref{bLop} and the Boltzmann operator \eqref{bBop} will be included in our study.

\begin{itemize}
  \item[{\bf (PT):}]  Consider \eqref{LLeq} and \eqref{idf} in the torus domain $\Omega$ as in \eqref{d.pb}. Due to the conservation laws, we may always assume that $f(t,x,v)$ satisfies
\begin{eqnarray}
&&\int_{\T^3}\int_{\R^3} \mu^{\frac{1}{2}}f(t,x,v)\,dvdx=0,\label{pt.id.cl.1}
\\
&&\int_{\T^3}\int_{\R^3} v_i\mu^{\frac{1}{2}}f(t,x,v)\,dvdx=0,\quad i=1,2,3,\label{pt.id.cl.2}
\\
&&\int_{\T^3}\int_{\R^3}|v|^2 \mu^{\frac{1}{2}}f(t,x,v)\,dvdx=0,\label{pt.id.cl.3}
\end{eqnarray}
for any $t\geq 0$. In other words if these are satisfied initially then the equation shows that they will continue to be satisfied for any $t\geq 0$.

  \item[{\bf (PC):} ] Consider \eqref{LLeq}, \eqref{idf}, and \eqref{ifb} or \eqref{srb} in a finite channel domain $\Omega$ as in \eqref{d.fc}.

  For the specular reflection boundary condition \eqref{srb}, we assume further that
\begin{eqnarray}
&&\int_{-1}^1\int_{\T^2}\int_{\R^3} \mu^{\frac{1}{2}}f(t,x_1,\bar{x},v)\,dvd\bar{x} dx_1=0,\label{PC.cl1}
\\
&&\int_{-1}^1\int_{\T^2}\int_{\R^3} v_i\mu^{\frac{1}{2}}f(t,x_1,\bar{x},v)\,dvd\bar{x} dx_1=0,
\quad i=1,2,3,
\label{PC.cl2}
\\
&&\int_{-1}^1\int_{\T^2}\int_{\R^3}|v|^2 \mu^{\frac{1}{2}}f(t,x_1,\bar{x},v)\,dvd\bar{x} dx_1=0,
\label{PC.cl3}
\end{eqnarray}
for any $t\geq 0$. Indeed, \eqref{PC.cl1}, \eqref{PC.cl2} with $i=2,3$,  and \eqref{PC.cl3} are the direct consequence of the  conservation of mass, momentum and energy, respectively. To obtain \eqref{PC.cl2} for $i=1$, we further assume that the initial data satisfy the following symmetry assumption:
\begin{equation}
\label{ass.sym}
F_0(x_1,\bar{x},v_1,\bar{v})=F_0(-x_1,\bar{x},-v_1,\bar{v}),
\end{equation}
and hence the solution will also satisfy
\begin{equation*}
F(t,x_1,\bar{x},v_1,\bar{v})=F(t,-x_1,\bar{x},-v_1,\bar{v}),
\end{equation*}
which directly gives \eqref{PC.cl2} when $i=1$.  We only assume the symmetry condition \eqref{ass.sym} for the specular reflection boundary condition \eqref{srb}.

\end{itemize}

\subsection{Function space}

To study the well-posedness of the problems {\bf (PT)} and {\bf (PC)}, we first introduce a function space $X_T$ with $0<T\leq \infty$, which is a key point in  this work. We will present the motivation for this space later below. For the problem {\bf (PT)} in a torus $\T^3$, we define
 \begin{equation*}
X_T:=L^1_kL^\infty_TL^2_v
\end{equation*}
with norm
\begin{equation}
\label{def.XT}
\|f\|_{X_T}:=\int_{\Z^3}\sup_{0< t< T} \|\hat{f}(t,k,\cdot)\|_{L^2_v}\,d\Sigma (k)<\infty,
\end{equation}
where the Fourier transformation of $f(t,x,v)$ with respect to $x\in \T^3$ is denoted
\begin{equation*}
\hat{f}(t,k,v)=\CF_x f(t,k,v)=\int_{\T^3} e^{-  i k\cdot x}f(t,x,v)\,dx,\quad k\in \Z^3,
\end{equation*}
In this paper, for convenience we would use $d\Sigma(k)$ to denote the discrete measure in $\Z^3$, namely,
\begin{equation*}
\int_{\Z^3}g(k)\,d\Sigma(k)=\sum_{k\in \Z^3}g(k).
\end{equation*}

\begin{remark}
The main reason for using the integral form $\int_{\Z^3} d\Sigma(k)$ is to make it clear that when the torus domain $\T^3_x$ changes to the whole space domain $\R^3_x$, one may directly change $d\Si (k)$ over $k\in \Z^3$ into $dk$ over $k\in \R_k^3$ and the similar results in this paper also hold true in $\R^3_x$ after taking into account the appropriate modifications to the time-decay of solutions.
\end{remark}

Now, similar to the above, for the problem {\bf (PC)} in a finite channel, we define
\begin{equation*}
X_T:=L^1_{\bar{k}}L^\infty_TL^2_{x_1,v}
\end{equation*}
with norm
\begin{equation*}
\|f\|_{X_T}:=\int_{\Z^2_{\bar{k}}}\sup_{0< t< T} \|\hat{f}(t,\bar{k},\cdot)\|_{L^2_{x_1,v}}\,d\Sigma(\bar{k})<\infty,
\end{equation*}
with $\bar{k} = (\bar{k}_2, \bar{k}_3)\in \Z^2$ and $\bar{x}=(x_2,x_3)\in \T^2$ we take the Fourier transform as
\begin{equation*}
\hat{f}(t,x_1,\bar{k},v)=\CF_{\bar{x}} f(t,x_1,\bar{k},v)=\int_{\T^2} e^{-  i \bar{k}\cdot \bar{x}}f(t,x_1,\bar{x},v)\,d\bar{x},\quad \bar{k}\in \Z^2.
\end{equation*}
Notice that Fourier transform is taken only with respect to $\bar{x}=(x_2,x_3)$.

We also introduce the velocity weighted norm
\begin{align}
\|f\|_{X^w_T}:=&\int_{\Z^3}\sup_{0\leq t\leq T} \|w\hat{f}(t,k,\cdot)\|_{L^2_v}\,d\Sigma(k)\notag\\
&\text{ or }\int_{\Z^2}\sup_{0\leq t\leq T} \|w\hat{f}(t,\bar{k},\cdot)\|_{L^2_{x_1,v}}\,d\Sigma(\bar{k}),\label{def.xtw}
\end{align}
for {\bf (PT)} or {\bf (PC)}, respectively. Here $w=w_{q,\vth}(v)$ is a velocity weight function defined as
\begin{equation}
\label{def.w}
w_{q,\vth}(v)=e^{\frac{q\langle v\rangle^\vth}{4}},\quad \langle v\rangle=\sqrt{1+|v|^2},
\end{equation}
with two parameters $q$ and $\vth$.
Note that whenever there is no confusion,  we would omit the dependence of  $w_{q,\vth}(v)$ on $q$ and $\vth$, and write $w_{q,\vth}(v)=w(v)$ for brevity. In the situation when $q=0$, we have $w_{q,\vth}(v)\equiv 1$ and hence the function space $X^{w}_T$ with velocity weight is reduced to $X_T$ without weight. Throughout this paper we require that $(q,\vth)$ satisfies the following hypothesis in terms of $\ga$ and $s$:
\begin{eqnarray}\label{q}
\text{\bf (H)}\left\{\begin{array}{rl}
&\text{Landau case: if $-2\leq \ga\leq 1$ then $q=0$;}\\
&\qquad\qquad\qquad\,\text{if $-3\leq \ga<-2$ then $q> 0$ and $0<\vth\leq 2$ with}\\
&\qquad\qquad\qquad\,\text{the restriction that $0< q<1$ if $\vth=2$. 
}
\\[2mm]
&\text{Boltzmann case: if $\ga+2s\geq 0$ then $q=0$;} \\
&\qquad\qquad\qquad\quad\ \ \text{if $-3<\ga<-2s$ then $q> 0$ and $\vth=1$.}
\end{array}\right.
\end{eqnarray}
For convenience of terminology in relation to \eqref{q}, in the rest of this paper we will call them {\it soft} potentials if either $\ga+2<0$ in Landau case or $\ga+2s<0$ in Boltzmann case, and we will call them the {\it hard} potentials otherwise.

In general we define the $L^p_kL^q_TL^r_v$ mixed Lebesgue space norms as follows.
\begin{equation*}
\|f\|_{L^p_kL^q_TL^r_v}
:=
\left( \int_{\Z^3_{k}}
\left(
\int_0^T
\left(
\int_{\R^3}
\left| \hat{f}(t,k,v)\right|^r
dv
\right)^{q/r} dt
\right)^{p/q}
d\Sigma(k)
\right)^{1/p},
\end{equation*}
where $1\le p,q,r < \infty$ and we use the standard modification of the above when any of $p$, $q$ or $r$ is equal to  $\infty$.  Analogously for the case of a finite channel in problem {\bf (PC)} we will define standard $L^p_{\bar{k}}L^q_TL^r_{x_1,v}$ mixed Lebesgue space norms as
\begin{multline*}
\|f\|_{L^p_{\bar{k}}L^q_TL^r_{x_1,v}}
:=
\\
\left( \int_{\Z^2_{\bar{k}}}
\left(
\int_0^T
\left(
\int_{-1}^1
\int_{\R^3}w
\left| \hat{f}(t,x_1,{\bar{k}},v)\right|^r
dv dx_1
\right)^{q/r} dt
\right)^{p/q}
d\Sigma(\bar{k})
\right)^{1/p},
\end{multline*}
where again $1\le p,q,r < \infty$ and we use the same standard modification when any of $p$, $q$ or $r$ is equal to  $\infty$.  In this paper we typically use $p=1$, $q=2$ or $q=\infty$, and $r=2$ in the above definitions.

We note further that for ease of the notation we will use the general norm notation $\| \cdot \|$ in different sections to denote different norms, the specific norm that we are using is identified within the section before it is used each time.

\subsection{Goal}

Using the above preliminary notations, for the problems {\bf (PT)} and {\bf (PC)} our main goal of this paper  is to study the following three issues:

\medskip
\noindent$\bullet$ {\it Global-in-time existence of small amplitude solutions in $X_T$.}
For the problem {\bf (PT)} we are able to show that if
\begin{equation*}
\|wf_0\|_{L^1_k L^2_v}:=\int_{\Z^3} \|w\hat{f_0}(k,\cdot)\|_{L^2_v}\,d\Sigma(k)
\end{equation*}
is suitably small and $F_0(x,v)=\mu+\mu^{\frac{1}{2}}f_0(x,v)\geq 0$, then {\bf (PT)} admits a unique global-in-time mild solution $f(t,x,v)$ such that $F(t,x,v)=\mu+{\mu}^{\frac{1}{2}} f(t,x,v)\geq 0$, and $\|f\|_{X^w_T}$  is uniformly bounded for all $T>0$.  We prove similar results for the problem {\bf (PC)} with physical boundaries when the solution space is replaced by
$$
\|f\|_{X^w_T}+{\|\na_xf\|_{X^w_T}}.
$$
We will give the precise statements for all of these results in the next section.

\medskip
\noindent$\bullet$ {\it Large-time behavior of solutions.}  To obtain the rate of convergence, under the hypothesis {\bf (H)}, associated with the velocity weight function $w=w_{q,\vth}(v)$ and $\ga$, we define a parameter $\kappa$ in the Landau case as
\begin{equation}
\label{def.kap1}
\kappa=\left\{
\begin{array}{ll}
 1&\text{ for $q= 0$, $-2\leq \ga\leq 1$},\\[3mm]
 \dis \frac{\vth}{\vth+|\ga+2|} &\text{ for $q>0$, $-3\leq \ga<-2$},
\end{array}\right.
\end{equation}
and in the non-cutoff Boltzmann case we define
\begin{equation}
\label{def.kap2}
\kappa=\left\{
\begin{array}{ll}
 1&\text{ for $q= 0$, $ \ga+2s\geq 0$},\\[3mm]
 \dis \frac{\vth}{\vth+|\ga+2s|} &\text{ for $q>0$, $-3<\ga<-2s$}, ~\vth =1.
\end{array}\right.
\end{equation}
Note that $\kappa\in (0,1]$, and $0< \vth \le 2$ in \eqref{q}.  Additionally we use $\vth =1$ only in the Boltzmann case as in \eqref{q} and \eqref{def.kap2} because of Lemma \ref{bnp.es}  from \cite{DLYZ-VMB}.

We then are able to show that the obtained solutions decay in time as follows:
\begin{equation*}
\|f(t)\|_{L^1_kL^2_v}\lesssim e^{-\la t^\kappa},
\end{equation*}
 for the problem {\bf (PT)}, and
 \begin{equation*}
\|f(t)\|_{L^1_{\bar{k}}L^2_{x_1,v}}+\|\na_x f(t)\|_{L^1_{\bar{k}}L^2_{x_1,v}}\lesssim e^{-\la t^\kappa}
\end{equation*}
for the problem {\bf (PC)},  where $\la>0$ is a suitably small uniform constant that is independent of time.\\

\noindent$\bullet$  {\it Propagation of spatial regularity.} For the problem {\bf (PT)}, we will establish the propagation of spatial regularity of solutions. Specifically, we are able to show that if the initial data $f_0$ additionally satisfies that $\|\langle \na_x\rangle^m f_0\|_{L^1_kL^2_v}$ is small enough
with $m\geq 0$, then the solution $f(t,x,v)$ satisfies that  $\|\langle \na_x\rangle^m f\|_{L^1_kL^\infty_TL^2_v}$ is bounded uniformly for all $T>0$, where $\langle \na_x\rangle^m f=\CF_x^{-1} (\langle k\rangle^m \hat{f})$. For the problem {\bf (PC)}, similar results also hold true but instead they regard the propagation of $\bar{x}$-regularity, namely, along the tangential direction of $x_1\in (-1,1)$, and we will give the precise statement of these results in the next section.

\subsection{Motivation of the current work}

The global-in-time existence theory in the perturbation framework for the Landau equation and the Boltzmann equation without angular cutoff has been well established  for suitably small initial data in smooth Sobolev spaces. Particularly, for the Landau equation in a periodic box with $\ga\geq -3$ including the physical Coulomb potential, Guo \cite{Guo-L} gave the first result on the construction of global classical solutions near Maxwellian with $f(t,x,v)\in L^\infty (0,\infty; H^8_{x,v})$. As pointed out in \cite[page 393]{Guo-L}, such a high Sobolev norm is a standard choice due to the nonlinearity of the Landau operator since the $L^\infty$ norm in $x$ is easily controlled by the Sobolev embedding and then a high Sobolev norm of a product is bounded by a product of the same norms. For the Boltzmann equation without angular cutoff in a periodic box, Gressman-Strain \cite{Gra-Pr} first constructed the global small-amplitude solution $f(t,x,v)$ in $L^\infty (0,\infty; L^2_vH^2_x)$ for hard potential case $\ga+2s\geq 0$ (including $\ga+2s\geq -\frac{3}{2}$) and in $L^\infty (0,\infty; H^4_{x,v})$ for the general soft potential case $\ga+2s<0$. The similar result was also independently obtained by AMUXY  in their series of works (cf.~\cite{AMUXY-2011-CMP,AMUXY-2012-JFA}, for instance) in case of the whole space. A key point in those well-known works is to characterize the dissipation property in the $L^2$ norm in $v$ for the linearized self-adjoint operator (cf.~\cite{ADVW,MS}) and further carry out the energy estimates by controlling the trilinear term in an appropriate  way.

Regarding the solution space, a natural question is to ask whether or not it is possible to construct the global solutions $f(t,x,v)$ in some function spaces with much lower regularity in space and velocity.  In fact, we remark that the local existence with mild regularity for the Boltzmann equation with or without cutoff was discussed in AMUXY \cite{AMUXY13KRM} in the function space $L^\infty (0,T_\ast; L^2_vH^s_x)$ with $s>3/2$ and some $T_\ast>0$ for even large initial data. Note that if one uses the Sobolev space $L^2_vH^s_x$ it seems necessary to require $s>3/2$ in order to obtain the $L^\infty$ bound in $x$ by the imedding. Motivated by \cite{AMUXY13KRM} as well as \cite{SS}, for the Boltzmann equation with angular cutoff in the whole space, Duan-Liu-Xu \cite{DLX-2016} found a Chemin-Lerner-type space $\widetilde{L}_T^\infty B^{3/2}_x L^2_v$ (cf.~\cite{CL} and \cite{BCD}) such that the solution of small-amplitude can be bounded in this function space uniformly in time $T>0$, where $B_x^s=B^s_{2,1} (\R^3_x)$ for  $s\geq 0$  denotes the usual Besov space in $x$, and $f(t,x,v)\in \widetilde{L}_T^\infty B^{s}_x L^2_v$ means
\begin{equation*}
\left\|\left(2^{qs} \sup_{0\leq t\leq T} \|\Delta_qf(t,\cdot,\cdot)\|_{L^2_{x,v}}\right)_{q\geq -1}\right\|_{\ell_q^1}=\sum_{q\geq -1}2^{qs} \sup_{0\leq t\leq T} \|\Delta_qf(t,\cdot,\cdot)\|_{L^2_{x,v}}<\infty.
\end{equation*}
Above the $\Delta_q$ are the standard frequency projection operators in Fourier space.
A key observation is that the $L^\infty$ norm in $x$ still can be controlled by the embedding $B_x^{3/2}\subset L^\infty_x$ in three dimensions. One feature of the Chemin-Lerner-type space is that the supremum over $0< t< T$ is taken before the $\ell_q^1$ norm, and this feature plays a role in controlling the quadratically nonlinear term in terms of the chosen energy functional and the dissipation rate functional. Later, Morimoto-Sakamoto \cite{MS-2016-JDE} carried over the same idea to treat the Boltzmann equation without angular cutoff for the hard potentials $\ga+2s\geq 0$, and Duan-Sakamoto \cite{DSa} further studied the case of soft potentials $\ga+2s<0$. We remark that it seems still true that those results in \cite{DSa,MS-2016-JDE}   can be also obtained for the Landau equation with $\ga\geq -3$ including the Coulomb potential.

On the other hand, we notice that the space $L^\infty_{x,v}$ with suitable velocity weight seems the simplest one with mild regularity in space and velocity variables for directly treating the nonlinearity to obtain the global existence; see an $L^2\cap L^\infty$ approach by Ukai-Yang \cite{UY} based on spectral theory and the decomposition $L=-\nu+K$ (cf.~\cite{Gra,Gra-Pr}) for the Boltzmann equation with cutoff hard potentials in the whole space. In fact, for the Boltzmann equation with cutoff hard potentials in the  torus or even in a general bounded domain, Guo \cite{Guo-2010} developed a mathematical theory of global existence of small-amplitude $L^\infty_{x,v}$ solutions by an $L^2$-$L^\infty$ interplay method. Since then, there have been extensive generalizations and further developments of the Boltzmann global existence theory in the angular cutoff setting, for instance, \cite{BG-JDE,Cao2019,DHWY,DW-2019,GKTT-Invention,GL-2017,Kim-2011-CMP,KL-2018-CPAM,LY-2016, Ni}. Particularly, motivated by \cite{Guo-2010} and \cite{DHWY}, Nishimura \cite{Ni} provided an interesting result on global existence of the angular cutoff Boltzmann equation in the space $L^p_vL^\infty_{t,x}$ with suitable velocity weight for $1<p\leq \infty$; this work extends to the case of finite $p<\infty$. The case of $p=2$ is of its own interest because this means that it may no longer be necessary to study the interplay between $L^2_v$ and $L^p_v$ if $p\neq 2$.


However, so far there are not many results on the  global  existence theory in $L^\infty_{x,v}$  for the angular non-cutoff Boltzmann equation. The main reason is that the collision operator in the non-cutoff case exposes the fractional velocity diffusion property so that it is very difficult to apply the characteristic approach as in \cite{Guo-2010} to obtain the $L^\infty$ bounds in $x$ for the solutions. The same situation occurs to the Landau equation with velocity diffusion. Recently Kim-Guo-Hwang \cite{KGH} developed an $L^2$ to $L^\infty$ approach to the Landau equation in the torus domain, where initial data  are required to be small in $L^\infty_{x,v}$ but additionally belong to $H^1_{x,v}$. Since the Sobolev embedding can no longer be used, a key point of the proof in \cite{KGH} is to control the $L^\infty$ bound by the $L^2$ estimates via De Giorgi's method \cite{GIMV}, and also to control the velocity derivatives to ensure uniqueness by the H{\"o}lder estimates again via De Giorgi's method \cite{GIMV}.   The De Giorgi's method applied to the Landau equation has been recently developed in Golse-Imbert-Mouhot-Vasseur \cite{GIMV}. Notice that for the Boltzmann equation without cutoff, the regularity issue has also been  studied in many recent works, for instance, Silvestre \cite{Sil}, Imbert-Silvestre \cite{IS}, Imbert-Mouhot-Silvestre \cite{IMS}, and Chen-Hu-Li-Zhan \cite{CHLZ}. See also a recent survey by Mouhot \cite{Mou} and references therein.


Moreover, we observe that there recently have been several research works to apply the Wiener algebra $A(\Omega)$, or in the notation of this paper $L^1_k$, to study the global-in-time existence and gain of analyticity for solutions to some evolution equations with diffusions, for instance, Constantin-C\'ordoba-Gancedo-Strain \cite{CCGS-13}, Constantin - C{\'o}rdoba - Gancedo - Rodr\'{i}guez-Piazza - Strain \cite{CCGRS-16}, Patel-Strain \cite{PS-17}, Gancedo - Garc\'{i}a-Ju\'{a}rez - Patel - Strain \cite{GGPS}, Liu-Strain \cite{LS}, Granero-Belinch{\'o}n - Magliocca \cite{1804.09645}, and Ambrose \cite{Am}. See also a recent work by Lei-Lin \cite{LL} for the global well-posedness of mild solutions to the three-dimensional,
incompressible Navier-Stokes equations if the initial data satisfy $\na_x f_0\in L^1_k$. Here a function $f(x)$ with $x\in \T^3$ is in $L^1_k$ if its Fourier transform $\hat{f}(k)$ is summable in $k\in \Z^3$, namely,
\begin{equation*}
\int_{\Z^3_k} |\hat{f}(k)|\,d\Si (k)=\sum_{k\in \Z^3_k} |\hat{f}(k)|<\infty.
\end{equation*}
It is obvious to see that for two functions $f(x)$ and $g(x)$ with $x\in \T^3$,
\begin{equation*}
\|fg\|_{L^1_k}=\int_{\Z^3}d\Si (k)\, |\widehat{fg}(k)| =\int_{\Z^3_k}d\Si (k)\,\int_{\Z^3_l}d\Si (l)\, |\hat{f}(k-l)\hat{g} (l)|\leq  \|f\|_{L^1_k}\|g\|_{L^1_k},
\end{equation*}
where we have used Fubini's Theorem in the last inequality. Thus using this Banach algebra property one may expect that the $L^1_k$ norm can play a similar role to the $L^\infty_x$ norm when studying the Landau equation or the non-cutoff Boltzmann equation. Note that, as pointed out in \cite{Am}, the space $L^1_kL^\infty_T$ is also a Banach algebra, where $f=f(t,x)$ with $t\geq 0$ and $x\in \T^3$ is in $L^1_kL^\infty_T$ if
\begin{equation*}
\int_{\Z^3} \sup_{0\le t\le T} |\hat{f}(t,k)|\,d\Si (k)<\infty.
\end{equation*}
Those observations together with the aforementioned works \cite{DLX-2016,DSa, MS-2016-JDE} motivated us to introduce  the function space $X_T=L^1_kL^\infty_TL^2_v$ with the norm $\|\cdot\|_{X_T}$  as defined previously in \eqref{def.XT}.

\subsection{Related literature}

In what follows we recall some known results on the Landau and Boltzmann equations with a focus on the topics under consideration in this paper, particularly on global existence and large-time behavior of solutions to the spatially inhomogeneous equations in the perturbation framework. For global solutions to the renormalized equation with large initial data, we mention the classical works by DiPerna-Lions \cite{DiPLi89,DiPLi88}, Lions \cite{Li94}, Villani \cite{Vil98,Vil96}, Desvillettes-Villani \cite{DV}, and Alexandre-Villani \cite{AVil}.

When the spatial domain is either the whole space or a torus, we mention Ukai \cite{Uk,Ukai-86}, Caflisch \cite{Caf1,Ca-1980}, Guo \cite{Guo-03,GuoIU04}, Liu-Yang-Yu \cite{LYY}, and Strain-Guo \cite{SG-CPDE,SG-08-ARMA} for the Boltzmann equation with cutoff. In the non-cutoff case, besides those works \cite{AMUXY-2011-CMP,AMUXY-2012-JFA,GS}  mentioned before, we also mention the recent works by Alonso-Morimoto-Sun-Yang \cite{AMSY}, He-Jiang \cite{HeJi}, and H\'erau-Tonon-Tristani \cite{HTT} for global existence and large-time behavior of solutions with the polynomial velocity weight for the Boltzmann equation in the torus, as well as Carrapatoso-Mischler \cite{CM} in the case of the Landau equation. Here, \cite{AMSY,CM,HTT}  are motivated by the recent work Gualdani-Mischler-Mouhot \cite{GMM} in the case of the cutoff Boltzmann equation in the torus, while \cite{HeJi} is based on a new energy-entropy approach in terms of quasi-linear equations.

Moreover, for the large-time behavior of solutions in soft potential cases, particularly for the optimal time-decay rates, we mention Caflisch \cite{Ca-1980}, Strain-Guo \cite{SG-CPDE,SG-08-ARMA}, and Strain \cite{Str12}, see also Sohinger-Strain \cite{SS}, in different settings.  We also mention the recent numerical study on the possible sharp $2/3$ rate of  large time decay for the Landau equation with Coulomb interactions in Bobylev-Gamba-Zhang \cite{MR3670754};  we obtain this decay rate in \eqref{Torus.Time.Decay.Est}, \eqref{ifdec} and \eqref{srdec} with \eqref{def.kap1} when
$\gamma = -3$ and $\vth = 2$.

Since one part of this paper is concerned with the boundary value problem on the Landau and Boltzmann equations, we would also make some comments on the related known results in perturbation framework. We mainly focus on the case of the Boltzmann equation, as it seems there are few results in case of the Landau equation with boundaries. In fact,  the mathematical study for the the boundary value problems of the Boltzmann equation can date back to 1960s. Cercignani \cite{Cer-67,Cer-68} proved the existence and uniqueness for the inflow boundary value problem of the discrete time linearized Boltzmann equation with angular cutoff in two parallel plates, which was immediately extended to the nonlinear case by Pao \cite{Pao-67}. With the aid of these approaches as well as the work by Caflish \cite{Ca-1980}, Esposito-Lebowitz-Marra \cite{ELM-94,ELM-95} studied the hydrodynamic limits of the stationary Boltzmann equation with diffuse reflection boundary condition in a slab. For the general bounded domains, Shizuta-Asano \cite{SA-77} in 1977 announced that the global existence and time decay rates to the equilibrium for the Boltzmann equation with the specular reflection boundary condition could be established by using the  Vidav's multi-iteration method \cite{Vi}. Almost ten years later, Ukai \cite{Ukai-86} in 1986 constructed the famous trace theorem
for the Boltzmann equation, which was then used by Hamdache \cite{Ha-92} to construct the normalized weak solutions for the Boltzmann equation with Maxwell boundary condition. Later on, Ukai's trace theorem was improved by Cercignani \cite{Cer-92} so that Hamdache's result
could be eventually extended to more general cases. As mentioned before, with the foundational work by Guo \cite{Guo-2010} on  the $L^2$-$L^\infty$ method, a great many of achievements have been made over the past decades in the study of the initial boundary value problems for the kinetic equations, particularly for the Boltzmann equation. For instance, Esposito-Guo-Kim-Marra \cite{EGKM-13} constructed a genuine non-equilibrium stationary solution. Briant-Guo \cite{BG-JDE} partially proved the stability of the Boltzmann equation with Maxwell boundary condition. Kim-Lee \cite{KL-2018-CPAM} removed the analytical boundary condition which was required in \cite{Guo-2010} for the study of the sole specular reflection boundary value problem by iterating the Duhamel's formula three times. Liu-Yang \cite{LY-2016} extended the results  in \cite{Guo-2010} to the cutoff soft potential case. Duan-Wang \cite{DW-2019} showed the global well-posedness for a class of large-amplitude initial data. Guo-Liu \cite{GL-2017} constructed the global existence for the Boltzmann equation with specular reflection boundary condition around the polynomial initial data differing from a global equilibrium Maxwellian. Esposito-Guo-Kim-Marra \cite{EGKM-18-AP} and  Duan-Liu \cite{DL-2018} studied the hydrodynamic limits of the Boltzmann equation in bounded domains based on an improved $L^2$-$L^\infty$ method.

The main difficulty to study the initial boundary value problems for the Boltzmann equation is that there should be a singularity at the boundary \cite{GKTT-Invention}, and actually this singularity may propagate to the interior of the domains \cite{Kim-2011-CMP}. In this sense, unlike the Cauchy problem, it is quite hard to establish the global existence in Sobolev spaces for the Boltzmann equation with physical boundary. Thus, although the well-posedness for the Cauchy problem of the Landau equation \cite{Guo-L, SG-08-ARMA} and non-cutoff Boltzmann equation \cite{AMUXY-2012-JFA,GS} have already been established several years ago, the boundary value problem for either the Landau equation \cite{Guo-L} or the non-cutoff Boltzmann equation is left as a big open problem. Our results in this paper seem to be the first ones concerning this issue.

 \subsection{Organization of the paper} The rest of this paper is organized as follows. In Section \ref{sec2}, we give the precise statements of the main results of the paper regarding the global existence, large-time behavior and propagation of spatial regularity of solutions to the problems {\bf (PT)} and {\bf (PC)}. In Section \ref{sec3}, we present our strategy of  proof by showing the uniform a priori estimates without any velocity weight for both the non-cutoff Boltzmann equation and the Landau equation in the torus. Motivated by our  strategy of proof, Section \ref{sec4} is concerned with the trilinear estimates with velocity weights in the corresponding function spaces. Section \ref{sec5} is then devoted to establishing the macroscopic estimates both in the torus case and in the finite channel case. In Sections \ref{sec6} and \ref{sec7}, we will give the proof of the main results in the case of a torus and a finite channel, respectively. In Section \ref{sec8}, we explain the local-in-time existence of mild solutions for completeness. Then at the end of the paper, Appendix \ref{sec.app} includes some basic lemmas that are used in the previous sections.

\section{Main results}\label{sec2}

\subsection{Notations}

To state the main results in this section, we will now introduce  some more notation. Recall that to characterize the energy functional for the problem {\bf (PT)} or {\bf (PC)}, we have introduced  the function space
\begin{equation*}
X_T=L^1_k L^\infty_T L^2_v\text{ or }L^1_{\bar{k}} L^\infty_T L^2_{x_1,v},
\end{equation*}
respectively, as well as the velocity weighted space $X_T^w$ as in \eqref{def.xtw}, where the velocity weight $w=w_{q,\vth}(v)$ is defined in \eqref{def.w} under the assumption {\bf (H)} as in \eqref{q}. In what follows, we further define the corresponding energy dissipation rate functionals.  In this section, and in the rest of the paper we will use $f$, $g$ and $h$ as generic smooth real valued functions in our estimates, when $f$ is not being used as the solution to an equation such as \eqref{LLeq}.  Then since we are taking the Fourier transform we will also use the standard complex conjugate as $\bar{f}$.
Now, for the Landau equation, we recall the Landau kernel in \eqref{def.Lker}.  Then we define
$$
\sigma^{jm}=\sigma^{jm}(v)=\int_{{\R}^{3}}\psi^{jm}(v-u)\mu(u)du.
$$
It is convenient to define the following velocity weighted $D$-norm: 
$$
\left|w_{q,\vth}f\right|_{D}^{2}=\sum\limits_{j,m=1}^3
\int_{{\R}^{3}}w^{2}_{q,\vth}\left\{\sigma^{jm}\partial_{v_j}f\partial_{v_m}\bar{f}
+\frac{1}{4}\sigma^{jm}v_{j}v_{m}f\bar{f}\right\}dv.
$$
In the case of the finite channel, we also define
\begin{equation}
\label{def.fc.disDL}
\left\|w_{q,\vth}f\right\|_{D}^{2}=\sum\limits_{j,m=1}^3
\int_{I} \int_{\R^3}w^{2}_{q,\vth}\left\{\sigma^{jm}\partial_{v_j}f\partial_{v_m}\bar{f}
+\frac{1}{4}\sigma^{jm}v_{j}v_{m}f\bar{f}\right\}\,dvdx_1,
\end{equation}
by including an extra integration in $x_1\in I=(-1,1)$.
For the case of the non-cutoff Boltzmann equation, we define accordingly
\begin{multline}
|w_{q,\vth}f|_{D}^2=
\label{def.fc.nospace.disDB}
\\
\int_{\R_v^3}\int_{\R_u^3}\int_{\S^2}B(v-u,\sigma)w^{2}_{q,\vth}(v)\mu(u)(f(v')-f(v))\overline{(f(v')-f(v))} \,d\si du dv
\\
+
\int_{\R_v^3}\int_{\R_u^3}\int_{\S^2}B(v-u,\sigma)w^{2}_{q,\vth}(v)f(u)\overline{f(u)}\left(\mu^{\frac{1}{2}}(v')-\mu^{\frac{1}{2}}(v)\right)^2 \,d\si du dv,
\end{multline}
and in the presence of the spatial variable in the finite channel we have
\begin{multline}
\|w_{q,\vth}g\|_{D}^2=
\\
\int_I\int_{\R_v^3}\int_{\R_u^3}\int_{\S^2}B(v-u,\sigma)w^{2}_{q,\vth}(v)\mu(u)(f(v')-f(v))\overline{(f(v')-f(v))}  \,d\si du dvdx_1
\\
+
\int_I\int_{\R_v^3}\int_{\R_u^3}\int_{\S^2}B(v-u,\sigma)w^{2}_{q,\vth}(v)f(u)\overline{f(u)}\left(\mu^{\frac{1}{2}}(v')-\mu^{\frac{1}{2}}(v)\right)^2 \,d\si du dv dx_1.\label{def.fc.disDB}
\end{multline}
Then, corresponding to the energy functional $X_T^w$, we define the weighted dissipation rate functionals:
\begin{equation*}
\left\|w_{q,\vth}f\right\|_{L^1_{{k}}L^2_TL^2_{v,D}}
=\int_{\Z_{k}^3}\left(\int_0^T\left|w_{q,\vth}\CF_xf(t,{k})\right|_D^2dt\right)^{1/2}d\Si({k}),
\end{equation*}
and
\begin{equation*}
\left\|w_{q,\vth}f\right\|_{L^1_{\bar{k}}L^2_TL^2_{x_1,v,D}}
=\int_{\Z_{\bar{k}}^2}\left(\int_0^T\left\|w_{q,\vth}\CF_{\bar{x}}f(t,\bar{k})\right\|_D^2dt\right)^{1/2}d\Si(\bar{k}),
\end{equation*}
for the torus and finite channel domains, respectively. In the case of a finite channel, we need to include an extra first-order derivative in $x$, and thus we define the total energy functional and energy dissipation rate functional respectively as
\begin{equation}
\mathcal {E}_{T,w}(f)=
\sum\limits_{|\al|\leq1}\|w_{q,\vth}\pa^{\al}f\|_{L^1_{\bar{k}}L^\infty_TL^2_{x_1,v}},
\label{def.et}
\end{equation}
and
\begin{equation}
\mathcal{D}_{T,w}(f)
=\sum\limits_{|\al|\leq1}\|\pa^{\al}[a,b,c]\|_{L^1_{\bar{k}}L^2_TL^2_{x_1}}+
\sum\limits_{|\al|\leq1}\|w_{q,\vth}\{\FI-\FP\}\pa^{\al}f\|_{L^1_{\bar{k}}L^2_TL^2_{x_1,v,D}},
\label{def.dt}
\end{equation}
where $\pa^{\al}=\pa^{\al}_{x}=\pa_{x_1}^{\al_1}\pa_{x_2}^{\al_2}\pa_{x_3}^{\al_3}$ with $\al=(\al_1,\al_2,\al_3)$ a standard multi-index.  Further the macroscopic part $(a,b,c)$ is defined by \eqref{abcdef} below, and the norm $\|\cdot\|_{L^1_{\bar{k}}L^2_TL^2_{x_1}}$ in only in the $t$ and $x$ variables for a function $g=g(t,x)$  is understood in the same integration order as the norm $\|\cdot\|_{L^1_{\bar{k}}L^2_TL^2_{x_1,v,D}}$:
\begin{equation*}
\|g\|_{L^1_{\bar{k}}L^2_TL^2_{x_1}}
=
\int_{\Z_{\bar{k}}^2}\left(\int_0^T \int_{-1}^{1} |\CF_{\bar{x}}g(t,\bar{k})|^2 dx_1 dt\right)^{1/2}d\Si(\bar{k}).
\end{equation*}
When $w\equiv 1$, we will use $\CE_T(f)$ and $\CD_T(f)$ for brevity to denote the norms \eqref{def.et} and \eqref{def.dt} without the weight \eqref{def.w}.

For any given $t\geq 0$, we define the following norms in $x$ and $v$:
\begin{equation*}
\|w_{q,\vth}f(t)\|_{L^1_{{k}}L^2_{v}}=\int_{\Z_{k}^3}\|w_{q,\vth}\CF_{{x}}f(t,{k})\|_{L^2_v}\,d\Si({k}),
\end{equation*}
and
\begin{equation*}
\|w_{q,\vth}f(t)\|_{L^1_{\bar{k}}L^2_{x_1,v}}=\int_{\Z_{\bar{k}}^2}\|w_{q,\vth}\CF_{\bar{x}}f(t,\bar{k})\|_{L^2_{x_1,v}}\,d\Si(\bar{k}).
\end{equation*}
The corresponding high-order norms are defined by
\begin{equation*}
\|w_{q,\vth}f(t)\|_{L^1_{{k,m}}L^2_{v}}=\int_{\Z_{k}^3}\langle k\rangle^m\|w_{q,\vth}\CF_{{x}}f(t,{k})\|_{L^2_v}\,d\Si({k}),
\end{equation*}
and
\begin{equation*}
\|w_{q,\vth}f(t)\|_{L^1_{\bar{k},m}L^2_{x_1,v}}=\int_{\Z_{\bar{k}}^2}\langle \bar{k}\rangle^m\|w_{q,\vth}\CF_{\bar{x}}f(t,\bar{k})\|_{L^2_{x_1,v}}\,d\Si(\bar{k}),
\end{equation*}
where $m$ is an integer.

For the inflow boundary value problem in the finite channel case, we also define the following norms to capture the boundary effect of the given functions $g_{\pm}$:
\begin{multline}
E_{\bar{k}}(\widehat{g_{\pm}})
=
\sum\limits_{\pm}\int_0^T\int_{\pm v_1<0}|v_1|^{-1}|\widehat{\pa_t g_\pm}|^2dvdt
+\sum\limits_{\pm}\int_0^T\int_{\pm v_1<0}|v_1|^{-1}|\bar{k}\cdot\bar{v}|^2|\widehat{g_\pm}|^2dvdt
\\
+\sum\limits_{\pm}\int_0^T\int_{\pm v_1<0}|v_1|^{-1}|L\widehat{g_\pm}|^2dvdt
+\sum\limits_{\pm}\int_0^T\int_{\pm v_1<0}|v_1|^{-1}|\Ga(\widehat{g_\pm},\widehat{g_\pm})|^2dvdt
\\
+\sum\limits_{\pm}\int_0^T\int_{\pm v_1<0}|v_1|(1+|\bar{k}|^2)|\widehat{g_\pm}|^2dvdt,
\label{Enormk}
\end{multline}
and
\begin{equation}\label{def.Enorm}
\begin{split}
E(\widehat{g_{\pm}})=\int_{\Z^2_{\bar{k}}}\sqrt{E_{\bar{k}}(\widehat{g_{\pm}})}d\Si(\bar{k}).
\end{split}
\end{equation}

\noindent{\it Other notations.}
Throughout this paper,  $C$ denotes a generic positive (generally large) uniform constant and $\la$ denotes
a generic positive (generally small) uniform constant, where both $C$ and $\la$
may take different values in different places. $A\lesssim B$ means that  there is a generic constant $C>0$
such that $A\leqslant CB$. $A\sim B$ means $A\lesssim B$ and $B\lesssim A$.  We also typically use $\eta>0$ to denote a constant that can be made arbitrarily small.


\subsection{Case of the Torus}

To state the main results, we are first concerned with the problem {\bf (PT)} for the Landau equation or the non-cutoff Boltzmann equation in torus.

\begin{theorem}[Existence and large-time behavior]\label{Torus Existence}
Let $\Omega=\T^3$. Assume that $f_0(x,v)$ satisfies \eqref{pt.id.cl.1}, \eqref{pt.id.cl.2} and \eqref{pt.id.cl.3}. Let $w_{q,\vartheta}$ be chosen under the assumption {\bf (H)} in \eqref{q}.  There is $\epsilon_0>0$  such that if $F_0(x,v)=\mu+\mu^{\frac{1}{2}}f_0(x,v)\ge 0$ and
\begin{align*}
\Vert w_{q,\vartheta} f_0\Vert_{L^1_k L^2_v} \le \epsilon_0,
\end{align*}
then there exists a unique global mild solution $f(t,x,v)$, $t>0$, $x\in \mathbb{T}^3$, $v\in \mathbb{R}^3$ to the problem {\bf (PT)} \eqref{LLeq} and \eqref{idf} for the Landau equation or the non-cutoff Boltzmann equation, satisfying that  $F(t,x,v)=\mu+\mu^{\frac{1}{2}}f(t,x,v)\ge 0$ and
\begin{align}
\|w_{q,\vartheta}f\|_{L^1_kL^\infty_TL^2_v}+\|w_{q,\vartheta}f\|_{L^1_kL^2_TL^2_{v,D}}\lesssim
\Vert w_{q,\vartheta} f_0\Vert_{L^1_k L^2_v},\label{Torus Existence.uet}
\end{align}
for any $T>0$.
Moreover, let $\kappa\in (0,1]$ be defined in \eqref{def.kap1} or \eqref{def.kap2} for the Landau case or the non-cutoff Boltzmann case, respectively, then there is $\lambda>0$ such that the solution also enjoys the time decay estimate
\begin{align}\label{Torus.Time.Decay.Est}
\Vert f(t)\Vert_{L^1_k L^2_v} \lesssim e^{-\lambda t^\kappa} \Vert w_{q,\vartheta} f_0\Vert_{L^1_k L^2_v},
\end{align}
for any $t\geq 0$.
\end{theorem}

\begin{theorem}[Propagation of spatial regularity]\label{Torus Propagation}
Let all the conditions in Theorem \ref{Torus Existence} be satisfied, then for any integer $m\geq 0$, there is an $\eps_0>0$  such that if
\begin{align}\label{thm.tp.as}
\Vert w_{q,\vartheta} f_0\Vert_{L^1_{k,m} L^2_v} \le \epsilon_0,
\end{align}
then the solution $f(t,x,v)$ to \eqref{LLeq}-\eqref{idf} established in Theorem \ref{Torus Existence}  satisfies
\begin{align}
\int_{\Z^3} \langle k\rangle^m&\sup_{0\le t\le T} \Vert w_{q,\vartheta}\hat{f}(t,k)\Vert_{L^2_v} d\Sigma(k)\notag\\
&+\int_{\Z^3}\langle k\rangle^m \Big( \int^T_0 |  w_{q,\vartheta}\hat{f}(t,k)|_D^2 dt\Big)^{1/2}d\Sigma(k)
\lesssim
\Vert w_{q,\vartheta} f_0\Vert_{L^1_{k,m} L^2_v},\label{Torus.Propagation.Est}
\end{align}
for any $T>0$.
\end{theorem}

\subsection{Case of the Finite channel}

Next, we are concerned with the problem {\bf (PC)} for the Landau equation or the non-cutoff Boltzmann equation in the finite channel.

\begin{theorem}[Inflow boundary condition]\label{mthif}
Let $\Omega=I\times \T^2$. Let $w_{q,\vartheta}$ be chosen under the assumption {\bf (H)} in \eqref{q}. There are $\eps_0>0$ and $C>0$ such that if $F_0(x_1,\bar{x},v)=\mu+\mu^{\frac{1}{2}}f_0(x_1,\bar{x},v)\geq0$, $F(t,\pm1,\bar{x},v)=\mu+\mu^{\frac{1}{2}}g_\pm(t,\bar{x},v)\geq0$ for $v_1>0$ at $x_1=-1$ and $v_1<0$ at $x_1=1$, and
\begin{equation}\label{idcd.ifb}
\sum\limits_{|\al|\leq1}\|w_{q,\vth}\pa^\al f_{0}\|_{L_{\bar{k}}^1L^2_{x_1,v}}+E(w_{q,\vth}\widehat{g_\pm})\leq \eps_0,
\end{equation}
then there exists a unique mild solution $f(t,x_1,\bar{x},v)$ to the inflow boundary problem {\bf (PC)} \eqref{LLeq},
\eqref{idf} and \eqref{ifb} for the Landau equation or the non-cutoff Boltzmann equation,  satisfying that  $F(t,x_1,\bar{x},v)=\mu+\mu^{\frac{1}{2}}f(t,x_1,\bar{x},v)\geq0$, and
\begin{equation}\label{ifenges}
\CE_{T,w}(f)+\CD_{T,w}(f)\leq C \left\{\sum\limits_{|\al|\leq1}\|w_{q,\vth}\pa^{\al}f_0\|_{L_{\bar{k}}^1L^2_{x_1,v}}
+E(w_{q,\vth}\widehat{g_{\pm}})\right\},
\end{equation}
for any $T>0$, where $\CE_{T,w}(f)$, $\CD_{T,w}(f)$, and $E(w_{q,\vth}\widehat{g_{\pm}})$ are defined in \eqref{def.et}, \eqref{def.dt}, and \eqref{def.Enorm}, respectively.
Moreover, let $\kappa\in (0,1]$ be defined in \eqref{def.kap1} or \eqref{def.kap2} for the Landau case or the non-cutoff Boltzmann case, respectively, then there is $\la>0$ such that if
$$
E(w_{q,\vartheta}\widehat{g_\pm})
+\sup_{s> 0} E(e^{\la s^\kappa}\widehat{g_{\pm}})\leq\eps_0,
$$
for $\eps_0>0$ further small enough, then it holds that
\begin{multline}\label{ifdec}
\sum\limits_{|\al|\leq1}\|\pa^{\al}f(t)\|_{L^1_{\bar{k}}L^2_{x_1,v}}
\lesssim
e^{-\la t^{\kappa}}\sum\limits_{|\al|\leq1}\|w_{q,\vth}\pa^{\al}f_0\|_{L_{\bar{k}}^1L^2_{x_1,v}}
\\
+
e^{-\la t^{\kappa}}\left\{E(w_{q,\vartheta}\widehat{g_\pm})
+\sup_{s>0}E(e^{\la s^\kappa}\widehat{g_{\pm}})\right\},
\end{multline}
for any $t\geq0.$
\end{theorem}

\begin{theorem}[Specular reflection boundary condition]\label{mthsr}
Let $\Omega=I\times \T^2$. Let $w_{q,\vartheta}$ be chosen under the assumption {\bf (H)} in \eqref{q}. Assume the symmetry condition \eqref{ass.sym} and let $f_0(x,v)$ satisfy \eqref{PC.cl1}, \eqref{PC.cl2}, and \eqref{PC.cl3}. There are $\eps_0>0$ and $C>0$ such that if $F_0(x_1,\bar{x},v)=\mu+\mu^{\frac{1}{2}}f_0(x_1,\bar{x},v)\geq0$
and
\begin{equation}\label{idcd.srb}
\sum\limits_{|\al|\leq1}\|w_{q,\vth}\pa^\al f_{0}\|_{L_{\bar{k}}^1L^2_{x_1,v}}\leq\eps_0,
\end{equation}
then there exists a unique mild solution $f(t,x_1,\bar{x},v)$ to the specular reflection boundary problem {\bf (PC)} \eqref{LLeq},
\eqref{idf} and \eqref{srb}  for the Landau equation or the non-cutoff Boltzmann equation,  satisfying that $F(t,x_1,\bar{x},v)=\mu+\mu^{\frac{1}{2}}f(t,x_1,\bar{x},v)\geq0$ with $f(t,-x_1,\bar{x},-v_1,\bar{v})=f(t,x_1,\bar{x},v_1,\bar{v})$
and
\begin{equation}\label{srenges}
\CE_{T,w}(f)+\CD_{T,w}(f)\leq C \sum\limits_{|\al|\leq1}\|w_{q,\vth}\pa^{\al}f_0\|_{L_{\bar{k}}^1L^2_{x_1,v}},
\end{equation}
for any $T>0$, where  $\CE_{T,w}(f)$ and $\CD_{T,w}(f)$ are defined in \eqref{def.et} and \eqref{def.dt}, respectively.
Moreover,
 let $\kappa\in (0,1]$ be defined in \eqref{def.kap1} or \eqref{def.kap2} for the Landau case or the non-cutoff Boltzmann case, respectively, then
there is $\la>0$ such that
\begin{equation}\label{srdec}
\sum\limits_{|\al|\leq1}\|\pa^{\al}f(t)\|_{L^1_{\bar{k}}L^2_{x_1,v}}
\lesssim e^{-\la t^{\kappa}}\sum\limits_{|\al|\leq1}\|w_{q,\vth}\pa^{\al}f_0\|_{L_{\bar{k}}^1L^2_{x_1,v}},
\end{equation}
for any $t\geq0.$
\end{theorem}

\begin{theorem}[Propagation of spatial regularity in $\bar{x}$]\label{regp.th}
Let all of the conditions in Theorem  \ref{mthif} and Theorem \ref{mthsr} be satisfied, then for any integer $m\geq 0$, there are $\eps_0>0$ and $C>0$ such that if
\begin{equation}
\label{thm.cp.as1}
\sum\limits_{|\al|\leq1}\|w_{q,\vth}\pa^\al f_{0}\|_{L_{\bar{k},m}^1L^2_{x_1,v}}+E(w_{q,\vth}\langle \bar{k}\rangle^m\widehat{g_\pm})\leq \eps_0,
\end{equation}
and
\begin{equation}
\label{thm.cp.as2}
\sum\limits_{|\al|\leq1}\|w_{q,\vth}\pa^\al f_{0}\|_{L_{\bar{k},m}^1L^2_{x_1,v}}\leq\eps_0,
\end{equation}
hold in the place of \eqref{idcd.ifb} and \eqref{idcd.srb}, respectively, then we obtain that
\begin{multline}\label{reg.ib}
\sum_{|\al|\leq 1}\int_{\Z^2}\langle\bar{k}\rangle^{m}\sup\limits_{0\le t \le T}\|w_{q,\vth}\widehat{\pa^{\al}f}(t,\bar{k})\|d\Si(\bar{k})
\\
+\sum_{|\al|\leq 1}\int_{\Z^2}\langle\bar{k}\rangle^{m}\left(\int_0^T\|w_{q,\vth}\widehat{\pa^{\al}f}\|^2_{D}dt\right)^{1/2}d\Si(\bar{k})
\\
\lesssim
\left\{
\sum_{|\al|\leq 1}\|w_{q,\vth}\pa^\al f_{0}\|_{L_{\bar{k},m}^1L^2_{x_1,v}}+E(w_{q,\vth}\langle\bar{k}\rangle^{m}\widehat{g_{\pm}})\right\},
\end{multline}
for the inflow boundary condition,
and
\begin{multline}\label{reg.sr}
\sum_{|\al|\leq 1}\int_{\Z^2}\langle\bar{k}\rangle^{m}\sup\limits_{0\le t \le T}\|w_{q,\vth}\widehat{\pa^{\al}f}(t,\bar{k})\|d\Si(\bar{k})
\\
+\sum_{|\al|\leq 1}\int_{\Z^2}\langle\bar{k}\rangle^{m}\left(\int_0^T\|w_{q,\vth}\widehat{\pa^{\al}f}\|^2_{D}dt\right)^{1/2}d\Si(\bar{k})
\\
\lesssim  \sum_{|\al|\leq 1}\|w_{q,\vth}\pa^\al f_{0}\|_{L_{\bar{k},m}^1L^2_{x_1,v}},
\end{multline}
for the specular reflection boundary condition, respectively.
\end{theorem}

\subsection{Comments}

Here a few comments are in order on the results of this paper.

\begin{itemize}
\item[(a)]  As mentioned before, the global existence of classical solutions close to the global Maxwellians was established in \cite{Guo-L} for the Landau equation and in \cite{GS} for the Boltzmann equation without angular cutoff both in the torus. To the best of our knowledge, the current results seem to be the first ones to revisit these global existence theories in a new function space with mild regularity in both space and velocity variables.  The proof can be carried out in a unified way for the Landau and Boltzmann cases. The function space
\begin{equation*}
L^1_kL^\infty_TL^2_v\cap L^1_kL^2_TL^2_{v,D}
\end{equation*}
appears to be very useful for treating the global dynamics of Landau and non-cutoff Boltzmann solutions in perturbation framework. \\

\item[(b)] In the case of a finite channel, to the best of our knowledge these results may be the first ones to provide an elementary understanding of the existence theories for the Landau or non-cutoff Boltzmann equations in the situation where the spatial domain has physical boundaries.

As mentioned before, the singularities can generally form for solutions to the boundary value problem.  When the boundary condition is specular reflection, we have made an essential use of the symmetric condition \eqref{ass.sym} on the initial data which then also remains true for the solution.   The singularity is killed by such a symmetric property of solutions.  It turns out that without this symmetric assumption the corresponding results are still unknown.  For the inflow boundary value problem, the singularity is transformed to the boundary data, see the boundary energy functional as in \eqref{Enormk} or \eqref{def.Enorm} which is induced by the estimates on the normal derivatives $\pa_{x_1}f$ in terms of the equations as in \eqref{paxifb} below.

Another issue is related to the case of the diffusive reflection boundary for which it seems too hard at the moment to treat estimates on the boundary terms in our settings. Of course with the developments in this work, we are still far away from dealing with the existence of solutions, even local in time, for the Landau or non-cutoff Boltzmann equation in general bounded domains. \\

\item[(c)] Regarding results on the propagation of spatial regularity, it should be pointed out that it may be unnecessary to require the smallness assumptions on the $m$-th order derivatives of the initial data as in \eqref{thm.tp.as}, \eqref{thm.cp.as1}, and \eqref{thm.cp.as2}.  Indeed, since the initial data is assumed to be sufficiently small in the function space without the Fourier mode $\langle k\rangle^m$, one could try to use an induction argument together with interpolation techniques and time-decay properties of solutions to obtain the similar results to those in Theorem \ref{Torus Propagation} and  Theorem \ref{regp.th} under the only assumption that the $m$-order derivatives of initial data have a finite norm in the space $L^1_kL^2_v$ with an appropriate velocity weight.  For brevity of the presentation we will not pursue such improved results in this paper.\\

\item[(d)] The question of the regularity of the obtained solutions in our situation of the global existence theory is an interesting issue to be further studied.  This is because either the Landau operator or the non-cutoff Boltzmann operator both enjoy the velocity diffusion property, and  the spatial regularity could also be gained through the hypoellipticity techniques of such collisional kinetic equations; see an aforementioned work \cite{CHLZ} and references therein.

\end{itemize}

\section{Strategy of the proof}\label{sec3}

In this section we shall present the strategy of the proof.  In particular we will explain how to obtain the uniform a priori estimates on the solutions in the function space $X_T=L^1_kL^\infty_TL^2_v$ for the problem {\bf (PT)} in the torus case $\Omega=\T^3$. For the problem {\bf (PC)} in the finite channel case $\Omega=(-1,1)\times \T^2$, we can carry out the same strategy to treat the periodic variable $\bar{x}=(x_2,x_3)\in \T^2$ with some extra consideation for the variable $x_1\in (-1,1)$; this will be clarified in Section \ref{sec7} later. In our presentation in this section we will not include any velocity weight in order to increase the claritiy of the current exposition. The local-in-time existence of solutions will be studied in Section \ref{sec8}.

\subsection{Non-cutoff Boltzmann case}
For convenience of the presentation we first start from the Boltzmann equation without angular cutoff. Let $f=f(t,x,v)$, $0\leq t\leq T$, $x\in \T^3$, $v\in \R^3$, be a smooth solution to
\begin{equation}
\label{str.be.eq}
\partial_t f +v\cdot\nabla_x f+Lf= \Gamma(f,f),
\end{equation}
with prescribed initial data $f(0,x,v)=f_0(x,v)$. Taking the Fourier transform in $x\in \T^3$, we obtain
\begin{align}\label{torus fourier transformed}
\partial_t \hat{f}(t,k,v) + iv\cdot k \hat{f}(t,k,v)+L\hat{f}(t,k,v)=\hat{\Gamma}(\hat{f},\hat{f})(t,k,v).
\end{align}
Here, for brevity, we have denoted the following operator
\begin{multline}\label{def.GaF}
\hat{\Gamma}(\hat{f},\hat{g})(k,v)
=
\\
\int_{\mathbb{R}^3}\int_{\mathbb{S}^{2}} B(v-u,\si) \mu^{1/2}(u)\left([\hat{f}(u')*\hat{g}(v')](k)-[\hat{f}(u)*\hat{g}(v)](k)\right)d\sigma du,
\end{multline}
where the convolutions are taken with respect to $k\in\Z^3$:
\begin{eqnarray*}
{[\hat{f}(u')*\hat{g}(v')]}(k) & := &\int_{\Z^3_l} \hat{f}(k-l,u')\hat{g}(l,v')\,d\Si(l),\\
{[\hat{f}(u)*\hat{g}(v)]}(k)  & := &\int_{\Z^3_l}  \hat{f}(k-l,u)\hat{g}(l,v)\,d\Si(l).
\end{eqnarray*}
Then for this system we have the uniform estimate in the following proposition.

\begin{proposition}\label{Torus Boltzmann Micro}
There is a universal constant $C>0$ such that
\begin{multline}
\int_{\mathbb{Z}^3} \sup_{0\le t\le T} \Vert \hat{f}(t,k,\cdot)\Vert_{L^2_v}\,d\Sigma(k) + \int_{\Z^3} \left(\int^T_0 | \{\mathbf{I}-\mathbf{P}\}\hat{f}(t,k,\cdot) |_D^2\,dt\right)^{1/2}d\Sigma(k)\\
\lesssim \Vert {f}_0\Vert_{L^1_kL^2_v}+ \left(\eta +\frac{\Vert{f}\Vert_{L^1_kL^\infty_TL^2_v}}{4\eta}\right) \int_{\Z^3} \left(\int^T_0 | \hat{f}(t,k,\cdot) |_D^2 \,dt\right)^{1/2}d\Sigma(k),
\label{est.fRHS}
\end{multline}
for any $T>0$, where the constant $\eta>0$ can be arbitrarily small.
\end{proposition}

\begin{proof}
Taking the product of  \eqref{torus fourier transformed} with the complex conjugate of $\hat{f}(t,k,v)$ and further taking the real part of the resulting equation, we have
\begin{align*}
\frac{1}{2}\frac{d}{dt}\vert \hat{f}(t,k,v)\vert^2+\mathscr{R}(L\hat{f}, \hat{f})=\mathscr{R}(\hat{\Gamma}(\hat{f},\hat{f}),\hat{f}),
\end{align*}
where $(\cdot,\cdot)$ denotes the complex inner product over the complex field, i.e., $(f,g)=f\cdot \bar{g}$, and $\mathscr{R}$ denotes the real part of a complex number.
Integrating the above identity with respect to $v$ and then $t$ we have
\begin{align}\label{pro.tpm.p1}
\frac{1}{2}\Vert \hat{f}(t,k,\cdot)\Vert_{L^2_v}^2+\int^t_0 \mathscr{R}(L\hat{f},\hat{f})_{L^2_v}d\tau=\frac{1}{2}\Vert \hat{f}_0(k,\cdot)\Vert_{L^2_v}^2+\int^t_0 \mathscr{R}(\hat{\Gamma}(\hat{f},\hat{f}),\hat{f})_{L^2_v}d\tau,
\end{align}
where correspondingly  $(\cdot,\cdot)_{L^2_v}$ denotes the complex inner product over $L^2_v$, i.e.,
$$
(f,g)_{L^2_v}=\int_{\R^3_v} f(v)\overline{g(v)}dv.
$$
Recall that by the coercivity estimate of $L$ (cf.~Lemma \ref{esBL}),
there is $\de_0>0$ such that
\begin{equation*}
\de_0| \{\mathbf{I}-\mathbf{P}\}g |_D^2\le \mathscr{R} (Lg,g)_{L^2_v}.
\end{equation*}
Thus it follows from \eqref{pro.tpm.p1} that
\begin{align*}
\frac{1}{2}\Vert \hat{f}(t,k,\cdot)\Vert_{L^2_v}^2 + \de_0 \int^t_0 | \{\mathbf{I}-\mathbf{P}\}\hat{f} |_D^2d\tau \le \frac{1}{2}\Vert \hat{f}_0(k,\cdot)\Vert_{L^2_v}^2+\int^t_0 \mathscr{R}(\hat{\Gamma}(\hat{f},\hat{f}),\hat{f})_{L^2_v}d\tau.
\end{align*}
Taking the square root on both sides and using the elementary inequalities
\begin{equation*}
\frac{1}{\sqrt{2}} (A+B)\leq \sqrt{A^2+B^2}\leq A+B,
\end{equation*}
we further have
\begin{align*}
\frac{1}{\sqrt{2}}\Vert \hat{f}(t,k,\cdot)\Vert_{L^2_v} &+ \sqrt{\de_0} \left(\int^t_0 | \{\mathbf{I}-\mathbf{P}\}\hat{f}(\tau,k,\cdot) |_D^2d\tau\right)^{1/2}\\
& \le \Vert \hat{f}_0(k,\cdot)\Vert_{L^2_v}+\sqrt{2}\left(\int^t_0 \left|\mathscr{R}(\hat{\Gamma}(\hat{f},\hat{f}),\hat{f})_{L^2_v}\right| d\tau\right)^{1/2}.
\end{align*}
So, we have derived, for any $0\leq t\leq T$ and $k\in \Z^3$, that we have
\begin{align}
\Vert \hat{f}(t,k,\cdot)\Vert_{L^2_v} &+  \left(\int^t_0 | \{\mathbf{I}-\mathbf{P}\}\hat{f}(\tau,k,\cdot) |_D^2d\tau\right)^{1/2}\notag\\
&\le C_0\left\{ \Vert \hat{f}_0(k,\cdot)\Vert_{L^2_v} +\left(\int^t_0\left| \mathscr{R}(\hat{\Gamma}(\hat{f},\hat{f}),\hat{f})_{L^2_v}\right|d\tau\right)^{1/2}\right\},\label{v-int}
\end{align}
with
\begin{equation*}
C_0=\frac{\sqrt{2}}{\min\{1/\sqrt{2},\sqrt{\de_0}\}}>0.
\end{equation*}
Moreover, taking $\sup_{0\le t\le T}$ on both sides of \eqref{v-int} and then integrating the resulting inequality with respect to $d\Sigma(k)$ over $\Z^3$, we have
\begin{align}
&\int_{\Z^3} \sup_{0\le t\le T} \Vert \hat{f}(t,k,\cdot)\Vert_{L^2_v}d\Sigma(k) + \int_{\Z^3} \left(\int^T_0 | \{\mathbf{I}-\mathbf{P}\}\hat{f}(t,k,\cdot) |_D^2\,dt\right)^{1/2}d\Sigma(k)\notag\\
&\le C_0\left\{ \Vert \hat{f}_0\Vert_{L^1_kL^2_v}+\int_{\Z^3} \left(\int^T_0 \left|\left(\hat{\Gamma}(\hat{f},\hat{f}),\hat{f}\right)_{L^2_v}\right|dt\right)^{1/2}d\Sigma(k)\right\}.\label{pro.tpm.p2}
\end{align}
In what follows we will estimate the last term on the right-hand side of \eqref{pro.tpm.p2}. Indeed, it can be bounded as
\begin{multline}
\int_{\Z^3} \left(\int^T_0 \left|\left(\hat{\Gamma}(\hat{f},\hat{f}),\hat{f}\right)_{L^2_v}\right|dt\right)^{1/2}d\Sigma(k)\\
\leq \int_{\Z^3}\left(\int^T_0 \int_{\Z^3}\Vert \hat{f}(k-l)\Vert_{L^2_v} | \hat{f}(l) |_D | \hat{f}(k) |_Dd\Sigma(l)dt\right)^{1/2}d\Sigma(k),
\label{pro.tpm.p3}
\end{multline}
in terms of the  trilinear estimate in the following lemma.

\begin{lemma}\label{lem: Boltzmann nonlinear}
It holds that
\begin{align}\label{ineq: Boltzmann nonlinear}
\left| \left(\hat{\Gamma}(\hat{f},\hat{g})(k),\hat{h}(k)\right)_{L^2_v}\right|\le C\int_{\Z^3}\Vert \hat{f}(k-l)\Vert_{L^2_v}| \hat{g}(l)|_D  | \hat{h}(k)|_D\,d\Sigma(l).
\end{align}
\end{lemma}

\begin{proof}
By the definition \eqref{def.GaF} of $\hat{\Ga} (\cdot,\cdot)$ as well as Fubini's theorem, we obtain
\begin{align*}
&
(\hat{\Gamma}(\hat{f},\hat{g})(k),\hat{h}(k))_{L^2_v}
\notag\\
&=
\int_{\mathbb{R}^3_v} \int_{\mathbb{R}^3_{u}} \int_{\mathbb{S}^{2}} B(v-u,\si) \mu^{1/2}(u)
\\
&\qquad\qquad\qquad\times\left\{[\hat{f}(u')*\hat{g}(v')](k)-[\hat{f}(u)*\hat{g}(v)](k)\right\}\bar{\hat{h}}(v,k) d\sigma dudv
\notag\\
&=
\int_{\mathbb{R}^3_v} \int_{\mathbb{R}^3_{u}} \int_{\mathbb{S}^{2}} B(v-u,\si) \mu^{1/2}(u)
\\
&\qquad\qquad\qquad\times\int_{\Z^3} \left\{\hat{f}(k-l,u')\hat{g}(l,v')-\hat{f}(k-l,u)\hat{g}(l,v)\right\} \bar{\hat{h}}(v,k) d\Sigma(l) d\sigma du dv
\\
&=
\int_{\Z^3} \int_{\mathbb{R}^3_v} \Gamma(\hat{f}(k-l),\hat{g}(l))\bar{\hat{h}}(k)dv  d\Sigma(l).
\end{align*}
Therefore, it follows that
\begin{equation}
\label{lem.ptri.p1}
\left\vert (\hat{\Gamma}(\hat{f},\hat{g})(k),\hat{h}(k))_{L^2_v}\right\vert\le \int_{\Z^3} \left\vert (\Gamma(\hat{f}(k-l),\hat{g}(l)),\hat{h}(k))_{L^2_v}\right\vert d\Sigma(l).
\end{equation}
Recall either from \cite[Theorem 1.2]{AMUXY13KRM} or from \cite[Theorem 2.1, page 782]{GS}
that one has
\begin{equation}\label{ad.vwte}
\vert (\Gamma(f,g),h)_{L^2_v}\vert \lesssim \Vert f\Vert_{L^2_v} | g |_D | h |_D.
\end{equation}
With the above inequality the desired estimate \eqref{ineq: Boltzmann nonlinear} follows from \eqref{lem.ptri.p1}.  This completes the proof of Lemma \ref{lem: Boltzmann nonlinear}.
\end{proof}

\begin{remark}\label{norm.remark}
We remark here, in regards to the previous proof and what follows, that the norm \eqref{def.fc.nospace.disDB} is equivalent to the norm $N^{s,\gamma}$ from \cite[Equation (1.8), page 774]{GS}; this can be shown directly for example by using the estimates in \cite{GS}.  See for example  \cite[Equation (2.13), page 784]{GS}.
\end{remark}

We continue to estimate the upper bound in \eqref{pro.tpm.p3}. Applying Cauchy-Schwarz's inequality with respect to $\int_0^T(\cdot)dt$ and further using Young's inequality with an arbitrary small constant $\eta>0$, we have
\begin{align}
&\int_{\Z^3_k}\left(\int^T_0 \int_{\Z^3_l}\Vert \hat{f}(t,k-l)\Vert_{L^2_v} | \hat{f}(t,l) |_D | \hat{f}(t,k) |_Dd\Sigma(l)dt\right)^{1/2}d\Sigma(k)
\notag
\\
&\leq \int_{\Z^3_k} \left(\int^T_0 \left(\int_{\Z^3_l} \Vert \hat{f}(t,k-l)\Vert_{L^2_v} | \hat{f}(t,l) |_Dd\Sigma(l)\right)^2dt\right)^{1/4}
\notag
\\
&\qquad\qquad\qquad\times \left(\int^T_0 | \hat{f}(t,k) |_D^2 dt\right)^{1/4} d\Sigma(k)
\notag
\\
&\leq \eta\int_{\Z^3_k}\left(\int^T_0 | \hat{f}(t,k) |_D^2 dt\right)^{1/2}d\Sigma(k)
\notag
\\
&\qquad+\frac{1}{4\eta} \int_{\Z^3_k}\left(\int^T_0 \left(\int_{\Z^3_l} \Vert \hat{f}(t,k-l)\Vert_{L^2_v} | \hat{f}(t,l) |_Dd\Sigma(l)\right)^2dt\right)^{1/2} d\Sigma(k).\label{pro.tpm.p4}
\end{align}
To treat the second term on the right-hand side of \eqref{pro.tpm.p4}, we first use Minkowski's inequality to obtain 
\begin{equation}\label{minkowski.ineq}
\big\|\|\cdot\|_{L^1_l}\big\|_{L^2_t}\leq \big\|\|\cdot\|_{L^2_t}\big\|_{L^1_l}.
\end{equation}
Then we have
\begin{multline}\notag
\left(\int^T_0 \left(\int_{\Z^3_l} \Vert \hat{f}(t,k-l)\Vert_{L^2_v} | \hat{f}(t,l) |_Dd\Sigma(l)\right)^2dt\right)^{1/2}
\\
\leq
\int_{\Z^3_l} \left(\int^T_0 \Vert \hat{f}(t,k-l)\Vert_{L^2_v}^2 | \hat{f}(t,l) |_D^2 dt\right)^{1/2}d\Sigma(l),
\end{multline}
and hence it follows that
\begin{multline*}
\int_{\Z^3_k}\left(\int^T_0 \left(\int_{\Z^3_l} \Vert \hat{f}(t,k-l)\Vert_{L^2_v} | \hat{f}(t,l) |_Dd\Sigma(l)\right)^2dt\right)^{1/2} d\Sigma(k)\\
\leq \int_{\Z^3_k} \int_{\Z^3_l} \sup_{0\le t\le T} \Vert \hat{f}(t,k-l) \Vert_{L^2_v} \left(\int^T_0| \hat{f}(t,l) |_D^2 dt\right)^{1/2}d\Sigma(l) d\Sigma(k).
\end{multline*}
Further by Fubini's theorem and translation invariance, the upper bound in the above inequality can be computed as
\begin{align*}
& \int_{\Z^3_k} \int_{\Z^3_l} \sup_{0\le t\le T} \Vert \hat{f}(t,k-l) \Vert_{L^2_v} \left(\int^T_0| \hat{f}(t,l) |_D^2 dt\right)^{1/2}d\Sigma(l) d\Sigma(k)\\
&= \int_{\Z^3_l}  d\Sigma(l)\left(\int^T_0| \hat{f}(t,l) |_D^2 dt\right)^{1/2} \int_{\Z^3_k}d\Sigma(k) \sup_{0\le t\le T} \Vert \hat{f}(t,k-l) \Vert_{L^2_v} \\
&=\Vert\hat{f}\Vert_{L^1_kL^\infty_TL^2_v} \int_{\Z^3_l} \left(\int^T_0| \hat{f}(t,l) |_D^2 dt\right)^{1/2} d\Sigma(l).
\end{align*}
Then, applying those estimates above to the second term on the right-hand side of \eqref{pro.tpm.p4} and further using \eqref{pro.tpm.p2} and \eqref{pro.tpm.p3} leads to  the desired estimate \eqref{est.fRHS}. This completes the proof of Proposition \ref{Torus Boltzmann Micro}.
\end{proof}

We remark that we have the following identity
\begin{equation*}
\left(\hat{\Gamma}(\hat{f},\hat{f}),\hat{f}\right)_{L^2_v}=\left(\hat{\Gamma}(\hat{f},\hat{f}),\{\FI-\FP\}\hat{f}\right)_{L^2_v}.
\end{equation*}
Then using this identity one can modify the proof above slightly to obtain
\begin{multline*}
\int_{\mathbb{Z}^3} \sup_{0\le t\le T} \Vert \hat{f}(t,k,\cdot)\Vert_{L^2_v}\,d\Sigma(k) + \int_{\Z^3} \left(\int^T_0 | \{\mathbf{I}-\mathbf{P}\}\hat{f}(t,k,\cdot) |_D^2\,dt\right)^{1/2}d\Sigma(k)
\\
\lesssim
 \Vert {f}_0\Vert_{L^1_kL^2_v}+ \Vert{f}\Vert_{L^1_kL^\infty_TL^2_v}\int_{\Z^3} \left(\int^T_0 | \hat{f}(t,k,\cdot) |_D^2 \,dt\right)^{1/2}d\Sigma(k).
\end{multline*}
This indicates that as long as one can further appropriately estimate (such as in Theorem \ref{abcest.torus} below) the macroscopic dissipation:
\begin{equation*}
\int_{\Z^3} \left(\int^T_0 | \FP \hat{f}(t,k,\cdot) |_D^2\,dt\right)^{1/2}d\Sigma(k)\sim \int_{\Z^3} \left(\int^T_0 | (\widehat{a,b,c})(t,k) |^2\,dt\right)^{1/2}d\Sigma(k),
\end{equation*}
where $(a,b,c)$ is defined in \eqref{abcdef}, we can then obtain the uniform estimates under the smallness assumption on $\Vert{f}\Vert_{L^1_kL^\infty_TL^2_v}$ that can be closed provided that $f_0$ is suitably small in $L^1_kL^2_v$.

\subsection{Landau case}

When the non-cutoff Boltzmann operator is replaced by the Landau operator as in \eqref{bLop}, we are still able to carry out the same strategy. Indeed, the nonlinear Landau operator \eqref{nopdef} takes the following form, cf.~\cite[page 395]{Guo-L}:
\begin{align}
\Gamma (f,g) = &\sum_{i,j=1}^3 \Big[\partial_{v_i}[\{\psi^{ij}*(\mu^{\frac{1}{2}}f)\}\partial_{v_j}g]-\{\psi^{ij}*(v_i\mu^{\frac{1}{2}}f)\}\partial_{v_j}g\notag\\
&\qquad\quad-\partial_{v_i}[\{\psi^{ij}*(\mu^{\frac{1}{2}}\partial_{v_j}f)\}g]+\{\psi^{ij}*(v_i\mu^{\frac{1}{2}}\partial_{v_j}f)\}g\Big].
\notag 
\end{align}
We consider the linearized Landau equation of the same form as in \eqref{str.be.eq}. Taking the Fourier transform in $x\in \T^3$ gives the same form \eqref{torus fourier transformed} with the nonlinear part now taking the form:
%
\begin{align}\label{landau.NL.FT}
\hat{\Gamma}(\hat{f},\hat{g}) = \sum_{i,j=1}^3 &\Big[ \partial_{v_i}[\{\psi^{ij}*_v(\mu^{\frac{1}{2}}\hat{f})\}*_k\partial_{v_j}\hat{g}]
-\{\psi^{ij}*_v(v_i\mu^{\frac{1}{2}}\hat{f})\}*_k\partial_{v_j}\hat{g}
\\
&\quad -\partial_{v_i}[\{\psi^{ij}*_v(\mu^{\frac{1}{2}}\partial_{v_j}\hat{f})\}*_k\hat{g}]
+\{\psi^{ij}*_v(v_i\mu^{\frac{1}{2}}\partial_{v_j}\hat{f})\}*_k\hat{g}\Big],
\notag
\end{align}
where $*_v$ denotes the convolution in $v\in \R^3$ and $*_k$ denotes the convolution in $k\in \Z^3$.  We now apply the estimate \cite[Theorem 3, page 406]{Guo-L} to directly obtain

\begin{lemma}\label{landau.nonlinear.est}
The following estimate holds uniformly
\begin{align*}
\vert ( \hat{\Gamma}(\hat{f},\hat{g})(k),\hat{h}(k))_{L^2_v}\vert
\lesssim
(\Vert \hat{f}\Vert_{L^2_v}*_k\vert \hat{g}\vert_D)(k)
| \hat{h}(k)|_D,
\end{align*}
\end{lemma}

More precisely the estimate above can be found in \cite[Proposition 1, page 621]{MR3101794}, which is a simplification of the original estimate from \cite[Theorem 3, page 406]{Guo-L}.  We remark that the estimates in those papers involve weights and derivatives, however the above estimate in Lemma \ref{landau.nonlinear.est} only differs in the lack of weights and derivatives and the inclusion of the convolution $*_k$ in the terms.  And Lemma \ref{landau.nonlinear.est} follows directly from the exact same proofs.
Based on the above lemma, one can obtain the same estimate as stated in Proposition \ref{Torus Boltzmann Micro}. For brevity we omit the rest of the details in the Landau case since they are the same as in the Boltzmann case.

\section{Trilinear estimates}\label{sec4}

In this section we will treat trilinear estimates by three parts. The first part is concerned with the velocity weighted trilinear estimates as used in \eqref{ad.vwte}. Moreover, corresponding to estiamte the left-hand term of \eqref{pro.tpm.p3}, the second and third parts are devoted to considering the velocity weighted trilinear estimates on mixed variables in cases of the torus and the finite channel, respectively.

\subsection{Trilinear estimates in $L^2_v$ with velocity weight}

We first give some basic estimates on the velocity weighted trilinear terms in the following lemma whose proof can be found
in \cite[Lemma 10, pp.327]{SG-08-ARMA} and \cite[Lemma 2.3, pp.176; Lemma 2.4, pp.121]{DLYZ-VMB,FLLZ-2018}, respectively. Note that for {\it hard} potentials corresponding to either $-2\leq \ga\leq 1$ in Landau case or $\ga+2s\geq 0$ in Boltzmann case, we have $q=0$ and hence $w_{q,\vth}\equiv 1$, namely, it is not necessary to include any velocity weight in the {\it hard} potential cases. In fact, in all the cases under the hypothesis {\bf (H)}, there is no need to include any velocity weight for the purposes of establishing the existence theory.  However the velocity weight will be needed for deducing the sub-exponential time-decay in the {\it soft}  potential cases.

\begin{lemma}\label{bnp.es}
Let $(q,\vth)$ in the velocity weight function $w_{q,\vth}$ be chosen in terms of the hypothesis {\bf (H)} in \eqref{q}. Then, for the Landau operator,
it holds that
\begin{equation}\label{bnpld}
\begin{split}
\left|\left(
\Gamma(f,g), w^2_{q,\vartheta}h\right)_{L^2_{v}}\right|
\lesssim\left(\|w_{q,\vth}f\|_{L^2_v}
\left|w_{q,\vth}g\right|_{D}
+\left|w_{q,\vth}f\right|_{D}
\|w_{q,\vth}g\|_{L^2_v}
\right)\left|w_{q,\vth}h\right|_{D}.
\end{split}
\end{equation}
Similarly, for the non-cutoff Boltzmann operator,
it holds that
\begin{equation}\label{bnpbl.h}
\begin{split}
\left|\left(
\Gamma(f,g), w^2_{q,\vartheta}h\right)_{L^2_{v}}\right|
\lesssim\left(\|w_{q,\vth}f\|_{L^2_v}
\left|w_{q,\vth}g\right|_{D}
+\left|w_{q,\vth}f\right|_{D}
\|w_{q,\vth}g\|_{L^2_v}
\right)\left|w_{q,\vth}h\right|_{D},
\end{split}
\end{equation}
if $\ga+2s\geq 0$, and
\begin{multline}
\left|\left( \Gamma(f,g), w^2_{q,\vth} h\right)_{L^2_{v}}\right|
\lesssim
\\
\left\{\left\|\langle v\rangle^{\gamma/2+s}w_{q,\vth}f\right\|_{L^2_v}
\left|g\right|_D+\left\|\langle v\rangle^{\gamma/2+s} g\right\|_{L^2_v}
\left|w_{q,\vth}f\right|_D\right\}
\left|w_{q,\vth}h\right|_D
\\
+\min\left\{\left\|w_{q,\vth}f\right\|_{L^2_v}
\left\|\langle v\rangle^{\gamma/2+s}g\right\|_{L^2_v},\left\|g\right\|_{L^2_v}
\left\|\langle v\rangle^{\gamma/2+s}w_{q,\vth}f\right\|_{L^2_v}\right\}
\left|w_{q,\vth}h\right|_D
\\
+\left\|w_{q,\vth}g\right\|_{L^2_v}
\left\|\langle v\rangle^{\gamma/2+s}w_{q,\vth}f\right\|_{L^2_v}
\left|w_{q,\vth}h\right|_D,
\label{bnpbl}
\end{multline}
if $\ga+2s<0$.
\end{lemma}

\subsection{Trilinear estimates on mixed variables in isotropic case}

Recall that we have derived \eqref{est.fRHS} without any velocity weights.  Based on Lemma \ref{bnp.es}, we may employ the same idea to include the velocity weight  $w_{q,\vth}$ in our estimates.

\begin{lemma}\label{lem.tei}
Let $(q,\vth)$ in the velocity weight function $w_{q,\vth}$ be chosen in terms of the hypothesis {\bf (H)} in \eqref{q}. Then, for both the Landau and Boltzmann cases it holds that
\begin{align}
\int_{\Z^3_{{k}}}&\left(\int_0^T\left|\left(
\widehat{\Gamma(f,g)}, w^2_{q,\vartheta}\hat{h}\right)_{L^2_{v}}\right|dt\right)^{1/2}d\Si({k})\notag \\
\leq&C_\eta\bigg(\left\|w_{q,\vth} f\right\|_{L^1_{{k}}L^\infty_TL^2_{v}}
\left\|w_{q,\vth} g\right\|_{L^1_{{k}}L^2_TL^2_{v,D}}
\notag\\
&\qquad\ +\left\|w_{q,\vth}f\right\|_{L^1_{{k}}L^2_TL^2_{v,D}}
\left\|w_{q,\vth} g\right\|_{L^1_{{k}}L^\infty_TL^2_{v}}
\bigg)+\eta\left\|w_{q,\vth}h\right\|_{L^1_{{k}}L^2_TL^2_{v,D}},\label{lem.tei.1}
\end{align}
where the Fourier transform $\hat{\,\cdot\,}$ is taken in ${x}=(x_1,x_2,x_3)\in \T^3$,  $\eta>0$ is an arbitrary small constant, and $C_\eta$ is a universal large constant depending only on $\eta$.
\end{lemma}

\begin{proof}In the rest of this proof $\left( \cdot, \cdot \right)$ denotes the $L^2_{v}$ complex inner product.
We consider Proposition \ref{Torus Boltzmann Micro} and its proof.  In particular we recall \eqref{pro.tpm.p4} and the estimates below it.   Then to show \eqref{lem.tei.1} it suffices to verify  that
\begin{align}
&\left|\left(\widehat{\Gamma( f,g)}, w^2_{q,\vartheta}\hat{h}\right)\right|\notag\\
&\lesssim\int_{\Z^3_l}\left(\left\|w_{q,\vth}\hat{f}({k}-{l})\right\|_{L^2_v}
\left|w_{q,\vth}\hat{g}({l})\right|_{D}
+\left|w_{q,\vth}\hat{f}({k}-{l})\right|_{D}
\left\|w_{q,\vth}\hat{g}({l})\right\|_{L^2_v}
\right)
\notag
\\
& \quad \quad \quad \times \left|w_{q,\vth}\hat{h}(k)\right|_{D}~d\Si({l}).
\label{lem.tei.p1}
\end{align}
Indeed, by \eqref{nopdef}, it holds that
\begin{equation}
\label{lem.tei.p2}
\left(\widehat{\Gamma( f,g)}({k}),w^2_{q,\vartheta}\hat{h}({k})\right)
=\int_{\R^3_v}w^2_{q,\vartheta}\mu^{-\frac{1}{2}}(v)
\CF_x{Q(\mu^{\frac{1}{2}} f,\mu^{\frac{1}{2}}g)} (k)\bar{\hat{h}}({k})dv.
\end{equation}
Then, it follows from \eqref{bLop.e} and  \eqref{lem.tei.p2} together with
Fubini's theorem that
\begin{equation}
\begin{aligned}
&\left|\left(\widehat{\Gamma(f,g)}(k),w^2_{q,\vartheta}\hat{h}(k)\right)\right|
\\
&= \bigg|\int_{\R^3_v}dv\, w^2_{q,\vartheta}\mu^{-\frac{1}{2}}(v)\int_{\Z^3_l}d\Si(l)\,
\\
& \quad \quad \quad \times
\pa_{v_j}
\left\{\left[\psi^{jm}\ast_v\left(\mu^{\frac{1}{2}}\hat{f}({k}-{l})\right)\right]
\mu^{\frac{1}{2}}(v)\left[\pa_{v_m}\hat{g}({l})-\frac{v_m}{2}\hat{g}({l})\right]\right\}\bar{\hat{h}}({k})
\\
&
\quad -\int_{\R^3_v}dv\,w^2_{q,\vartheta}\mu^{-\frac{1}{2}}(v)\int_{\Z^3_l}d\Si (l)\,
\\
& \quad \quad \quad \times
\pa_{v_j}\left\{\left[\psi^{jm}\ast_v
\left\{\left(\pa_{v_m}\hat{f}-\frac{v_m}{2}\hat{f}\right)({k}-{l})\mu^{\frac{1}{2}}\right\}\right]
\mu^{\frac{1}{2}}(v)\hat{g}({l})\right\}\bar{\hat{h}}({k})\bigg|
\\
&
= \left|
\int_{\Z^3_l}d\Si (l) \int_{\R^3_v}dv\, \Gamma(\hat{f}({k}-{l}),\hat{g}({l}))
w^2_{q,\vartheta}\bar{\hat{h}}({k})
\right|
\\
&\le \int_{\Z^3_l} \left| \left(\Gamma(\hat{f}({k}-{l}),\hat{g}({l})),w^2_{q,\vartheta}\hat{h}(k)\right)\right|d\Si({l}),
\notag
\end{aligned}
\end{equation}
for the Landau collision operator, and
\begin{equation}
\begin{aligned}
&\left| (\widehat{\Gamma(f,g)}({k}),\hat{h}({k}))\right|\\
&=\left| \int_{\R^3_v}dv \int_{\R^3_u}du
\int_{\mathbb{S}^{2}}d\si\, B \mu^{\frac{1}{2}}(u)w^2_{q,\vartheta} \left(\widehat{f(u')g(v')}({k})-
\widehat{f(u)g(v)}({k})\right)\bar{\hat{h}}({k})\right|\\
&=\left\vert \int_{\R^3_v}dv \int_{\R^3_u}du
\int_{\mathbb{S}^{2}}d\si\,  B\mu^{\frac{1}{2}}(u)w^2_{q,\vartheta}\right.\\
&\qquad\qquad\qquad \left.\times\int_{\Z^3_l}d\Si (l) \left(\widehat{f(u')}({k}-{l})
\widehat{g(v')}(l)-\widehat{f(u)}({k}-{l})\widehat{g(v)}(l)\right)\bar{\hat{h}}({k})\right\vert\\
&=\left\vert \int_{\Z^3_l}d\Si (l) \int_{\mathbb{R}^3_v}dv\,
\Gamma(\hat{f}({k}-{l}),\hat{g}({l}))w^2_{q,\vartheta}\bar{\hat{h}}({k})  \right\vert\\
&\le \int_{\Z^3_l} \left\vert \left(\Gamma(\hat{f}({k}-{l}),\hat{g}({l})),w^2_{q,\vartheta}\hat{h}(k)\right)\right\vert d\Si({l}),
\end{aligned}\notag
\end{equation}
for the Boltzmann collision operator. Hence, from the above two estimates, one can further derive \eqref{lem.tei.p1} with the help of \eqref{bnpld} for the Landau operator, and \eqref{bnpbl.h} and \eqref{bnpbl} for the Boltzmann operator. This completes the proof of Lemma \ref{lem.tei}.
\end{proof}

The following lemma will be useful for dealing with the nonlinear terms arising from the macroscopic estimates in Section \ref{sec5.1}. The proof of Lemma \ref{tril21} is omitted for brevity of presentation, since it is similar to and much easier than that of Lemma \ref{tril2} below in the case of the finite channel where the proof will be provided.  

\begin{lemma}\label{tril21}
Assume that $\zeta(v)$ depends only on $v$ and decays rapidly at infinity. Then, for $|\al|=0,1$, it holds that 
\begin{multline}
\int_{\Z^3_{k}}\left(\int_0^T\left|\left(
\CF_{\bar{x}}{\Gamma(f,g)},\zeta(v)\right)_{L_v^2}\right|^2dt\right)^{1/2}d\Si(k)\notag
\\
\leq C\bigg(\left\| f\right\|_{L^1_{k}L^\infty_TL^2_{v}}
\left\| g\right\|_{L^1_{k}L^2_TL^2_{v,D}}
+\left\|f\right\|_{L^1_{k}L^2_TL^2_{v,D}}
\left\| g\right\|_{L^1_{k}L^\infty_TL^2_{v}}
\bigg).
\notag
\end{multline}
\end{lemma}


\subsection{Trilinear estimates on mixed variables in anisotropic case}

Similar to the proof of Lemma \ref{lem.tei}, the proof of the  following lemma is also based on Lemma \ref{bnp.es} and some significant properties of the $L^1_{\bar{k}}$ Wiener algebra space $A(\T^2)$.

\begin{lemma}\label{keynp.es} Let $(q,\vth)$ in the velocity weight function $w_{q,\vth}$ be chosen in terms of the hypothesis {\bf (H)} in \eqref{q}. Then, for both Landau and Boltzmann cases it holds that   for $|\al|=0,1$, we have
\begin{align}
\int_{\Z^2_{\bar{k}}}&\left(\int_0^T\left|\left(
\CF_{\bar{x}}{\Gamma(\pa^{\al} f,g)}, w^2_{q,\vartheta}\hat{h}\right)\right|dt\right)^{1/2}d\Si(\bar{k})
\notag \\
\leq&C_\eta\bigg(\left\|w_{q,\vth}\pa^{\al} f\right\|_{L^1_{\bar{k}}L^\infty_TL^2_{x_1,v}}
\left\|w_{q,\vth} g\right\|_{L^1_{\bar{k}}L^2_TH^1_{x_1}L^2_{v,D}}
\notag \\&\qquad+\left\|w_{q,\vth}\pa^{\al} f\right\|_{L^1_{\bar{k}}L^2_TL^2_{x_1}L^2_{v,D}}
\left\|w_{q,\vth} g\right\|_{L^1_{\bar{k}}L^\infty_TH^1_{x_1}L^2_{v}}
\bigg)+\eta\left\|w_{q,\vth}h\right\|_{L^1_{\bar{k}}L^2_TL^2_{x_1}L^2_{v,D}},\label{tri.es1}
\end{align}
and
\begin{align}
\int_{\Z^2_{\bar{k}}}&\left(\int_0^T\left|\left(
\CF_{\bar{x}}{\Gamma( f,\pa^{\al}g)}, w^2_{q,\vartheta}\hat{h}\right)\right|dt\right)^{1/2}d\Si(\bar{k})
\notag\\
\leq&C_\eta\bigg(\left\|w_{q,\vth}\pa^{\al} g\right\|_{L^1_{\bar{k}}L^\infty_TL^2_{x_1,v}}
\left\|w_{q,\vth} f\right\|_{L^1_{\bar{k}}L^2_TH^1_{x_1}L^2_{v,D}}
\notag\\&\qquad+\left\|w_{q,\vth}\pa^{\al} g\right\|_{L^1_{\bar{k}}L^2_TL^2_{x_1}L^2_{v,D}}
\left\|w_{q,\vth} f\right\|_{L^1_{\bar{k}}L^\infty_TH^1_{x_1}L^2_{v}}
\bigg)+\eta\left\|w_{q,\vth}h\right\|_{L^1_{\bar{k}}L^2_TL^2_{x_1}L^2_{v,D}},\label{tri.es2}
\end{align}
where the Fourier transform $\hat{\,\cdot\,}$ is taken in $\bar{x}=(x_2,x_3)\in \T^2$, the inner product $(\cdot,\cdot)$ is taken over $L^2_{x_1,v}$ where $x_1 \in I = (-1,1)$, $\eta>0$ is an arbitrary constant, and $C_\eta$ is a universal constant depending only on $\eta$.
\end{lemma}

\begin{proof}
We only prove \eqref{tri.es1}, since the proof of \eqref{tri.es2} is very similar. For this purpose, we first show that
\begin{align}
&\left|\left(\CF_{\bar{x}}{\Gamma(\pa^{\al} f,g)}, w^2_{q,\vartheta}\hat{h}\right)_{L^2_{x_1,v}}\right|\notag\\
&\lesssim\int_{\Z^2_{\bar{l}}}\int_{I}\left(\left\|w_{q,\vth}\widehat{\pa^{\al} f}(\bar{k}-\bar{l})\right\|_{L^2_v}
\left|w_{q,\vth}\hat{g}(\bar{l})\right|_{D}\right.\notag\\
&\qquad\qquad\qquad\left.+\left|w_{q,\vth}\widehat{\pa^{\al} f}(\bar{k}-\bar{l})\right|_{D}
\left\|w_{q,\vth}\hat{g}(\bar{l})\right\|_{L^2_v}
\right)\left|w_{q,\vth}\hat{h}\right|_{D}d{x_1}~d\Si(\bar{l}).\label{bnp}
\end{align}
To verify \eqref{bnp}, for the Landau operator or Boltzmann operator, similar to the proof of Lemma \ref{lem.tei}, we get from \eqref{nopdef} and
Fubini's theorem that
\begin{align}
&\left|\left(\CF_{\bar{x}}{\Gamma(\pa^{\al} f,g)}(\bar{k}),w^2_{q,\vartheta}\hat{h}(\bar{k})\right)_{L^2_{x_1,v}}\right|\notag \\
&=\left|\int_{I}dx_1\int_{\R^3_v}dv\,w^2_{q,\vartheta}\mu^{-\frac{1}{2}}(v)
\CF_{\bar{x}}{Q(\mu^{\frac{1}{2}}\pa^{\al} f,\mu^{\frac{1}{2}}g)}(\bar{k})\bar{\hat{h}}(\bar{k})\right|\notag \\
&= \left|\int_{I}dx_1\int_{\R^3_v}dv\int_{\Z^2_{\bar{l}}}d\Si (\bar{l})\,\Gamma(\widehat{\pa^{\al} f}(\bar{k}-\bar{l}),\hat{g}(\bar{l}))
w^2_{q,\vartheta}\bar{\hat{h}}(\bar{k})\right|  \notag \\
&= \left|\int_{\Z^2_{\bar{l}}}d\Si (\bar{l}) \int_{I}dx_1\int_{\R^3_v}dv\,\Gamma(\widehat{\pa^{\al} f}(\bar{k}-\bar{l}),\hat{g}(\bar{l}))
w^2_{q,\vartheta}\bar{\hat{h}}(\bar{k})\right|  \notag \\
&\le \int_{\Z^2_{\bar{l}}} \left| \left(\Gamma(\widehat{\pa^{\al} f}(\bar{k}-\bar{l}),\hat{g}(\bar{l})),w^2_{q,\vartheta}\hat{h}(\bar{k})\right)\right|d\Si(\bar{l}).\label{blnp}
\end{align}
Hence \eqref{bnp} follows from \eqref{blnp} with the help of Lemma \ref{bnp.es}. Next, we let $J$ denote the left-hand term of \eqref{tri.es1}:
\begin{equation}\notag
J := \int_{\Z^2_{\bar{k}}}\left(\int_0^T\left|\left(
\CF_{\bar{x}}{\Gamma(\pa^{\al} f,g)}, w^2_{q,\vartheta}\hat{h}\right)_{L^2_{x_1,v}}\right|dt\right)^{1/2}d\Si(\bar{k}).
\end{equation}
Then by applying \eqref{bnp} together with the functional Sobolev embedding inequality $\|g\|_{L^\infty(I)} \lesssim \|g\|_{H^1(I)}$, one has
\begin{multline}
J
\lesssim
\int_{\Z^2_{\bar{k}}}\left(\int_0^T\int_{\Z^2_{\bar{l}}}\sum\limits_{\al'\leq 1}\left\|w \widehat{\pa^{\al} f}(\bar{k}-\bar{l})\right\|
\left\|w \pa^{\al'}_{x_1}\hat{g}(\bar{l})\right\|_{D}d\Si(\bar{l})
\left\|w \hat{h}(\bar{k})\right\|_{D}dt\right)^{1/2}d\Si(\bar{k})
\\
+\int_{\Z^2_{\bar{k}}}\left(\int_0^T\int_{\Z^2_{\bar{l}}}\sum\limits_{\al'\leq 1}\left\|w \widehat{\pa^{\al} f}(\bar{k}-\bar{l})\right\|_{D}
\left\|w\pa^{\al'}_{x_1}\hat{g}(\bar{l})\right\| d\Si(\bar{l})\left\|w \hat{h}(\bar{k})\right\|_{D}dt\right)^{1/2}d\Si(\bar{k}).
\notag
\end{multline}
Now above and in the rest of this proof the norm $\| \cdot \|$ is $L^2_{x_1, v}$ and the norm $\| \cdot \|_{D}$ is as defined in \eqref{def.fc.disDL}.  Also the above integrals both contain an implicit $d\Si(\bar{k})$, and the weight $w = w_{q,\vth}$ as usual.

It further follows from Cauchy-Schwarz's inequality in the time integral that
\begin{multline*}
J
\lesssim
\\
C_\eta\int_{\Z^2_{\bar{k}}}\left(\int_0^T\left[\int_{\Z^2_{\bar{l}}}\sum\limits_{\al'\leq 1}\left\|w_{q,\vth}\widehat{\pa^{\al} f}(\bar{k}-\bar{l})\right\|
\left\|w_{q,\vth}\pa^{\al'}_{x_1}\hat{g}(\bar{l})\right\|_{D}d\Si(\bar{l})\right]^2dt\right)^{1/2}d\Si(\bar{k})
\\
\qquad+C_\eta\int_{\Z^2_{\bar{k}}}\left(\int_0^T\left[\int_{\Z^2_{\bar{l}}}\sum\limits_{\al'\leq 1}\left\|w_{q,\vth}\widehat{\pa^{\al} f}(\bar{k}-\bar{l})\right\|_{D}
\left\|w_{q,\vth}\pa^{\al'}_{x_1}\hat{g}(\bar{l})\right\| d\Si(\bar{l})\right]^2dt\right)^{1/2}d\Si(\bar{k})
\\
\qquad+\eta\int_{\Z^2_{\bar{k}}}\left(\int_0^T\left\|w_{q,\vth}\hat{h}(\bar{k})\right\|^2_{D}dt\right)^{1/2}d\Si(\bar{k}),
\end{multline*}
for an arbitrary small constant $\eta>0$.  Now we use Minkowski's inequality as in \eqref{minkowski.ineq} to  deduce from the above estimate that
\begin{multline*}
J
\lesssim
\\
C_\eta\int_{\Z^2_{\bar{k}}}\sup\limits_{0\le t\le T}\left\|w_{q,\vth}\widehat{\pa^{\al} f}(\bar{k})\right\| d\Si(\bar{k})\int_{\Z^2_{\bar{l}}}\left(\int_0^T\sum\limits_{\al'\leq 1}
\left\|w_{q,\vth}\pa^{\al'}_{x_1}\hat{g}(\bar{l})\right\|^2_{D}dt\right)^{1/2}d\Si(\bar{l})
\\
+C_\eta\int_{\Z^2_{\bar{l}}}\sup\limits_{0\le t\le T}\sum\limits_{\al'\leq 1}\left\|w_{q,\vth}\pa^{\al'}_{x_1}\hat{g}(\bar{l})\right\| d\Si(\bar{l})
\int_{\Z^2_{\bar{k}}}\left(\int_0^T\left\|w_{q,\vth}\widehat{\pa^{\al} f}(\bar{k})\right\|^2_{D}
dt\right)^{1/2}d\Si(\bar{k})
\\
+\eta\int_{\Z^2_{\bar{k}}}\left(\int_0^T\left\|w_{q,\vth}\hat{h}(\bar{k})\right\|^2_{D}dt\right)^{1/2}d\Si(\bar{k}),
\end{multline*}
that is,
\begin{equation*}
\begin{split}
J &\lesssim C_\eta\left\|w_{q,\vth}\pa^{\al} f\right\|_{L^1_{\bar{k}}L^\infty_TL^2_{x_1,v}}
\left\|w_{q,\vth} g\right\|_{L^1_{\bar{k}}L^2_TH^1_{x_1}L^2_{v,D}}
\\
&\qquad+C_\eta\left\|w_{q,\vth}\pa^{\al} f\right\|_{L^1_{\bar{k}}L^2_TL^2_{x_1}L^2_{v,D}}
\left\|w_{q,\vth} g\right\|_{L^1_{\bar{k}}L^\infty_TH^1_{x_1}L^2_{v}}
+\eta\left\|w_{q,\vth}\hat{h}\right\|_{L^1_{\bar{k}}L^2_TL^2_{x_1}L^2_{v,D}}.
\end{split}
\end{equation*}
This then proves \eqref{tri.es1}, and thus completes the proof of Lemma \ref{keynp.es}.
\end{proof}

Similar to Lemma \ref{tril21}, the following lemma will be used for treating the nonlinear term in the macroscopic estimates. We will provide a brief proof for completeness. 

\begin{lemma}\label{tril2}
Assume that $\zeta(v)$ depends only on $v$ and decays rapidly at infinity. Then, for $|\al|=0,1$,  it holds that 
\begin{multline}
\int_{\Z^2_{\bar{k}}}\left(\int_0^T\left\|\left(
\CF_{\bar{x}}{\Gamma(\pa^{\al} f,g)},\zeta(v)\right)_{L_v^2}\right\|_{L^2_{x_1}}^2dt\right)^{1/2}d\Si(\bar{k})
\\
\leq C\bigg(\left\|\pa^{\al} f\right\|_{L^1_{\bar{k}}L^\infty_TL^2_{x_1,v}}
\left\| g\right\|_{L^1_{\bar{k}}L^2_TH^1_{x_1}L^2_{v,D}}
\\
+\left\|\pa^{\al} f\right\|_{L^1_{\bar{k}}L^2_TL^2_{x_1}L^2_{v,D}}
\left\| g\right\|_{L^1_{\bar{k}}L^\infty_TH^1_{x_1}L^2_{v}}
\bigg),\label{tri.l21}
\end{multline}
and
\begin{multline}
\int_{\Z^2_{\bar{k}}}\left(\int_0^T\left\|\left(
\CF_{\bar{x}}{\Gamma(f,\pa^{\al}g)},\zeta(v)\right)_{L_v^2}\right\|_{L^2_{x_1}}^2dt\right)^{1/2}d\Si(\bar{k})
\\
\leq C\bigg(\left\|\pa^{\al} g\right\|_{L^1_{\bar{k}}L^\infty_TL^2_{x_1,v}}
\left\| f\right\|_{L^1_{\bar{k}}L^2_TH^1_{x_1}L^2_{v,D}}\\
+\left\|\pa^{\al} g\right\|_{L^1_{\bar{k}}L^2_TL^2_{x_1}L^2_{v,D}}
\left\| f\right\|_{L^1_{\bar{k}}L^\infty_TH^1_{x_1}L^2_{v}}
\bigg).\label{tri.l21c}
\end{multline}
\end{lemma}

\begin{proof}
As in the proof for Lemma \ref{keynp.es}, we prove \eqref{tri.l21} only, since the proof of \eqref{tri.l21c} is similar. Firstly, to treat the left-hand term of \eqref{tri.l21c}, we write that
\begin{multline}
\left(\int_0^T\left\|\left(
\CF_{\bar{x}}{\Gamma(\pa^{\al} f,g)},\zeta(v)\right)_{L_v^2}\right\|_{L^2_{x_1}}^2dt\right)^{1/2}\\
= \left(\int_0^T\int_{I}\left(\int_{\R^3_v}\int_{\Z^2_{\bar{l}}}\,\Gamma(\widehat{\pa^{\al} f}(\bar{k}-\bar{l}),\hat{g}(\bar{l}))
\zeta(v)\,d\Si (\bar{l})dv\right)^2dt\,dx_1\right)^{1/2}.
\label{tri.l21.p1}
\end{multline} 
From Lemma \ref{bnp.es}, the above term is bounded by
\begin{multline*}
C_\zeta\left(\int_0^T\int_{I}\left(\int_{\Z^2_{\bar{l}}}\left\|\widehat{\pa^{\al} f}(\bar{k}-\bar{l})\right\|_{L^2_v}
\left|\hat{g}(\bar{l})\right|_{D}d\Si (\bar{l})\right)^2dt\,dx_1\right)^{1/2}  \\
+C_\zeta\left(\int_0^T\int_{I}\left(\int_{\Z^2_{\bar{l}}}\left\|\hat{g}(\bar{l})\right\|_{L^2_v}
\left|\widehat{\pa^{\al} f}(\bar{k}-\bar{l})\right|_{D}d\Si (\bar{l})\right)^2dt\,dx_1\right)^{1/2},
\end{multline*} 
where $C_\zeta>0$ is a universal constant depending only on $\zeta$. Further by Minkowski's inequality $\|\|\cdot\|_{L^1_{\bar{l}}}\|_{L^2_{t,x_1}}\leq \|\|\cdot\|_{L^2_{t,x_1}}\|_{L^1_{\bar{l}}}$, \eqref{tri.l21.p1} is then bounded by 
\begin{multline*}
C_\zeta\int_{\Z^2_{\bar{l}}}\left(\int_0^T\int_{I}\left(\left\|\widehat{\pa^{\al} f}(\bar{k}-\bar{l})\right\|_{L^2_v}
\left|\hat{g}(\bar{l})\right|_{D}\right)^2dt\,dx_1\right)^{1/2}d\Si (\bar{l})   \\
+C_\zeta\int_{\Z^2_{\bar{l}}}\left(\int_0^T\int_{I}\left(\left\|\hat{g}(\bar{l})\right\|_{L^2_v}
\left|\widehat{\pa^{\al} f}(\bar{k}-\bar{l})\right|_{D}\right)^2dt\,dx_1\right)^{1/2}d\Si (\bar{l}). 
\end{multline*}
According to Sobolev's inequality $\|g\|_{L^\infty(I)} \lesssim \|g\|_{H^1(I)}$, the above term can be bounded by 
\begin{multline}\label{zeta1}
 C_\zeta\int_{\Z^2_{\bar{l}}}\sup\limits_{0\leq t\leq T}\left\|\widehat{\pa^{\al} f}(\bar{k}-\bar{l})\right\|_{L^2_{x_1,v}}
\sum\limits_{0\leq \al_1\leq1}\left(\int_0^T\left\|\widehat{\pa^{\al_1}_{x_1} g}(\bar{l})\right\|_{D}^2dt\right)^{1/2}d\Si (\bar{l})   \\
+C_\zeta\int_{\Z^2_{\bar{l}}}\sup\limits_{0\leq t\leq T}\sum\limits_{0\leq \al_1\leq1}\left\|\widehat{\pa^{\al_1}_{x_1} g}(\bar{l})\right\|_{L^2_{x_1,v}}
\left(\int_0^T\left\|\widehat{\pa^{\al} f}(\bar{k}-\bar{l})\right\|_{D}^2dt\right)^{1/2}d\Si (\bar{l}). 
\end{multline}
Consequently, \eqref{tri.l21} follows from taking the integration of \eqref{zeta1} in $\bar{k}\in \Z^2_{\bar{k}}$ 
with respect to $d\Si (\bar{k})$ and further applying Fubini's Theorem. This completes the proof of Lemma \ref{tril2}.
\end{proof}

\section{Macroscopic estimates}\label{sec5}

Throughout this section, we let $T>0$ be an arbitrary fixed constant. We further emphasize that the universal constant $C>0$ in all estimates below is independent of $T$.  Let us start from the  macro-micro decomposition of the solution $f$, i.e. we split $f$
as $f=\FP f+\{\FI-\FP\}f$, where
\begin{equation}\label{abcdef}
\FP f=\{a+b\cdot v+\frac{1}{2}(|v|^2-3)c\}{\mu}^{\frac{1}{2}}.
\end{equation}
This expression above defines $[a,b,c]$, where $[\cdot, \cdot ,\cdot]$ represents a vector.

\subsection{Isotropic case}\label{sec5.1}

In this section we will derive the uniform a priori estimates for the macroscopic part of a solution to the  equation
\begin{align}\label{leqh}
\partial_t f+v\cdot \nabla_x f + Lf =H,\ t>0,\  x\in \mathbb{T}^3, v\in \mathbb{R}^3,
\end{align}
with initial data given as in \eqref{idf}, where  generally the inhomogeneous source term $H=H(t,x,v)$ is assumed to be a functional of an arbitrary distribution $h(t,x,v)$ and  $H=H(h(x,v))\in \CN^\perp(L)$, where $\CN^\perp(L)$ is the perpendicular to the null space of $L$ for any $t$ and $x$. For the proof we will follow the same strategy as in \cite{EGKM-13} by the dual argument.

\begin{theorem}\label{abcest.torus}
Under the assumptions of Theorem \ref{Torus Existence}, it holds that
\begin{align}
\Vert [a,b,c]\Vert_{L^1_k L^2_T} &\lesssim\Vert \{ \mathbf{I}-\mathbf{P}\} f \Vert_{L^1_k L^2_T L^2_{v,D}} +\Vert f \Vert_{L^1_k L^\infty_T L^2_v} + \Vert f_0 \Vert_{L^1_k L^2_v} \notag\\
&\qquad+ \int_{\mathbb{Z}^3_k} \Big( \int^T_0 \vert  ( \hat{H}(t,k), \mu^{\frac{1}{4}}) \vert^2 dt \Big)^{1/2} d\Sigma (k),
\label{abcest.torus1}
\end{align}
where the inner product in the last term is taken over $L^2$ in $v$.
\end{theorem}

\begin{proof}
By taking the following velocity moments
\begin{equation*}
    \mu^{\frac{1}{2}}, v_j\mu^{\frac{1}{2}}, \frac{1}{6}(|v|^2-3)\mu^{\frac{1}{2}},
    (v_j{v_m}-1)\mu^{\frac{1}{2}}, \frac{1}{10}(|v|^2-5)v_j \mu^{\frac{1}{2}}
\end{equation*}
with {$1\leq j,m\leq 3$} for the equation \eqref{LLeq} with $\Ga(f,f)$ replaced by $H$,
one sees that  
the coefficient functions $[a,b,c]=[a,b,c](t,x)$ satisfy
the fluid-type system
\begin{equation}\label{mac.law.torus}
    \left\{\begin{array}{l}
      \dis \pa_t a +\nabla_x \cdot b=0,\\
      \dis \pa_t b +\na_x (a+2c)+\na_x\cdot \highG (\{\FI-\FP\} f)=0,\\[3mm]
      \dis \pa_t c +\frac{1}{3}\na_x\cdot b +\frac{1}{6}\na_x\cdot
      \Lambda (\{\FI-\FP\} f)=0,\\[3mm]
      \dis \pa_t[\highG_{{ jm}}(\{\FI-\FP\} f)+2c\de_{{ jm}}]+\pa_jb_m+\pa_m
      b_j=\highG_{jm}(\mathbbm{r}+\mathbbm{h}),\\[3mm]
      \dis \pa_t \Lambda_j(\{\FI-\FP\} f)+\pa_j c = \Lambda_j(\mathbbm{r}+\mathbbm{h}),
    \end{array}\right.
\end{equation}
where the
high-order moment functions $\highG=(\highG_{jm}(\cdot))_{3\times 3}$ and
$\highB=(\highB_j(\cdot))_{1\leq j\leq 3}$ are respectively defined by
\begin{eqnarray}
  \highG_{jm}(f) = \left ((v_jv_m-1)\mu^{\frac{1}{2}}, f\right),\ \ \
  \highB_j(f)=\frac{1}{10}\left ((|v|^2-5)v_j\mu^{\frac{1}{2}},
  f\right),\notag
\end{eqnarray}
with the inner product taken with respect to velocity variable $v$ only, and the terms $\mathbbm{r}$ and $\mathbbm{h}$ on the right are given by
\begin{eqnarray}\notag
\mathbbm{r}= -{v}\cdot \na_{{x}} \{\FI-\FP\}f,\ \ \mathbbm{h}=-L \{\FI-\FP\}f+H.
\end{eqnarray}
Note that the conservation laws \eqref{pt.id.cl.1}, \eqref{pt.id.cl.2} and \eqref{pt.id.cl.3} imply that
\begin{equation*}
[\widehat{a,b,c}](t,0)=0
\end{equation*}
for any $t\geq 0$.

In order to carry out the estimate in a unified way, we take a general function as
$$
\hat{\Phi}(t,k,v)\in C^1 ( (0,\infty)\times \mathbb{Z}^3 \times \mathbb{R}^3),
$$
which will be fixed later according to the different cases that we study. Applying the Fourier transform to \eqref{leqh}, taking the inner product of it and $\hat{\Phi}$ in $L^2_v$, then integrating the resultant over $[0,T]$, we have
\begin{multline*}
(\hat{f},\hat{\Phi})|_{t=T}-(\hat{f},\hat{\Phi})|_{t=0} -\int^T_0 (\hat{f},\partial_t \hat{\Phi})dt -\int^T_0 (\hat{f},v\cdot  ik\hat{\Phi})dt +\int^T_0 (L\hat{f},\hat{\Phi}) dt
\\
= \int^T_0 (\hat{H},\hat{\Phi})dt.
\end{multline*}
Note that above and in the rest of this proof $( \cdot, \cdot) = ( \cdot, \cdot)_{L^2_v}$.
Here we have used the shorthand notions $(\hat{f},\hat{\Phi})|_{t=T}=(\hat{f},\hat{\Phi})(T)$ and $(\hat{f},\hat{\Phi})|_{t=0}=(\hat{f},\hat{\Phi})(0)$.
Plugging in the macro-micro decomposition yields
\begin{multline*}
-\int^T_0 (\mathbf{P}\hat{f}, v\cdot ik\hat{\Phi})dt
=
(\hat{f},\hat{\Phi})(0)-(\hat{f},\hat{\Phi})(T)+\int^T_0 (\hat{f},\partial_t \hat{\Phi})dt
\\
+\int^T_0 ((\mathbf{I}-\mathbf{P})\hat{f},v\cdot ik\hat{\Phi})dt
-\int^T_0 (L\hat{f},\hat{\Phi}) dt +\int^T_0 (\hat{H},\hat{\Phi})dt.
\end{multline*}
Now we will define the following notation:  
\begin{equation*}
\left\{\begin{aligned}
    &S_1=(\hat{f},\hat{\Phi})(0)-(\hat{f},\hat{\Phi})(T),\\
    &S_2=\int^T_0 (\hat{f},\partial_t\hat{\Phi})dt,\\
    &S_3=\int^T_0 ((\mathbf{I}-\mathbf{P})\hat{f}, v\cdot ik\hat{\Phi} )dt,\\
   & S_4=-\int^T_0 (L\hat{f},\hat{\Phi}) dt +\int^T_0 (\hat{H},\hat{\Phi})dt.
\end{aligned}\right.
\end{equation*}

First, we consider the estimate of $c$. We choose
\begin{align*}
\hat{\Phi}= \hat{\Phi}_c=(\vert v \vert^2-5) \{ v\cdot ik\hat{\phi_c}(t,k)\}\mu^{\frac{1}{2}},
\end{align*}
where $\phi_c$ satisfies
\begin{align}\label{eqn.psi.c}
\vert k\vert^2 \hat{\phi}_c(t,k)=\hat{c}(t,k).
\end{align}
Note that since $\hat{c}(t,0)\equiv 0$, we formally write $\hat{\phi}_c(t,k)=\hat{c}(t,k)/|k|^2$ for any $k\in \Z^3$ with the understanding that we define $\hat{\phi}_c(t,0)=0$. 
By this choice, we can calculate that
\begin{multline*}
-\int^T_0 (\mathbf{P}\hat{f}, v\cdot ik\hat{\Phi})dt
\\
=-\sum_{j} \int^T_0 (\{\hat{a}+\hat{b}\cdot v +\frac{1}{2}(\vert v\vert^2-3) \hat{c}\}\mu^{\frac{1}{2}}, v_j ik_j \hat{\Phi}_c)dt
\\
=-\sum_{j,n} \int^T_0 (\{\hat{a}+\hat{b}\cdot v + \frac{1}{2}(\vert v \vert^2 -3)\hat{c}\} \mu^{\frac{1}{2}}, v_j v_n (\vert v\vert^2 -5) \mu^{\frac{1}{2}} (-k_jk_n)\hat{\phi_c})dt
\\
= \int^T_0 (\hat{c},\vert k\vert^2\hat{\phi_c})dt=\int^T_0 |\hat{c}(t,k)|^2 dt,
\end{multline*}
where we have used the orthogonality of the different integrands for the third equality. Now we estimate the $S_j$'s.
Due to $k\in \mathbb{Z}^3$ it holds that
\begin{align*}
\vert \hat{\Phi}_c (t,k)\vert \le C\mu^{\frac{1}{4}}\vert k\vert \vert \hat{\phi}_c (t,k)\vert \le C\mu^{\frac{1}{4}}\vert k\vert^2 \vert \hat{\phi}_c (t,k) \vert=C\mu^{\frac{1}{4}} \vert \hat{c}(t,k)\vert,
\end{align*}
and then we have $\Vert \hat{\Phi}_c(t,k)\Vert_{L^2_v} \lesssim |\hat{c}(t,k)|\lesssim \Vert \hat{f}(t,k)\Vert_{L^2_v}$. Thus
\begin{align*}
\vert S_1\vert \lesssim \Vert \hat{f}(k,T)\Vert_{L^2_v}^2+\Vert \hat{f}_0(k)\Vert_{L^2_v}^2.
\end{align*}
To obtain the estimate of $S_2$, we first notice
\begin{align*}
| \partial_t \hat{c}(t,k)| \lesssim \vert k\vert (|\hat{b}(t,k)|+|\{\mathbf{I}-\mathbf{P}\}\hat{f}(t,k)|_D)
\end{align*}
by the third equation of \eqref{mac.law.torus}. Therefore, we have
\begin{align*}
\vert S_2\vert &\le \int^T_0 \vert (\hat{f},\partial_t \hat{\Phi}_c)\vert dt = \int^T_0 \vert (\{\mathbf{I}-\mathbf{P}\})\hat{f},\partial_t \hat{\Phi}_c)\vert dt\\
&\lesssim \eta \int^T_0 \Vert \partial_t \hat{\Phi}_c(t,k) \Vert_{L^2_v}^2 dt +C_\eta\int^T_0 | \{\mathbf{I}-\mathbf{P}\}\hat{f}(t,k) |_D^2 dt\\
&\lesssim \eta \int^T_0\frac{| \partial_t\hat{c}(t,k) |^2}{|k|^2} dt +C_\eta\int^T_0 | \{\mathbf{I}-\mathbf{P}\}\hat{f}(t,k) |_D^2 dt\\
&\lesssim \eta \int^T_0 |\hat{b}(t,k) |^2 dt +C_\eta\int^T_0 | \{\mathbf{I}-\mathbf{P}\}\hat{f}(t,k) |_D^2 dt.
\end{align*}
Note that the equality above is due to the choice of the velocity moments in definition of $\hat{\Phi}_c$.
By H\"{o}lder's inequality and \eqref{eqn.psi.c}, we have
\begin{align*}
\vert S_3 \vert \lesssim \eta \int^T_0 |\hat{c}(t,k)|^2 dt +C_\eta \int^T_0 | \{\mathbf{I}-\mathbf{P}\}\hat{f}(t,k)|_D^2 dt.
\end{align*}
Regarding the estimate of $S_4$, we can also show that
\begin{align*}
\vert S_4\vert\lesssim \eta \int^T_0 |\hat{c}(t,k)|^2 dt + C_\eta \int^T_0 | \{\mathbf{I}-\mathbf{P}\}\hat{f}(t,k)|_D^2dt +C_\eta \int^T_0 \vert (\vert \hat{H}(t,k) \vert , \mu^{\frac{1}{4}})\vert^2 dt.
\end{align*}
In particular in $S_4$ we use that $L\hat{f} = L\{\mathbf{I}-\mathbf{P}\}\hat{f}$, and then we can use the estimate \cite[Equation (6.12) on page 819]{GS}, recalling also Remark \ref{norm.remark}, then similar to the estimate for $S_1$ the above estimate follows.
Summing up the estimates of $S_j$'s with $\eta>0$ suitably small, we obtain
\begin{align}\label{torus.c}
\int^T_0 |\hat{c}(t,k)|^2 dt
&\lesssim
\Vert \hat{f}(k,T)\Vert_{L^2_v}^2 + \Vert \hat{f}_0(k)\Vert_{L^2_v}^2+ \eta \int^T_0 |\hat{b}(t,k)|^2 dt
\notag
\\
&\qquad+C_\eta \int^T_0 | \{\mathbf{I}-\mathbf{P}\}\hat{f}(t,k)|_D^2dt
\notag
\\
&\qquad\qquad
+C_\eta \int^T_0 \vert (\vert \hat{H}(t,k) \vert, \mu^{\frac{1}{4}})\vert^2 dt.
\end{align}

Next, we consider the estimate of $b$. For this purpose we choose
\begin{align*}
\hat{\Phi}=\hat{\Phi}_b =\sum_{m=1}^3 \hat{\Phi}^{J,m}_b,\ J=1,2,3,
\end{align*}
where
\begin{eqnarray*}
\hat{\Phi}^{J,m}_b
=\left\{\begin{array}{rll}
\begin{split}
&
\left\{ \vert v\vert^2 v_mv_J ik_m\hat{\phi}_J(t,k)-\frac{7}{2}(v_m^2-1) ik_J \hat{\phi}_J(t,k)
\right\}{\mu}^{\frac{1}{2}},\ \ J\neq m,
\\
&\frac{7}{2}(v_J^2-1)ik_J\hat{\phi}_J(t,k){\mu}^{\frac{1}{2}},\ \ J=m,
\end{split}
\end{array}\right.
\end{eqnarray*}
and
\begin{align}\label{eqn.psi.b}
\vert k\vert^2 \hat{\phi}_J(t,k)=\hat{b}_J(t,k).
\end{align}
Note again that since $\hat{b}_J(t,0)\equiv 0$, it is valid to formally write $\hat{\phi}_J(t,k)=\hat{b}_J(t,k)/|k|^2$ for any $k\in \Z^3$ with $\hat{\phi}_J(t,0)=0$. Under this choice we have
\begin{align*}
&-\sum_{m=1}^3 \int^T_0 (\mathbf{P}\hat{f}, v \cdot ik\hat{\Phi}^{J,m}_b)dt
\\
&=-\sum_{m=1}^3 \int^T_0 \sum_{j=1}^3 (\{\hat{a}+ \hat{b}\cdot v +\frac{1}{2}(\vert v\vert^2-3)\hat{c}\}\mu^{\frac{1}{2}}, v \cdot ik\hat{\Phi}^{J,m}_b)dt\\
&=-\sum_{m=1,m\neq J}^3 \int^T_0 (v_mv_J\mu^{\frac{1}{2}}\hat{b}_J, \vert v\vert^2 v_m v_J \mu^{\frac{1}{2}} (-k_m^2){\hat{\phi}_J})dt\\
&\qquad-\sum_{m=1,m\neq J}^3 \int^T_0 (v_mv_J\mu^{\frac{1}{2}}\hat{b}_m, \vert v\vert^2 v_m v_J \mu^{\frac{1}{2}} (-k_Jk_m){\hat{\phi}_J})dt\\
&\qquad+7\sum_{m=1,m\neq J}^3 \int^T_0 (\hat{b}_m, (-k_Jk_m){\hat{\phi}_J})dt-7\int^T_0 (\hat{b}_J, (-k_J^2){\hat{\phi}_J})dt\\
&=-7\sum_{m=1}^3 \int^T_0 (\hat{b}_J, (-k_m^2){\hat{\phi}_J})dt=7 \int^T_0 | \hat{b}_J(t,k)|^2dt.
\end{align*}
The cancellation of the terms  above is due to the choice of $\hat{\Phi}_b$. Here we will only establish the estimate of $S_2$ since the other terms can be controlled similarly using the methods in the previous case.
Since the second equation of \eqref{mac.law.torus} gives
\begin{align*}
| \partial_t \hat{b}(t,k)| \lesssim \vert k\vert \left(|(a+c)(t,k)| + | \{\mathbf{I}-\mathbf{P}\}\hat{f}(t,k)|_D\right),
\end{align*}
it holds by \eqref{eqn.psi.b} that
\begin{multline*}
\vert S_2\vert
\le \sum_m \int^T_0 \vert (\{\mathbf{I}-\mathbf{P}\}\hat{f}, \partial_t \hat{\Phi}^{J,m}_b)\vert dt +\sum_m \int^T_0 \vert (\mathbf{P}\hat{f}, \partial_t \hat{\Phi}^{J,m}_b)\vert dt
\\
\lesssim \eta \sum_m \int^T_0 \Vert \partial_t \hat{\Phi}^{J,m}_b(t,k)\Vert_{L^2_v}^2 dt +C_\eta \int^T_0 | \{\mathbf{I}-\mathbf{P}\} \hat{f}(t,k) |^2_D dt
\\
+C_\eta \int^T_0 |\hat{c}(t,k)|^2dt
\\
\lesssim \eta \sum_m \int^T_0 \vert \partial_t k \hat{\phi}_J (t,k)\vert^2 dt +C_\eta \int^T_0 | \{\mathbf{I}-\mathbf{P}\}\hat{f}(t,k)|^2_D dt
\\
+C_\eta \int^T_0 |\hat{c}(t,k)|^2dt
\\
\lesssim \eta \int^T_0 |\hat{a}(t,k)|^2dt +C_\eta \int^T_0 |\hat{c}(t,k)|^2dt + C_\eta \int^T_0 | \{\mathbf{I}-\mathbf{P}\}\hat{f}(t,k)|^2_D dt.
\end{multline*}
Here in the second line we again have used the orthogonality in the $v$-integration for the estimate of
$$
\int^T_0 \vert (\mathbf{P}\hat{f}, \partial_t \hat{\Phi}^{J,m}_b)\vert dt.
$$
One can then deduce that
\begin{align}
\int^T_0 |\hat{b}(t,k)|^2 dt &\lesssim \Vert \hat{f}(k,T)\Vert_{L^2_v}^2 + \Vert \hat{f}_0(k)\Vert_{L^2_v}^2+ \eta \int^T_0 |\hat{a}(t,k)|^2 dt \notag \\
&\qquad +C_\eta \int^T_0 |\hat{c}(t,k)|^2dt+C_\eta \int^T_0 | \{\mathbf{I}-\mathbf{P}\}\hat{f}(t,k)|_D^2dt \notag\\
&\qquad+C_\eta \int^T_0 \vert (\vert \hat{H}(t,k) \vert, \mu^{\frac{1}{4}})\vert^2 dt,\label{torus.b}
\end{align}
where $\eta>0$ can be arbitrarily small.

Lastly, we consider $a$. We set
\begin{align*}
\hat{\Phi}=\hat{\Phi}_a = (\vert v\vert^2 -10) \{ v\cdot ik\hat{\phi}_a(t,k)\}\mu^{\frac{1}{2}},
\end{align*}
where $\hat{\phi}_a$ is a solution to
\begin{align*}
\vert k \vert^2 \hat{\phi}_a =\hat{a}.
\end{align*}
Note again that since $\hat{a}(t,0)\equiv 0$, it is valid to formally write $\hat{\phi}_a(t,k)=\hat{a}(t,k)/|k|^2$ for any $k\in \Z^3$ with the understanding that we define $\hat{\phi}_a(t,0) =0$.
For this choice the inner product of the macroscopic part is recast as
\begin{align*}
&-\int^T_0 (\mathbf{P}\hat{f}, v \cdot ik\hat{\Phi}_a)dt
\\
&=-\sum_{j,n}\int^T_0 (\{ \hat{a}+\hat{b}\cdot v +\frac{1}{2}(\vert v\vert^2-3)\hat{c}\} \mu^{\frac{1}{2}}, v_j v_n (\vert v\vert^2-10)\mu^{\frac{1}{2}}(-k_jk_n\hat{\phi_a})dt\\
&=5 \int^T_0 (\hat{a}, \vert k\vert^2 \hat{\phi_a}) dt=5 \int^T_0 |\hat{a}(t,k)|^2 dt.
\end{align*}
By the first equation of \eqref{mac.law.torus}, one has
\begin{align*}
\vert \partial_t \hat{a}(t,k)\vert \lesssim \vert k\vert |\hat{b}(t,k)|,
\end{align*}
and then we have
\begin{align*}
\vert S_2\vert&\le \int^T_0 \vert (\{\mathbf{I}-\mathbf{P}\}\hat{f}, \partial_t \hat{\Phi}_a)\vert dt +\int^T_0 \vert (\mathbf{P}\hat{f}, \partial_t \hat{\Phi}_a)\vert dt\\
&\lesssim \int^T_0 \Vert \partial_t \hat{\Phi}_a(t,k)\Vert_{L^2_v}^2 dt +\int^T_0 | \{\mathbf{I}-\mathbf{P}\} \hat{f}(t,k) |^2_D dt +\int^T_0 |\hat{b}(t,k)|^2dt\\
&\lesssim \int^T_0 \vert \partial_t k \hat{\phi}_a (t,k)\vert^2 dt + \int^T_0 | \{\mathbf{I}-\mathbf{P}\}\hat{f}(t,k)|^2_D dt + \int^T_0 |\hat{b}(t,k)|^2dt\\
&\lesssim \int^T_0 |\hat{b}(t,k)|^2dt + \int^T_0 | \{\mathbf{I}-\mathbf{P}\}\hat{f}(t,k)|^2_D dt.
\end{align*}
Therefore similar to the previous cases we obtain
\begin{align}\label{torus.a}
\int^T_0 |\hat{a}(t,k)|^2 dt &\lesssim \Vert \hat{f}(k,T)\Vert_{L^2_v}^2 + \Vert \hat{f}_0(k)\Vert_{L^2_v}^2+ \int^T_0 |\hat{b}(t,k)|^2 dt \notag\\
&\quad+\int^T_0 | \{\mathbf{I}-\mathbf{P}\}\hat{f}(t,k)|_D^2dt +\int^T_0 \vert (\hat{H}(t,k), \mu^{\frac{1}{4}})\vert^2 dt.
\end{align}
Combining those estimates \eqref{torus.c}, \eqref{torus.b}, and \eqref{torus.a} yields the desired estimate \eqref{abcest.torus1}. This completes the proof of Lemma \ref{abcest.torus}.
\end{proof}

\subsection{Anisotropic case with boundary}
In this section, we derive the important estimates for the macroscopic part of the solutions to the problem \eqref{LLeq},
\eqref{idf}, \eqref{ifb} and \eqref{srb}. Recall the macro-micro decomposition \eqref{abcdef}.
Given a function $H=H(h(x,v))\in \CN^\perp(L)$,
we then consider the following linear problem
\begin{equation}\label{Leqh}
\pa_tf+v_1\pa_{x_1}f+\bar{v}\cdot\na_{\bar{x}}f+Lf=H,\ t>0,\ (x_1,\bar{x})\in\Om, \ v=(v_1,\bar{v})\in\R^3,
\end{equation}
with the same initial data \eqref{idf} and the boundary condition \eqref{ifb} or \eqref{srb}.
As in the torus case, the $H^1$-estimates of the coefficients
$[a,b,c](t,x_1,\bar{x})$ will be obtained by using the dual argument as in \cite{EGKM-13}. We would like to emphasize that the computations in our situation are  more complicated than those in \cite{EGKM-13},
since the function space that we are using here is anisotropic. We shall give the full details of the proof for completeness.
Before doing this, we first define the boundary integral functionals:
\begin{multline}
\label{def.ga+}
\vert \Upsilon_{T,w}^+(h)\vert^2
=
\int_0^T\int_{v_1> 0}|v_1|w^2_{q,\vth}|h(t,1,v)|^2dvdt
\\
+\int_0^T\int_{v_1<0}|v_1|w^2_{q,\vth}|h(t,-1,v)|^2dv
dt,
\end{multline}
for particles with outgoing velocities on boundaries $x_1=\pm1$, and
\begin{multline}
\label{def.ga-}
\vert \Upsilon_{T,w}^-(h) \vert^2 =\int_0^T\int_{v_1< 0}|v_1|w^2_{q,\vth}|h(t,1,v)|^2dvdt
\\
+\int_0^T\int_{v_1>0}|v_1|w^2_{q,\vth}|h(t,-1,v)|^2dv
dt,
\end{multline}
for particles with incoming velocities on boundaries $x_1=\pm1$,  where the argument $h=h(t,x_1,v)$ is a distribution function that is well-defined on the boundaries.
For the shorthand notation, we set $\Upsilon_T^\pm(\cdot):=\Upsilon_{T,w}^\pm(\cdot)$ in the case when $w\equiv 1$, i.e.~when there is no velocity weight.

\begin{theorem}\label{abc.th}
Assume all the conditions listed in Theorems \ref{mthif} and \ref{mthsr} are valid and let $H(h(-v_1))=H(h)(-v_1)$. For $|\al|=0$ or $1$,
it holds that
\begin{align}
&\|\pa^{\al}\left[a,b,c\right]\|_{L^1_{\bar{k}}L^2_TL^2_{x_1,v}}\notag\\
&\lesssim\sum\limits_{|\al|\leq1}\left\|\{\FI-\FP\}\pa^{\al} f\right\|_{L^1_{\bar{k}}L^2_TL^2_{x_1}L^2_{v,D}}
+\sum\limits_{|\al|\leq1}\|\pa^{\al}f\|_{L^1_{\bar{k}}L^\infty_TL^2_{x_1,v}}
\notag \\&\quad+\sum\limits_{|\al|\leq1}\|\pa^{\al}f_0\|_{L^1_{\bar{k}}L^2_{x_1,v}}+E(\widehat{g_\pm})
+\sum\limits_{|\al|\leq1}\int_{\Z^2}|\Upsilon_{T}^+(\widehat{\pa^\al f})|d\Si(\bar{k})
\notag\\
&\quad+\sum\limits_{|\al|\leq1}\int_{\Z^2}\left(\int_0^T\left\|(\widehat{\pa^\al H},\mu^{\frac{1}{4}})_{L_v^2}\right\|^2dt\right)^{1/2}d\Si(\bar{k})
\notag \\
&\quad+\int_{\Z^2}|\Upsilon_{T}^-(\frac{\widehat{H}}{|v_1|})|d\Si(\bar{k}),\label{abcif.es}
\end{align}
for the inflow boundary condition \eqref{ifb},
and
\begin{align}
\|\pa^{\al}\left[a,b,c\right]\|_{L^1_{\bar{k}}L^2_TL^2_{x_1,v}}
&\lesssim \sum\limits_{|\al|\leq1}\left\|\{\FI-\FP\}\pa^\al f\right\|_{L^1_{\bar{k}}L^2_TL^2_{x_1}L^2_{v,D}}
\notag
\\
& \qquad
+\sum\limits_{|\al|\leq1}\|\pa^\al f\|_{L^1_{\bar{k}}L^\infty_TL^2_{x_1,v}}+\sum\limits_{|\al|\leq1}\|\pa^\al f_0\|_{L^1_{\bar{k}}L^2_{x_1,v}}
\notag \\
&\qquad+\sum\limits_{|\al|\leq1}\int_{\Z^2}\left(\int_0^T\left\|(\widehat{\pa^\al H},\mu^{\frac{1}{4}})_{L_v^2}\right\|^2dt\right)^{1/2}d\Si(\bar{k}),\label{abcsr.es}
\end{align}
for the specular reflection boundary condition \eqref{srb}.
\end{theorem}

\begin{proof}
In what follows  we only show that \eqref{abcif.es} and \eqref{abcsr.es} are valid for $\pa^{\al}=\pa_{x_1}$ with $\al=(1,0,0)$, since
the proof in other cases is quite similar.

Acting $\pa:=\pa_{x_1}$ to \eqref{Leqh}
and taking the Fourier transform of the resulting equations with respect to $\bar{x}$, we have
\begin{equation}\label{FLeqh}
\pa_t\widehat{\pa f}+v_1\pa_{x_1}\widehat{\pa f}+i\bar{k}\cdot\bar{v}\widehat{\pa f}+L\widehat{\pa f}=\widehat{\pa H},
\end{equation}
with the corresponding initial data
\begin{equation}\label{Fidf}
\hat{f}(0,x_1,\bar{k},v)=\hat{f}_0(x_1,\bar{k},v),
\end{equation}
and the inflow boundary condition
\begin{equation}\label{Fifb}
\hat{f}(t,-1,\bar{k},v)|_{v_1>0}=\widehat{g_-}(t,\bar{k},v),\ \ \hat{f}(t,1,\bar{k},v)|_{v_1<0}=\widehat{g_+}(t,\bar{k},v),
\end{equation}
or the specular reflection condition
\begin{eqnarray}\label{Fsrb}
\hat{f}(t,-1,\bar{k},v_1,\bar{v})|_{v_1>0}&=\hat{f}(-1,\bar{k},-v_1,\bar{v}),
\\
\hat{f}(t,1,v_1,\bar{v})|_{v_1<0}&=\hat{f}(1,x_1,\bar{k},-v_1,\bar{v}).
\notag
\end{eqnarray}
We note for clarity that in the proof of Theorem \ref{abc.th} the Fourier transform is always taken with respect to only $\bar{x}$.

\begin{remark}\label{property.remark}
Since $H(h(-v_1))=H(h)(-v_1)$, it is easy to see that if $f(t,x_1,\bar{x},v_1,\bar{v})$ is a solution of \eqref{FLeqh}, \eqref{Fidf}
and \eqref{Fsrb} with $f_0(x_1,v_1)=f(-x_1,-v_1)$, then
$$
f(t,x_1,\bar{x},v_1,\bar{v})=f(t,-x_1,\bar{x},-v_1,\bar{v}).
$$
We also note that if $H=\Ga(h,h)$, then the property proved in \cite[Lemma 3.1, page 637]{Guo-06-DL} guarantees that $H(h(-v_1))=H(h)(-v_1)$.

\end{remark}

Similar to obtaining \eqref{mac.law.torus}, one has the following fluid-type system for the coefficient functions $[a,b,c]$:
\begin{equation}\label{mac.law}
    \left\{\begin{array}{l}
      \dis \pa_t a +\pa_{x_1}b_1+\pa_{\bar{x}}\cdot \bar{b}=0,\ \bar{b}=(b_2,b_3),\\[3mm]
      \dis \pa_t b +\na_x (a+2c)+\na_x\cdot \highG (\{\FI-\FP\} f)=0,\\[3mm]
      \dis \pa_t c +\frac{1}{3}\na_x\cdot b +\frac{1}{6}\na_x\cdot
      \Lambda (\{\FI-\FP\} f)=0,\\[3mm]
      \dis \pa_t[\highG_{{ jm}}(\{\FI-\FP\} f)+2c\de_{{ jm}}]+\pa_jb_m+\pa_m
      b_j=\highG_{jm}(\mathbbm{r}+\mathbbm{h}),\\[3mm]
      \dis \pa_t \Lambda_j(\{\FI-\FP\} f)+\pa_j c = \Lambda_j(\mathbbm{r}+\mathbbm{h}),
    \end{array}\right.
\end{equation}
where the terms $\mathbbm{r}$ and $\mathbbm{h}$ are given by
\begin{eqnarray}\notag
\mathbbm{r}= -v_1 \pa_{x_1} \{\FI-\FP\}f-\bar{v}\cdot \na_{\bar{x}} \{\FI-\FP\}f,\ \ \mathbbm{h}=-L \{\FI-\FP\}f+H.
\end{eqnarray}
We also recall the definitions of $\Lambda_j$ and $\highG_{jm}$ below \eqref{mac.law.torus}.

Let $\hat{\Phi}(t,x_1,\bar{k},v)\in C^1((0,+\infty)\times(-1,1)\times\R^3)$ with $\bar{k}=(k_2,k_3)\in \Z^2$ be a test function.
Taking the inner product of $\hat{\Phi}(t,x_1,\bar{k},v)$ and \eqref{FLeqh} with respect to $(x_1,v)$
and integrating the resulting identity with respect to $t$ over $[0,T]$ for any $T>0$, we obtain
\begin{equation*}
\begin{split}
(\widehat{\pa f},&\hat{\Phi})(T)-(\widehat{\pa f},\hat{\Phi})(0)-\int_0^T(\widehat{\pa f},\pa_t\hat{\Phi})dt
-\int_0^T(\widehat{\pa f},v\cdot\widehat{\na_{x_1,\bar{x}}\Phi})dt
\\&
+\int_0^T\langle v_1\widehat{\pa f}(1),\hat{\Phi}(1)\rangle dt-\int_0^T\langle v_1\widehat{\pa f}(-1),\hat{\Phi}(-1)\rangle dt
+\int_0^T(L\widehat{\pa f},\hat{\Phi})dt
\\&
\qquad
=\int_0^T(\widehat{\pa H},\hat{\Phi})dt.
\end{split}
\end{equation*}
In the rest of this proof the inner product $(\cdot, \cdot)$ is always with respect to $L^2_{x_1,v}$, and the inner product $\langle \cdot, \cdot \rangle$ is always with respect to $L^2_{v}$.
The above identity together with $\hat{f}=\FP\hat{f}+\{\FI-\FP\}\hat{f}$  implies that
\begin{equation}\label{Leqh.abc}
\begin{split}
-\int_0^T(\widehat{\pa \FP f},v\cdot\widehat{\na_{x_1,\bar{x}}\Phi})dt
=\sum\limits^5_{j=1}S_j,
\end{split}
\end{equation}
where $S_j$ $(1\leq j\leq 5)$ are defined by
\begin{eqnarray*}
\left\{\begin{array}{rll}
\begin{split}
S_1=&(\widehat{\pa f},\hat{\Phi})(0)-(\widehat{\pa f},\hat{\Phi})(T),\
S_2=\int_0^T(\widehat{\pa f},\pa_t\hat{\Phi})dt,\\
S_3=&\int_0^T( \{\FI-\FP\} \widehat{\pa f},v\cdot\widehat{\na_{x_1,\bar{x}}\Phi})dt,\\
S_4=&-\int_0^T(L\widehat{\pa f},\hat{\Phi})dt+\int_0^T(\widehat{\pa H},\hat{\Phi})dt,\\
S_5=&-\int_0^T\langle v_1\widehat{\pa f}(1),\hat{\Phi}(1)\rangle dt+\int_0^T\langle v_1\widehat{\pa f}(-1),\hat{\Phi}(-1)\rangle dt.
\end{split}
\end{array}\right.
\end{eqnarray*}

\noindent\underline{\it{Estimates on} $\hat{c}(t,x_1,\bar{k})$:}
We choose the following test function
$$
\hat{\Phi}=\hat{\Phi}_c=(\vert v \vert^2-5)\left\{v\cdot\widehat{\na_{x_1,\bar{x}}\phi}_c(t,x_1,\bar{k})\right\}\mu^{\frac{1}{2}},
$$
where
\begin{equation}\label{ep.c}
-\pa_{x_1}^2\hat{\phi}_c+\vert \bar{k} \vert^2\hat{\phi}_c(\bar{k})=\widehat{\pa c}(\bar{k}),\quad \hat{\phi}_c(\pm1,\bar{k})=0.
\end{equation}
It is straightforward to see that
\begin{equation}\label{ep.es}
\|\hat{\phi}_c\|_{H_{x_1}^2}+|\bar{k}|\|\hat{\phi}_c\|\lesssim  \|\widehat{\pa c}\|.
\end{equation}
Here and in the rest of the proof of Theorem \ref{abc.th}, for brevity we shall use $\|\cdot\|=\|\cdot\|_{L^2_{x_1}}$ to denote the $L^2$ norm in $x_1$ only if the function inside the norm only depends on $x_1$.  If the function inside the norm depends on both $x_1$ and $v$ then in the rest of this proof we shall also use $\|\cdot\|=\|\cdot\|_{L^2_{x_1,v}}$ without any risk of ambiguity.   Using the definition \eqref{abcdef}, we have
$$
\FP\hat{f}=\{\hat{a}+\hat{b}\cdot v+\frac{1}{2}(|v|^2-3)\hat{c}\}\mu^{\frac{1}{2}}.
$$
Then one has
\begin{equation*}
\begin{split}
&-\int_0^T(\widehat{\pa \FP f},v\cdot\widehat{\na_{x_1,\bar{x}}\Phi})dt\\
&=-\sum\limits_{j}\int_0^T\left(\left\{\widehat{\pa a}+\widehat{\pa b}\cdot v
+\frac{1}{2}(|v|^2-3)\widehat{\pa c}\right\}\mu^{\frac{1}{2}},v_j\widehat{\pa_j\Phi}_c\right)dt\\
&=-\sum\limits_{j,n}\int_0^T\left(\left\{\widehat{\pa a}+\widehat{\pa b}\cdot v
+\frac{1}{2}(|v|^2-3)\widehat{\pa c}\right\}\mu^{\frac{1}{2}},v_jv_n(|v|^2-5)\mu^{\frac{1}{2}}\widehat{\pa_j\pa_n\phi}_c
\right)dt\\
&=\sum\limits_{j}\int_0^T(\widehat{\pa c},-\widehat{\pa_j^2\phi}_c)dt=\int_0^T\|\widehat{\pa c}(\bar{k})\|^2dt.
\end{split}
\end{equation*}
We now turn to estimate the $S_j$ $(1\leq j\leq5)$ term by term.
By H\"{o}lder's  inequality and the elliptic estimate \eqref{ep.es}, it follows that
$$
|S_1|\lesssim \| \widehat{\pa f}(T)\|^2+\|\widehat{\pa f}_0\|^2.
$$
Above we also used that $ \| \widehat{\pa c}\|  \lesssim\| \widehat{\pa f}\|$.
For the delicate term $S_2$, we first get from the third equation of \eqref{mac.law} that
\begin{multline}\label{patPhi1}
\|\pa_t\pa_{x_1}\hat{c}(\bar{k})\|_{H_{x_1}^{-1}}\\\lesssim \left\|\pa_{x_1}\hat{b}_1(\bar{k})\right\|
+|\bar{k}|\left\|\widehat{(b_2,b_3)}(\bar{k})\right\|+\left\|\pa_{x_1}\{\FI-\FP\}\hat{f}\right\|_{D}
+|\bar{k}|\left\|\{\FI-\FP\}\hat{f}\right\|_{D},
\end{multline}
and for $k\neq0$,
\begin{multline}\label{patPhi2}
|\bar{k}|^{-1}\|\pa_t\widehat{\pa_{\bar{x}} c}(\bar{k})\|\\\lesssim
\left\|\pa_{x_1}\hat{b}_1(\bar{k})\right\|+|\bar{k}|\left\|\widehat{(b_2,b_3)}(\bar{k})\right\|+\left\|\pa_{x_1}\{\FI-\FP\}\hat{f}\right\|_{D}
+|\bar{k}|\left\|\{\FI-\FP\}\hat{f}\right\|_{D}.
\end{multline}
On the other hand, it follows from \eqref{ep.c} that 
\begin{equation*}
\|\pa_t\hat{\phi}_c\|_{H_{x_1}^{1}}\lesssim \|\pa_t \widehat{\pa c}\|_{H_{x_1}^{-1}},
\ \ |\bar{k}|\|\pa_t\hat{\phi}_c\|^2\lesssim |\bar{k}|^{-1}\|\pa_t \widehat{\pa c}\|^2, \ k\neq0.
\end{equation*}
Then we compute
\begin{equation*}
\begin{split}
|S_2|&\leq \int_0^T\left|(\widehat{\pa f},\pa_t\hat{\Phi})\right|dt
=\int_0^T\left|(\{\FI-\FP\}\widehat{\pa f},\pa_t\hat{\Phi})\right|dt\\
&\lesssim \eta\int_0^T\|\pa_t\hat{\Phi}\|^2dt+C_\eta\int_0^T\left\|\{\FI-\FP\}\widehat{\pa f}\right\|_{D}^2dt\\
&\lesssim \eta\int_0^T\|\pa_t\pa_{x_1}\hat{\phi}_c(t,\bar{k})\|^2dt
+
\eta|\bar{k}|^2\int_0^T\|\pa_t\hat{\phi}_c(t,\bar{k})\|^2dt\\
&\qquad+C_\eta\int_0^T\left\|\{\FI-\FP\}\widehat{\na_x f}\right\|_D^2dt
\\
&\lesssim \eta\int_0^T\left\|\widehat{\na_x b}(\bar{k})\right\|^2dt
+C_\eta\int_0^T\left\|\{\FI-\FP\}\widehat{\na_x f}\right\|_{D}^2dt.
\end{split}
\end{equation*}
Thanks to \eqref{ep.es} and Cauchy-Schwarz's inequality, 
$S_3$ is bounded as follows:
\begin{equation*}
\begin{split}
|S_3|
\lesssim\eta\int_0^T\left\|\widehat{\pa c}(\bar{k})\right\|^2dt
+C_\eta\int_0^T\left\|\{\FI-\FP\}\widehat{\na_x f}\right\|_{D}^2dt.
\end{split}
\end{equation*}
For the  term $S_4$, applying   the elliptic estimate \eqref{ep.es} yields
\begin{equation*}
\begin{split}
|S_4|\leq \eta\int_0^T\left\|\widehat{\pa c}(\bar{k})\right\|^2dt
+C_\eta\int_0^T\left\|\{\FI-\FP\}\widehat{\pa f}\right\|_{D}^2dt+C_\eta\int_0^T\left\|(\widehat{\pa H},\mu^{\frac{1}{4}})_{L^2_v}\right\|^2dt.
\end{split}
\end{equation*}
We further recall the explanation of the torus case of the estimate $S_4$ above \eqref{torus.c}.
In particular we also use the estimate \cite[Equation (6.12) on page 819]{GS}, then the above estimate follows after Cauchy-Schwarz's inequality.

For the boundary term $S_5$, we first consider the inflow boundary condition \eqref{Fifb},
noticing that
$$
\hat{\Phi}_c(\pm1)=(|v|^2-5)v_1\pa_{x_1}\hat{\phi}_c(t,\pm1,\bar{k})\mu^{\frac{1}{2}},
$$
and by the trace theory
$$
|\pa_{x_1}\hat{\phi}_c(t,\pm1,\bar{k})|\lesssim \|\hat{\phi}_c(t,\bar{k})\|_{H_{x_1}^2}\lesssim  \|\widehat{\pa c}\|.
$$
Then one has by using Cauchy-Schwarz's inequality that
\begin{align}
|S_5|
&\lesssim\eta\int_0^T\left\|\widehat{\pa c}(\bar{k})\right\|^2dt
+C_\eta\int_0^T\int_{\R^3}|v_1||\widehat{\pa f}(1)|^2dvdt
 \notag \\
&\qquad
+C_\eta\int_0^T\int_{\R^3}|v_1||\widehat{\pa f}(-1)|^2dvdt
\notag\\
&\lesssim\eta\int_0^T\left\|\widehat{\pa c}(\bar{k})\right\|^2dt
+C_\eta\int_0^T\int_{v_1>0}|v_1||\widehat{\pa f}(1)|^2dvdt
\notag \\
&\qquad +C_\eta\int_0^T\int_{v_1<0}|v_1||\widehat{\pa f}(1)|^2dvdt
+C_\eta\int_0^T\int_{v_1>0}|v_1||\widehat{\pa f}(-1)|^2dvdt\notag\\
&\qquad+C_\eta\int_0^T\int_{v_1<0}|v_1||\widehat{\pa f}(-1)|^2dvdt. \label{ifbc}
\end{align}
Next, in view of \eqref{Fifb},
we get from the equation \eqref{FLeqh} that
\begin{align}
&\int_0^T\int_{v_1<0}|v_1||\widehat{\pa_{x_1} f}(1)|^2dvdt\notag\\
&\lesssim \int_0^T\int_{v_1<0}|v_1|^{-1}|\widehat{\pa_t g_+}|^2dvdt
+\int_0^T\int_{v_1<0}|v_1|^{-1}|\bar{k}\cdot\bar{v}|^2|\widehat{g_+}|^2dvdt\notag\\
&\qquad+\int_0^T\int_{v_1<0}|v_1|^{-1}|L\widehat{g_+}|^2dvdt+\int_0^T\int_{v_1<0}|v_1|^{-1}|\widehat{H(1)}|^2dvdt\notag\\
&\lesssim E_{\bar{k}}(\widehat{g_+})+\int_0^T\int_{v_1<0}|v_1|^{-1}|\widehat{H(1)}|^2dvdt.
\label{ifbc2}
\end{align}
Similarly, it holds that
\begin{equation}\label{ifbc3}
\begin{split}
\int_0^T\int_{v_1>0}|v_1||\widehat{\pa_{x_1} f}(-1)|^2dvdt\lesssim E_{\bar{k}}(\widehat{g_-})+\int_0^T\int_{v_1>0}|v_1|^{-1}|\widehat{H}(-1)|^2dvdt.
\end{split}
\end{equation}
It is straightforward to see that if $\pa_{x_1}$ is replaced by $\pa_{\bar{x}}$, then it follows that
\begin{equation*}
\begin{split}
\int_0^T\int_{v_1<0}|v_1||\widehat{\pa_{\bar{x}} f}(1)|^2dvdt=& \int_0^T\int_{v_1<0}|v_1||\widehat{\pa_{\bar{x}} g_+}|^2dvdt,\\
\int_0^T\int_{v_1>0}|v_1||\widehat{\pa_{\bar{x}} f}(-1)|^2dvdt=& \int_0^T\int_{v_1>0}|v_1||\widehat{\pa_{\bar{x}} g_-}|^2dvdt.
\end{split}
\end{equation*}
Therefore we only consider the case $\pa=\pa_{x_1}$ in the following.

As to the specular reflection boundary condition \eqref{Fsrb}, the  symmetric property plays a key role in dealing with the intractable boundary term $S_5$.
In fact, from $\hat{f}(-x_1,-v_1)=\hat{f}(x_1,v_1)$ for all $x_1\in[-1,1]$ and from \eqref{Fsrb} we have that
$\hat{f}(\pm1,v_1)|_{v_1\neq0}=\hat{f}(\pm1,-v_1)$ at the boundary.  Then we can show that $S_5$ vanishes at this stage.
Since
$$
\hat{c}(t,x_1,\bar{k})=\frac{1}{6}\int_{\R^3}\hat{f}(t,x_1,\bar{k},v)(|v|^2-3)\mu^{\frac{1}{2}}(v)dv,
$$
We then see that $\hat{c}(t,x_1,\bar{k})$ is  even with respect to $x_1$,
which further implies that $\hat{\phi}_c$ is odd w.r.t.\ $x_1$  according to \eqref{ep.c}.
Therefore one has
$$
\hat{\Phi}_c(t,-1,\bar{k},-v_1)=-\hat{\Phi}_c(t,1,\bar{k},v_1)
$$
by the definition of $\hat{\Phi}_c$, note that the fact that $\widehat{\pa_{\bar{x}}\phi}_c(\pm1)=0$ was also used here. Moreover,
by $\hat{f}(-1,v_1)=\hat{f}(1,-v_1)=\hat{f}(1,v_1)$ for $v_1\neq0$, and using the equation \eqref{FLeqh} to define the derivative $\pa_{x_1}$, one also has
\begin{equation*}
\widehat{\pa_{x_1}f}(-1,-v_1)=\widehat{\pa_{x_1}f}(-1,v_1)=\widehat{\pa_{x_1}f}(1,v_1),\ \textrm{for}\ v_1\neq0.
\end{equation*}
Consequently, by a change of variable $v_1\rightarrow  -v_1$, we have
\begin{equation*}
\begin{split}
S_5=&
\int_0^T\langle v_1\widehat{\pa_{x_1} f}(1),\hat{\Phi}_c(1)\rangle dt-\int_0^T\langle v_1\widehat{\pa_{x_1} f}(-1),\hat{\Phi}_c(-1)\rangle dt\\
=&\int_0^T\langle v_1\widehat{\pa_{x_1} f}(1,v_1),\hat{\Phi}_c(1)\rangle dt
+\int_0^T\langle v_1\widehat{\pa_{x_1} f}(-1,-v_1),\hat{\Phi}_c(-1,-v_1)\rangle dt\\
=&\int_0^T\langle v_1\widehat{\pa_{x_1} f}(1,v_1),\hat{\Phi}_c(1)\rangle dt
-\int_0^T\langle v_1\widehat{\pa_{x_1} f}(1,v_1),\hat{\Phi}_c(1,v_1)\rangle dt=0.
\end{split}
\end{equation*}
Combining the above estimates for the $S_j$ $(1\leq j\leq5)$ together, we now arrive at
\begin{align}
\int_0^T\|\widehat{\pa c}(\bar{k})\|^2dt&\lesssim \| \widehat{\pa f}(T)\|^2+\|\widehat{\pa f}(0)\|^2+E_{\bar{k}}(g_\pm)
+C_\eta\int_0^T\int_{v_1>0}|v_1||\widehat{\pa f}(1)|^2dvdt\notag \\
&\qquad+C_\eta\int_0^T\int_{v_1<0}|v_1||\widehat{\pa f}(-1)|^2dvdt
+\eta\int_0^T\left\|\widehat{\na_x b}(\bar{k})\right\|^2dt
\notag \\
&\qquad+C_\eta\int_0^T\left\|\{\FI-\FP\}\widehat{\na_x f}\right\|_{D}^2dt
+C_\eta\int_0^T\left\|(\widehat{\pa^\al H},\mu^{\frac{1}{4}})_{L^2_v}\right\|^2dt
\notag \\
&\qquad+\int_0^T\int_{v_1<0}|v_1|^{-1}|\widehat{H(1)}|^2dvdt
\notag \\
&\qquad
+\int_0^T\int_{v_1>0}|v_1|^{-1}|\widehat{H(-1)}|^2dvdt,\label{es.pacif}
\end{align}
for the inflow boundary condition \eqref{Fifb}, and
\begin{multline}
\int_0^T\|\widehat{\pa c}(\bar{k})\|^2dt\lesssim \| \widehat{\pa f}(T)\|^2+\|\widehat{\pa f}_0\|^2
+\eta\int_0^T\left\|\widehat{\na_x b}(\bar{k})\right\|^2dt
\\
+C_\eta\int_0^T\left\|\{\FI-\FP\}\widehat{\na_x f}\right\|_{D}^2dt
+C_\eta\int_0^T\left\|(\widehat{\pa H},\mu^{\frac{1}{4}})_{L^2_v}\right\|^2dt,\label{es.pacsr}
\end{multline}
for the specular reflection boundary condition \eqref{Fsrb}. We thus conclude the estimates on $c$.

\medskip
\noindent\underline{\it{Estimates on} $\hat{b}(t,x_1,\bar{k})$:}
In this case, we choose the following test function
\begin{eqnarray*}
\hat{\Phi}=\sum\limits_{m=1}^3\hat{\Phi}^{J,m}_b
,\ J=1,2,3,
\end{eqnarray*}
with
\begin{eqnarray*}
\hat{\Phi}^{J,m}_b
=\left\{\begin{array}{rll}
\begin{split}
&|v|^2v_mv_J\widehat{\pa_m\phi}_J(t,x_1,\bar{k})-\frac{7}{2}(v_m^2-1)\widehat{\pa_J\phi}_J(t,x_1,\bar{k})\mu^{\frac{1}{2}},\ \ J\neq m,\\
&\frac{7}{2}(v_J^2-1)\widehat{\pa_J\phi}_J(t,x_1,\bar{k})\mu^{\frac{1}{2}},\ \ J=m,
\end{split}
\end{array}\right.
\end{eqnarray*}
where
\begin{equation}\label{ep.b}
-\pa_{x_1}^2\hat{\phi}_J+|\bar{k}|^2\hat{\phi}_J(\bar{k})=\widehat{\pa b}_J(\bar{k}),\ {\rm and}\ \hat{\phi}_J(\pm1,\bar{k})=0.
\end{equation}
In this paper we use the notation $(\pa_1,\pa_2,\pa_3)\eqdef (\pa_{x_1},\pa_{\bar{x}})$.

Standard elliptic estimates yield the following estimate
\begin{equation*}
\|\hat{\phi}_J\|_{H_{x_1}^2}+|\bar{k}|\|\hat{\phi}_J\|\lesssim C \|\widehat{\pa b}_J\|.
\end{equation*}
With the above choices in hand, we now turn to estimate both sides of \eqref{Leqh.abc} term by term. Notice
\begin{equation*}
\begin{split}
-\sum\limits_{m=1}^3\int_0^T&\left(\widehat{\pa\FP f},v\cdot\widehat{\na \Phi}^{J,m}_b\right)dt\\
=&-\sum\limits_{m=1}^3\int_0^T\sum\limits_{j=1}^3\left(\left\{\widehat{\pa a}+\sum\limits_{n=1}^3\widehat{\pa b}_n v_n
+\frac{1}{2}(|v|^2-3)\widehat{\pa c}\right\}\mu^{\frac{1}{2}},v_j\widehat{\pa_j\Phi}^{J,m}_b\right)dt.
\end{split}
\end{equation*}
Then it follows that
\begin{equation*}
\begin{split}
&-\sum\limits_{m=1}^3\int_0^T\left(\widehat{\pa\FP f},v\cdot\widehat{\na \Phi}^{J,m}_b\right)dt\\
&=-\sum\limits_{m=1,m\neq J}^3\int_0^T\left(v_mv_J\mu^{\frac{1}{2}}\widehat{\pa b}_J
,|v|^2v_mv_J\mu^{\frac{1}{2}}\widehat{\pa^2_m\phi}_J(t,x_1,\bar{k})\right)dt\\
&\qquad-\sum\limits_{m=1,m\neq J}^3\int_0^T
\left(v_mv_J\mu^{\frac{1}{2}}\widehat{\pa b}_m,|v|^2v_mv_J\mu^{\frac{1}{2}}\widehat{\pa_J\pa_m\phi}_J(t,x_1,\bar{k})\right)dt\\
&\qquad+7\sum\limits_{m=1,m\neq J}^3\int_0^T\left(\widehat{\pa b}_m
,\widehat{\pa_m\pa_J\phi}_J(t,x_1,\bar{k})\right)dt-7\int_0^T\left(\widehat{\pa b}_J
,\widehat{\pa^2_J\phi}_J(t,x_1,\bar{k})\right)dt.
\end{split}
\end{equation*}
Therefore one has
\begin{equation*}
\begin{split}
&-\sum\limits_{m=1}^3\int_0^T\left(\widehat{\pa\FP f},v\cdot\widehat{\na \Psi}^{J,m}_b\right)dt\\
&=-7\sum\limits_{m=1}^3\int_0^T\left(\widehat{\pa b}_J
,\widehat{\pa^2_m\phi}_J(t,x_1,\bar{k})\right)dt=7\int_0^T\left\|\widehat{\pa b}_J
\right\|^2dt.
\end{split}
\end{equation*}
In what follows, we only estimate $S_2$ and $S_5$, since the other terms are similar to obtaining \eqref{es.pacif}.
By using the second equation of \eqref{mac.law}, as in \eqref{patPhi1} and \eqref{patPhi2}, we obtain
\begin{multline*}
\|\pa_t\pa_{x_1}\hat{b}(\bar{k})\|_{H_{x_1}^{-1}}\lesssim \left\|\pa_{x_1}(\hat{a}+\hat{c})(\bar{k})\right\|
+|\bar{k}|\left\|(\hat{a}+\hat{c})(\bar{k})\right\|
\\
+\left\|\pa_{x_1}\{\FI-\FP\}\hat{f}\right\|_{D}
+|\bar{k}|\left\|\{\FI-\FP\}\hat{f}\right\|_{D},
\end{multline*}
and for $k\neq0$, we have
\begin{multline*}
|\bar{k}|^{-1}\|\pa_t\widehat{\pa_{\bar{x}} b}(\bar{k})\|\lesssim
\left\|\pa_{x_1}(\hat{a}+\hat{c})(\bar{k})\right\|+|\bar{k}|\left\|(\hat{a}+\hat{c})(\bar{k})\right\|
\\
+\left\|\pa_{x_1}\{\FI-\FP\}\hat{f}\right\|_{D}
+|\bar{k}|\left\|\{\FI-\FP\}\hat{f}\right\|_{D}.
\end{multline*}
As a consequence, it follows
\begin{equation*}
\begin{split}
|S_2|&\leq \sum\limits_{m}\int_0^T\left|(\widehat{\pa f},\pa_t\hat{\Phi}_b^{J,m})\right|dt\\
&\leq\sum\limits_{m}\int_0^T\left|(\{\FI-\FP\}\widehat{\pa f},\pa_t\hat{\Phi}_b^{J,m})\right|dt
+\sum\limits_{m}\int_0^T\left|(\FP\widehat{\pa f},\pa_t\hat{\Phi}_b^{J,m})\right|dt\\
&\lesssim \eta\sum\limits_{m}\int_0^T\|\pa_t\hat{\Phi}_b^{J,m}\|^2dt+C_\eta\int_0^T\left\|\{\FI-\FP\}\widehat{\na_x f}\right\|_{D}^2dt
+C_\eta\int_0^T\left\|\widehat{\na_x c}(\bar{k})\right\|^2dt\\
&\lesssim
\eta\int_0^T\|\pa_t\pa_{x_1}\widehat{\phi}_J(t,\bar{k})\|^2dt
+
\eta\int_0^T\|\pa_tik\widehat{\phi}_J(t,\bar{k})\|^2dt
\\
&\qquad+C_\eta\int_0^T\left\|\{\FI-\FP\}\widehat{\na_x f}\right\|_D^2dt+C_\eta\int_0^T\left\|\widehat{\na_x c}(\bar{k})\right\|^2dt\\
&\lesssim \eta\int_0^T\left\|\widehat{\na_x a}(\bar{k})\right\|^2dt+C_\eta\int_0^T\left\|\widehat{\na_x c}(\bar{k})\right\|^2dt
+C_\eta\int_0^T\left\|\{\FI-\FP\}\widehat{\na_x f}\right\|_{D}^2dt,
\end{split}
\end{equation*}
where from  \eqref{ep.b} we used the following elliptic estimates. 
\begin{equation*}
\|\pa_t\hat{\phi}_J\|_{H_{x_1}^{1}}\lesssim \|\pa_t \widehat{\pa b}\|_{H_{x_1}^{-1}},
\ \ |\bar{k}|\|\pa_t\hat{\phi}_J\|^2\lesssim |\bar{k}|^{-1}\|\pa_t \widehat{\pa b}\|^2, \ k\neq0.
\end{equation*}
We also then used the estimates just above in the upper bounds of these estimates.
For the inflow boundary condition \eqref{Fifb}, $S_5$ enjoys the similar estimates
to \eqref{ifbc}, \eqref{ifbc2} and \eqref{ifbc3}. Notice that that only the case $\pa=\pa_{x_1}$ should be considered.  We now {\it claim} that $S_5$ also vanishes for the specular reflection boundary condition \eqref{Fsrb}.
From
$$
\hat{b}(t,x_1,\bar{k})=\int_{\R^3}\hat{f}(t,x_1,\bar{k},v)v\mu^{\frac{1}{2}}(v)dv
$$
and $\hat{f}(-x_1,-v_1)=\hat{f}(x_1,v_1)$,
we see that $\hat{b}_1(t,x_1,\bar{k})$ is odd w.r.t.\ $x_1$,
which further implies that $\hat{\phi}_1$ is even according to \eqref{ep.b}.
Hence it follows that
$$
\hat{\Phi}_b^{1,m}(t,-1,\bar{k},-v_1)=-\hat{\Phi}_b^{1,m}(t,1,\bar{k},v_1)
$$
by the definition of $\hat{\Phi}^{J,m}$.
Eventually, from
\begin{equation*}
\widehat{\pa_{x_1}f}(-1,-v_1)=\widehat{\pa_{x_1}f}(1,v_1),\ \textrm{for}\ v_1\neq0.
\end{equation*}
and further with the change of variable $v_1\rightarrow -v_1$, we obtain
\begin{equation*}
\begin{split}
S_5=&
\int_0^T\langle v_1\widehat{\pa_{x_1} f}(1),\hat{\Phi}_b^{1,m}(1)\rangle dt
-\int_0^T\langle v_1\widehat{\pa_{x_1} f}(-1),\hat{\Phi}_b^{1,m}(-1)\rangle dt\\
=&\int_0^T\langle v_1\widehat{\pa_{x_1} f}(1,v_1),\hat{\Phi}_b^{1,m}(1)\rangle dt
+\int_0^T\langle v_1\widehat{\pa_{x_1} f}(-1,-v_1),\hat{\Phi}_b^{1,m}(-1,-v_1)\rangle dt=0.
\end{split}
\end{equation*}
For $J=2$ and $3$, both $\hat{b}_2$ and $\hat{b}_3$ are even w.r.t.\ $x_1$,
which further yields that  $\hat{\phi}_2$ and $\hat{\phi}_3$ are odd w.r.t.\ $x_1$. Then we still have
$$
\hat{\Phi}_b^{J,m}(t,-1,\bar{k},-v_1)=-\hat{\Phi}_b^{J,m}(t,1,\bar{k},v_1),\quad J=2,3.
$$
Thus $S_5$ also
vanishes in the case of $\hat{\Phi}_b^{J,m}(t,1,\bar{k},v_1)$ with $J=2$ and $3$.
We now have the following conclusion as in  \eqref{es.pacif} and \eqref{es.pacsr}:
\begin{equation*}
\begin{split}
\int_0^T\|\widehat{\pa b}(\bar{k})\|^2dt&\lesssim \| \widehat{\pa f}(T)\|^2+\|\widehat{\pa f}_0\|^2+E_{\bar{k}}(g_\pm)(T)
+C_\eta\int_0^T\int_{v_1>0}|v_1||\widehat{\pa f}(1)|^2dvdt\\&\qquad+C_\eta\int_0^T\int_{v_1<0}|v_1||\widehat{\pa f}(-1)|^2dvdt
\\
&\qquad
+\eta\int_0^T\left\|\widehat{\na_x a}(\bar{k})\right\|^2dt+C_\eta\int_0^T\left\|\widehat{\na_x c}(\bar{k})\right\|^2dt
\\
&\qquad+C_\eta\int_0^T\left\|\{\FI-\FP\}\widehat{\na_x f}\right\|_{D}^2dt
+C_\eta\int_0^T\left\|(\widehat{\pa H},\mu^{\frac{1}{4}})_{L^2_v}\right\|^2dt
\\
&\qquad+\int_0^T\int_{v_1<0}|v_1|^{-1}|\widehat{H(1)}|^2dvdt+\int_0^T\int_{v_1>0}|v_1|^{-1}|\widehat{H(-1)}|^2dvdt,
\end{split}
\end{equation*}
for the inflow boundary condition \eqref{Fifb}, and
\begin{equation*}
\begin{split}
\int_0^T\|\widehat{\pa b}(\bar{k})\|^2dt&\lesssim \| \widehat{\pa f}(T)\|^2+\|\widehat{\pa f}(0)\|^2
+\eta\int_0^T\left\|\widehat{\na_x a}(\bar{k})\right\|^2dt
+
C_\eta\int_0^T\left\|\widehat{\na_x c}(\bar{k})\right\|^2dt
\\&\qquad+C_\eta\int_0^T\left\|\{\FI-\FP\}\widehat{\na_x f}\right\|_{D}^2dt+C_\eta\int_0^T\left\|(\widehat{\pa H},\mu^{\frac{1}{4}})_{L^2_v}\right\|^2dt,
\end{split}
\end{equation*}
for the specular reflection boundary condition \eqref{Fsrb}. This then concludes the estimates on $b$.

\medskip
\noindent\underline{\it{Estimates on} $\hat{a}(t,x_1,\bar{k})$:}
In this case, we choose the test function
$$
\hat{\Phi}=\hat{\Phi}_a=(|v|^2-10)\left\{v\cdot\widehat{\na_{x_1,\bar{x}}\phi}_a(t,x_1,\bar{k})\right\}\mu^{\frac{1}{2}},
$$
where
\begin{equation}\label{ep.a1}
-\pa_{x_1}^2\hat{\phi}_a+|\bar{k}|^2\hat{\phi}_a(\bar{k})=\widehat{\pa a}(\bar{k}),\ {\rm and}\ \hat{\phi}_a(\pm1,\bar{k})=0,
\end{equation}
for the inflow boundary condition \eqref{Fifb},
and
\begin{equation}\label{ep.a2}
-\pa_{x_1}^2\hat{\phi}_a+|\bar{k}|^2\hat{\phi}_a(\bar{k})=\widehat{\pa a}(\bar{k}),\ {\rm and}\ \pa_{x_1}\hat{\phi}_a(\pm1,\bar{k})=0,
\end{equation}
for the specular reflection boundary condition \eqref{Fsrb}.

We compute both sides of \eqref{Leqh.abc} with $\hat{\Phi}$ replaced by $\hat{\Phi}_a$.
For the left hand side, we have
\begin{equation*}
\begin{split}
&-\int_0^T(\widehat{\pa \FP f},v\cdot\widehat{\na_{x_1,\bar{x}}\Phi})dt\\
&=-\sum\limits_{j}\int_0^T\left(\left\{\widehat{\pa a}+\widehat{\pa b}\cdot v
+\frac{1}{2}(|v|^2-3)\widehat{\pa c}\right\}\mu^{\frac{1}{2}},v_j\widehat{\pa_j\Phi}_a\right)dt\\
&=-\sum\limits_{j,n}\int_0^T\left(\left\{\widehat{\pa a}+\widehat{\pa b}\cdot v
+\frac{1}{2}(|v|^2-3)\widehat{\pa c}\right\}\mu^{\frac{1}{2}},v_jv_n(|v|^2-10)\mu^{\frac{1}{2}}\widehat{\pa_j\pa_n\phi}_a
\right)dt\\
&=-5\sum\limits_{j}\int_0^T(\widehat{\pa a},-\widehat{\pa_j^2\phi}_c)dt=-5\int_0^T\|\widehat{\pa a}(\bar{k})\|^2dt.
\end{split}
\end{equation*}
As to the right hand side,
similar to the estimates for $\hat{b}(t,x_1,\bar{k})$, we only show the estimates for $S_2$ and $S_5$ as the others follow in the same way as done previously.
For this, we first have from \eqref{ep.a1} or
\eqref{ep.a2} that
\begin{equation*}
\|\pa_t\hat{\phi}_a\|_{H_{x_1}^{1}}\lesssim \|\pa_t \widehat{\pa a}\|_{H_{x_1}^{-1}},
\ \ |\bar{k}|\|\pa_t\hat{\phi}_c\|^2
\lesssim |\bar{k}|^{-1}\|\pa_t \widehat{\pa a}\|^2, \ k\neq0,
\end{equation*}
and 
\begin{equation*}
\|\hat{\phi}_a\|_{H_{x_1}^2}+|\bar{k}|\|\hat{\phi}_a\|
\lesssim  \|\widehat{\pa a}\|.
\end{equation*}
Moreover,
one gets from the first equation of \eqref{mac.law} that
\begin{equation*}
\|\pa_t\pa_{x_1}\hat{a}(\bar{k})\|_{H_{x_1}^{-1}}
\lesssim
\left\|\pa_{x_1}\hat{b}_1(\bar{k})\right\|
+
|\bar{k}|
\left\| \widehat{(b_2,b_3)}(\bar{k})\right\|,
\end{equation*}
and
\begin{equation*}
|\bar{k}|^{-1}\|\pa_t\widehat{\pa_{\bar{x}} a}(\bar{k})\|\lesssim
\left\|\pa_{x_1}\hat{b}_1(\bar{k})\right\|+|\bar{k}|\left\|\widehat{(b_2,b_3)}(\bar{k})\right\|,
\end{equation*}
provided $k\neq0$.

With the above estimates in hand,
we now compute
\begin{equation*}
\begin{split}
|S_2|&\leq \int_0^T\left|(\widehat{\pa f},\pa_t\hat{\Phi}_a)\right|dt
\leq\int_0^T\left|(\{\FI-\FP\}\widehat{\pa f},\pa_t\hat{\Phi}_a)\right|dt
+\int_0^T\left|(\FP\widehat{\pa f},\pa_t\hat{\Phi}_a)\right|dt\\
&\lesssim \int_0^T\|\pa_t\hat{\Phi}_a\|^2dt+\int_0^T\left\|\{\FI-\FP\}\widehat{\pa f}\right\|_{D}^2dt
+\int_0^T\left\|\widehat{\pa b}(\bar{k})\right\|^2dt\\
&\lesssim \int_0^T\|\pa_t\pa_{x_1}\hat{\phi}_a(t,\bar{k})\|^2dt+|\bar{k}|^2\int_0^T\|\pa_t\hat{\phi}_a(t,\bar{k})\|^2dt
+\int_0^T\left\|\{\FI-\FP\}\widehat{\na_x f}\right\|_D^2dt
\\
&\lesssim\int_0^T\left\|\widehat{\na_x b}(\bar{k})\right\|^2dt
+\int_0^T\left\|\{\FI-\FP\}\widehat{\na_x f}\right\|_{D}^2dt.
\end{split}
\end{equation*}
We now turn to estimate $S_5$, for the inflow boundary condition \eqref{Fifb}, $S_5$ shares the similar estimates as
\eqref{ifbc}, \eqref{ifbc2} and \eqref{ifbc3}. Noticing that $\hat{\phi}_a$ possesses the same symmetry as $\hat{\phi}_c$ w.r.t.\ $x_1$,
one can show that $S_5$ also vanishes for the specular reflection boundary condition \eqref{Fsrb} by the same arguments
as were given for the estimates on $\hat{c}(t,x_1,\bar{k})$ and $\hat{b}(t,x_1,\bar{k})$.
In a conclusion, the following estimates for  $\hat{a}(t,x_1,\bar{k})$ hold:
\begin{equation*}
\begin{split}
\int_0^T\|\widehat{\pa a}(\bar{k})\|^2dt&\lesssim \| \widehat{\pa f}(T)\|^2+\|\widehat{\pa f}_0\|^2+E_{\bar{k}}(g_\pm)
+\int_0^T\int_{v_1>0}|v_1||\widehat{\pa f}(1)|^2dvdt\\&\qquad+\int_0^T\int_{v_1<0}|v_1||\widehat{\pa f}(-1)|^2dvdt
+\int_0^T\left\|\widehat{\na_x b}(\bar{k})\right\|^2dt
\\&\qquad+\int_0^T\left\|\{\FI-\FP\}\widehat{\na_x f}\right\|_{D}^2dt+\int_0^T\left\|(\widehat{\pa H},\mu^{\frac{1}{4}})_{L^2_v}\right\|^2dt
\\&\qquad
+\int_0^T\int_{v_1<0}|v_1|^{-1}|\widehat{H(1)}|^2dvdt
\\&\qquad
+\int_0^T\int_{v_1>0}|v_1|^{-1}|\widehat{H(-1)}|^2dvdt,
\end{split}
\end{equation*}
for the inflow boundary condition \eqref{Fifb}, and
\begin{equation*}
\begin{split}
\int_0^T\|\widehat{\pa a}(\bar{k})\|^2dt&\lesssim \| \widehat{\pa f}(T)\|^2+\|\widehat{\pa f}_0\|^2
+\int_0^T\left\|\widehat{\na_x b}(\bar{k})\right\|^2dt
\\
&\qquad+\int_0^T\left\|\{\FI-\FP\}\widehat{\na_x f}\right\|_{D}^2dt+\int_0^T\left\|(\widehat{\pa H},\mu^{\frac{1}{4}})_{L^2_v}\right\|^2dt,
\end{split}
\end{equation*}
for the specular reflection boundary condition \eqref{Fsrb}. Finally, combining all the above estimates on $[a,b,c]$ together, we see that \eqref{abcif.es}
and \eqref{abcsr.es} hold true for  the case when $\pa=\pa_{x_1}$. The same estimate with the other derivatives also holds, and the proof is analogous but easier.
This completes the proof of Theorem \ref{abc.th}.
\end{proof}

\section{Proof of the main results in the torus}\label{sec6}

In this section, we shall obtain the global existence and large time behavior (stated in Theorem \ref{Torus Existence}).  We further obtain the propagation of regularity in the $x$ variable (stated in Theorem \ref{Torus Propagation}) for solutions to the initial
 value problem {\bf (PT)} \eqref{LLeq} and  \eqref{idf}
in the torus \eqref{d.pb}.  We shall prove  Theorem \ref{Torus Existence} and Theorem \ref{Torus Propagation}.

\begin{proof}[Proof of Theorem \ref{Torus Existence}]
We first consider the uniform a priori estimates without any velocity weight. As shown in Section \ref{sec3}, by applying the Fourier transform to \eqref{LLeq}, as in \eqref{torus fourier transformed} with \eqref{def.GaF} or \eqref{landau.NL.FT}, and then taking the complex inner product of the result with $\hat{f}$ in $L^2_v$, we have
\begin{align*}
(\partial_t \hat{f}, \hat{f})+ i(k\cdot v \hat{f},\hat{f})+(L\hat{f},\hat{f})=(\widehat{\Gamma(f,f)},\hat{f}).
\end{align*}
Taking the real part of this identity and integrating over $[0,T]$ for fixed $T>0$ gives
\begin{align*}
&\int_{\mathbb{Z}^3} \sup_{0\le t\le T} \Vert \hat{f}(t,k)\Vert_{L^2_v} d\Sigma(k) +\int_{\mathbb{Z}^3} \Big( \int^T_0 | \{\mathbf{I}-\mathbf{P}\} \hat{f}(t,k)|^2_D dt \Big)^{1/2} d\Sigma(k)\\
&\lesssim \int_{\mathbb{Z}^3} \Vert \hat{f}_0 (k) \Vert_{L^2_v} d\Sigma(k) +\int_{\mathbb{Z}^3} \Big( \int^T_0 \Big\vert \mathscr{R} (\widehat{\Gamma(f,f)}, \{\mathbf{I}-\mathbf{P}\} \hat{f}) \Big\vert dt\Big)^{1/2} d\Sigma(k)\\
&\lesssim \int_{\mathbb{Z}^3} \Vert \hat{f}_0(k) \Vert_{L^2_v} d\Sigma(k)+C_\eta \Vert f\Vert_{L^1_k L^\infty_T L^2_v}\Vert f\Vert_{L^1_k L^2_T L^2_{v,D}}+\eta \Vert \{\mathbf{I}-\mathbf{P}\}f\Vert_{L^1_k L^2_T L^2_{v,D}},
\end{align*}
for an arbitrary constant $\eta>0$.  Together with Theorem \ref{abcest.torus} and Lemma \ref{tril21}, the above estimate yields
\begin{align}\label{a priori torus}
\|f\|_{L^1_kL^\infty_T L^2_v}+\Vert f\Vert_{L^1_k L^2_T L^2_{v,D}}\lesssim \Vert f_0\Vert_{L^1_k L^2_v}+\|f\|_{L^1_kL^\infty_T L^2_v} \Vert f\Vert_{L^1_k L^2_T L^2_{v,D}}.
\end{align}
This is the main estimate that is needed for the {\it hard} potentials.

In case of {\it soft} potentials, in order to treat the time-decay of solutions, we need to make additional velocity weighted estimates with the weight $w_{q,\vth}$ given as in \eqref{def.w} under the assumption {\bf (H)} \eqref{q}. Indeed, with the help of Lemma \ref{wgesL}, one can also show that
\begin{multline}\label{a priori torus weighted}
\|w_{q,\vth}f\|_{L^1_kL^\infty_T L^2_v}+\Vert w_{q,\vth}f\Vert_{L^1_k L^2_T L^2_{v,D}}
\\
\lesssim \Vert w_{q,\vth}f_0\Vert_{L^1_k L^2_v}+\|w_{q,\vth}f\|_{L^1_kL^\infty_T L^2_v} \Vert w_{q,\vth}f\Vert_{L^1_k L^2_T L^2_{v,D}}.
\end{multline}
Note that the negative term $-C\Vert f\Vert_{L^1_k L^2_T L^2_{v,D}}$ appearing in the use of Lemma \ref{wgesL} can be handled in terms of \eqref{a priori torus}, since one has $w_{q,\vartheta}\ge 1$. For brevity we omit these details in verifying \eqref{a priori torus weighted}. Hence, under the smallness assumption on $\Vert w_{q,\vth}f_0\Vert_{L^1_k L^2_v}$, one can then obtain the closed estimate:
\begin{equation*}
\|w_{q,\vth}f\|_{L^1_kL^\infty_T L^2_v}+\Vert w_{q,\vth}f\Vert_{L^1_k L^2_T L^2_{v,D}}\lesssim \Vert w_{q,\vth}f_0\Vert_{L^1_k L^2_v}.
\end{equation*}
This completes the proof of the uniform a priori estimate \eqref{Torus Existence.uet} under the smallness assumption on $\Vert w_{q,\vth}f_0\Vert_{L^1_k L^2_v}$. Combining this with the local existence to be discussed in Section \ref{sec8}, and applying the standard continuity argument, we obtain the  global existence and uniqueness of global mild solutions. The positivity of solutions is also guaranteed by the local existence result as in Theorem \ref{le.th}.

Next, we consider the rate of convergence of the obtained solutions. We only treat the case of the non-cutoff Boltzmann equation since the same method can be applied to the Landau case. Moreover, we focus on the soft potentials $\ga+2s<0$, since it is similar to carry out the time-weighted estimates with the exponential weight $e^{\la t}$ for a suitably small constant $\la>0$ in case of hard potentials $\ga+2s\geq 0$.

Therefore, for the {\it soft} potentials, let
$$
\hat{h}=e^{\lambda t^p}\hat{f}
$$ with $\lambda>0$ and $0<p<1$ chosen later. As $f$ solves \eqref{LLeq}, then $\hat{h}$ satisfies
\begin{align*}
\partial_t \hat{h}+ik\cdot v\hat{h}+L\hat{h}=e^{-\lambda t^p}\widehat{\Gamma(h,h)}+\lambda p t^{p-1}\hat{h},
\end{align*}
with initial data
\begin{align*}
\hat{h}(0,k,v)=\hat{h}_0(k,v).
\end{align*}
With the aid of the arguments used to derive \eqref{a priori torus}, we have
\begin{multline}
\int_{\mathbb{Z}^3} \sup_{0\le t\le T} \Vert \hat{h}(t,k)\Vert_{L^2_v} d\Sigma(k) +
 \int_{\mathbb{Z}^3} \Big(\int^T_0 | \hat{h}(t,k)|_D^2 dt \Big)^{1/2} d\Sigma(k)\\
\lesssim \int_{\mathbb{Z}^3} \Vert \hat{f}_0(k) \Vert_{L^2_v} d\Sigma(k) + \sqrt{\lambda p}\int_{\mathbb{Z}^3} \Big( \int^T_0 t^{p-1} \Vert \hat{h}(t,k)\Vert_{L^2_v}^2 dt \Big)^{1/2} d\Sigma(k).\label{thmt.p1}
\end{multline}
For $\rho>0$ to be small enough later on and $p'>0$ to be chosen later depending upon $p$, we define a set
$$
\mathbf{E}=\{ \langle v\rangle \le  \rho t^{p'}\}
$$
and make the decomposition $1={\bf 1}_\mathbf{E}+{\bf 1}_{\mathbf{E}^c}$, cf.~\cite{SG-08-ARMA}. Then the second term on the right-hand side of \eqref{thmt.p1} can be bounded by
\begin{multline*}
 \sqrt{\lambda p} \int_{\Z^3} \Big( \int^T_0 t^{p-1} {\bf 1}_{\mathbf{E}} \vert \hat{h}\vert^2 dvdt \Big)^{1/2}d\Sigma(k)
  \\
  +\sqrt{\lambda p} \int_{\Z^3} \Big( \int^T_0 t^{p-1}{\bf 1}_{\mathbf{E}^c}  \vert \hat{h}\vert^2 dvdt \Big)^{1/2}d\Sigma(k)
 =:I_1+I_2.
\end{multline*}
We will define $p=(\gamma+2s) p'+1$ in the Boltzmann case (or $p=(\gamma+2) p'+1$ in the Landau case).
Here, then over $\mathbf{E}$, $I_1$ is controlled by
\begin{align*}
I_1
\le
\sqrt{\lambda p} \rho^{-(p-1)/p'} \int_{\Z^3} \Big( \int^T_0 \langle v\rangle^{(p-1)/p'} e^{2\lambda t^p} \vert \hat{f}\vert^2 dv dt \Big)^{1/2}d\Si (k).
\end{align*}
We choose $p$ as in \eqref{def.kap2} in the Boltzmann case (or \eqref{def.kap1} in the Landau case), such that $(p-1)/p'=\gamma+2s<0$ (or $(p-1)/p'=\gamma+2<0$ for the Landau case), we then further
obtain
\begin{align*}
I_1\le \sqrt{\lambda p} \rho^{-(p-1)/p'} \int_{\Z^3} \Big( \int^T_0 | \hat{h}(t,k) |_D^2 dt \Big)^{1/2} d\Sigma(k).
\end{align*}
As $\lambda>0$  and $\rho>0$ can be chosen arbitrarily small, the above term is absorbed into the left-hand side of \eqref{thmt.p1}.
For $I_2$, since it holds that $w_{q,\vartheta}^{-2} \le e^{-q\rho^\vartheta t^{p'\vartheta}/2}$ on $\mathbf{E}^c$, and we notice that we can take $p=p'\vartheta$ and $2\lambda < q\rho^\vartheta/2$ we have
\begin{align*}
I_2 &\le \sqrt{\lambda p} \int_{\Z^3} \Big(\int^T_0 \int_{\mathbf{E}^c} t^{p-1} e^{2\lambda t^p} e^{-\frac{q}{2}\rho^\vartheta t^{p'\vartheta}} w_{q,\vartheta}^2 \vert \hat{f}\vert^2 dvdt\Big)^{1/2} d\Sigma(k)\\
&\le \sqrt{\lambda p} \int_{\Z^3} \sup_{0\le t\le T}\Vert w_{q,\vartheta} \hat{f}(t,k)\Vert_{L^2_v} \Big(\int^T_0 t^{p-1} e^{2\lambda t^p} e^{-\frac{q}{2}\rho^\vartheta t^{p}} dt\Big)^{1/2} d\Sigma(k)\\
&\le C\sqrt{\lambda p} \int_{\Z^3} \sup_{0\le t\le T}\Vert w_{q,\vartheta} \hat{f}(t,k)\Vert_{L^2_v} d\Sigma(k),
\end{align*}
due to the finiteness of the $t$-integral, cf.~\cite{SG-08-ARMA}. By the existence result, it further holds that
\begin{equation*}
I_2\lesssim \int_{\mathbb{Z}^3} \Vert w_{q,\vartheta}\hat{f}_0(k) \Vert_{L^2_v} d\Sigma(k).
\end{equation*}
 Plugging the above estimates back to \eqref{thmt.p1} gives
\begin{multline*}
\int_{\mathbb{Z}^3} \sup_{0\le t\le T} \Vert \hat{h}(t,k)\Vert_{L^2_v} d\Sigma(k) + \int_{\mathbb{Z}^3} \Big(\int^T_0 | \hat{h}(t,k)|_D^2 dt \Big)^{1/2} d\Sigma(k)
\\
\lesssim \int_{\mathbb{Z}^3} \Vert w_{q,\vartheta}\hat{f}_0(k) \Vert_{L^2_v} d\Sigma(k).
\end{multline*}
We use Minkowski's inequality $\|  \| \cdot\|_{L^1_k}  \|_{L^\infty_T} \le \|  \| \cdot\|_{L^\infty_T}  \|_{L^1_k}$
and the expression $h=e^{\la t^p} f$ to obtain the time-decay estimate \eqref{Torus.Time.Decay.Est} with $\kappa=p$ in the soft potential case.
This then completes the proof of Theorem \ref{Torus Existence}.
\end{proof}

We will now give the proof of Theorem \ref{Torus Propagation}.

\begin{proof}[Proof of Theorem \ref{Torus Propagation}]
Following closely the proofs of Lemma \ref{lem.tei} and Theorem \ref{abcest.torus}, one can show that
\begin{align*}
&\int_{\mathbb{Z}^3}\Big( \int^T_0 \vert (\widehat{\Gamma(f,g)},\langle k\rangle^{2m} w^2_{q,\vartheta}\hat{h})\vert dt \Big)^{1/2} d\Sigma(k)\\
&\le
C_\eta \Vert w_{q,\vartheta} f\Vert_{L^1_{k,m}L^\infty_T L^2_v}\Vert w_{q,\vartheta} g\Vert_{L^1_{k,m}L^2_T L^2_{v,D}}
\\
&\qquad
+C_\eta
\Vert w_{q,\vartheta} f\Vert_{L^1_{k,m}L^2_T L^2_{v,D}}\Vert w_{q,\vartheta} g\Vert_{L^1_{k,m}L^\infty_T L^2_v}
+\eta \Vert w_{q,\vartheta} h\Vert_{L^1_{k,m}L^2_T L^2_{v,D}},
\end{align*}
and
\begin{align}\label{propagation macro}
\Vert [a,b,c]\Vert_{L^1_{k,m}L^2_T} &\lesssim \Vert \{\mathbf{I}-\mathbf{P}\}f\Vert_{L^1_{k,m} L^2_T L^2_{v,D}}+\Vert f\Vert_{L^1_{k,m}L^\infty_T L^2_v}+\Vert f_0\Vert_{L^1_{k,m}L^\infty_T L^2_v}\notag\\
&\qquad+\int_{\mathbb{Z}^3}\Big( \int^T_0 \vert (\langle k\rangle^m \vert \hat{H}(t,k)\vert,\mu^{\frac{1}{4}})_{L^2_v}\vert^2dt\Big)^{1/2} d\Sigma(k).
\end{align}
Taking the $L^2_v$ inner product of the Fourier transform of \eqref{LLeq} and $\langle k\rangle^{2m} w_{q,\vartheta}^2 \hat{f}$, we have
\begin{multline*}
(\partial_t \hat{f}, \langle k\rangle^{2m} w_{q,\vartheta}^2 \hat{f})+ i(k\cdot v \hat{f}, \langle k\rangle^{2m} w_{q,\vartheta}^2\hat{f})+(L\hat{f}, \langle  k\rangle^{2m} w_{q,\vartheta}^2\hat{f})
\\
=(\widehat{\Gamma(f,f)},\langle k\rangle^{2m} w_{q,\vartheta}^2\hat{f}).
\end{multline*}
Then by using the same argument that was used to derive \eqref{a priori torus} and \eqref{a priori torus weighted}, we have
\begin{align}
&\int_{\mathbb{Z}^3}\sup_{0\le t\le T} \Vert \langle k\rangle^m \hat{f}(t,k)\Vert_{L^2_v} d\Sigma(k) + 
\int_{\mathbb{Z}^3}\Big( \int^T_0 \Vert \langle k\rangle^m \{\mathbf{I}-\mathbf{P}\} \hat{f}(t,k)\Vert_D^2 dt\Big)^{1/2} \notag\\
& \lesssim \int_{\mathbb{Z}^3} \Vert \langle k\rangle^m \hat{f}_0 \Vert d\Sigma(k) + \int_{\mathbb{Z}^3} \Big(\int^T_0 \vert (\widehat{\Gamma(f,f)},\langle k\rangle^{2m} \{\mathbf{I}-\mathbf{P}\}\hat{f})\vert dt\Big)^{1/2} d\Sigma(k)
\notag \\
& \lesssim \int_{\mathbb{Z}^3} \Vert \langle k\rangle^m \hat{f}_0 \Vert_{L^2_v} d\Sigma(k) +C_\eta \Vert f \Vert_{L^1_{k,m}L^\infty_T L^2_v}\Vert f \Vert_{L^1_{k,m}L^2_T L^2_{v,D}}
\notag \\
& \qquad+\eta \Vert \{\mathbf{I}-\mathbf{P}\} f\Vert_{L^1_{k,m}L^2_T L^2_{v,D}},
\label{propagation nonweighted}
\end{align}
and we further deduce the velocity-weighted estimate
\begin{align}\label{propagation weighted}
&\int_{\mathbb{Z}^3} \sup_{0\le t\le T} \Vert \langle k\rangle^{m} w_{q,\vartheta} \hat{f}(t,k)\Vert_{L^2_v} d\Sigma(k) +
\int_{\mathbb{Z}^3} \Big( \int^T_0 \Vert \langle k\rangle^{m} w_{q,\vartheta} \hat{f}\Vert_D^2 dt\Big)^{1/2} d\Sigma(k)\notag \\
&\qquad-C \int_{\mathbb{Z}^3} \Big( \int^T_0 \Vert \langle  k\rangle^m \hat{f}\Vert_D^2 dt\Big)^{1/2} d\Sigma(k)\notag \\
&\lesssim \int_{\mathbb{Z}^3} \Vert \langle k\rangle^{m}  \hat{f}_0 \Vert_{L^2_v} d\Sigma(k) +\int_{\mathbb{Z}^3} \Big(\int^T_0 \vert \widehat{(\Gamma(f,f)}, \langle k\rangle^{2m} w_{q,\vartheta}^2 \hat{f})\vert dt \Big)^{1/2} d\Sigma(k)\notag \\
&\lesssim \int_{\mathbb{Z}^3} \Vert \langle k\rangle^{m} w_{q,\vartheta} \hat{f}_0 \Vert_{L^2_v} d\Sigma(k) +C_\eta \Vert w_{q,\vartheta}f\Vert_{L^1_{k,m} L^\infty_T L^2_v}\Vert w_{q,\vartheta}f\Vert_{L^1_{k,m} L^2_T L^2_{v,D}}\notag\\
&\qquad
+\eta \Vert w_{q,\vartheta}f\Vert_{L^1_{k,m} L^2_T L^2_{v,D}}.
\end{align}
A combination of the estimates \eqref{propagation macro},  \eqref{propagation nonweighted}, and \eqref{propagation weighted} gives
\begin{multline*}
\int_{\mathbb{Z}^3} \sup_{0\le t\le T} \Vert \langle k\rangle^m w_{q,\vartheta} \hat{f}(t,k)\Vert_{L^2_v} d\Sigma(k) +
\int_{\mathbb{Z}^3} \Big( \int^T_0 \Vert \langle k\rangle^m w_{q,\vartheta} \hat{f}\Vert^2_D dt\Big)^{1/2} d\Sigma(k)\\
 \lesssim \int_{\mathbb{Z}^3} \Vert \langle k\rangle^m w_{q,\vartheta} \hat{f}_0 \Vert_{L^2_v} d\Sigma(k),
\end{multline*}
which implies \eqref{Torus.Propagation.Est}. This completes the proof of Theorem \ref{Torus Propagation}.
\end{proof}

\section{Proof of the main results in the finite channel}\label{sec7}

In this section, we shall obtain the global existence, large time behavior and propagation of regularity in $\bar{x}$ variable for solutions to the initial
boundary value problem {\bf (PC)} \eqref{LLeq}, \eqref{idf},
\eqref{ifb} and \eqref{srb} in the case when the spatial domain is the finite channel.

\begin{proof}[Proof of Theorem \ref{mthif} and  Theorem \ref{mthsr}]
We divide the proof into three parts as follows. First of all, we start from the proof of global existence.  Then we explain the time decay rates.  After that we explain the positivity of a solution and uniqueness.

\medskip
\noindent\underline{\it Global existence.}
In this section we give a sequence of uniform {\it a priori} energy estimates assuming the smallness assumption and using the local existence as stated in Theorem \ref{le.th}. For the sake of brevity, in what follows we only present the proof of the {\it a priori}  energy estimates.
The computation is divided into two cases in terms of  the two prescribed boundary conditions.

\medskip
\noindent {\it Case 1: Inflow boundary condition.}
Let $|\al|\leq 1$, then we apply $\pa^{\al}$ to \eqref{LLeq} and then we take the Fourier transform $\hat{\cdot}=\CF_{\bar{x}}$ with respect to $\bar{x}=(x_2,x_3)$ to obtain
\begin{equation}\label{paLLeq}
\pa_t\widehat{\pa^{\al}f}+v_1\pa_{x_1}\widehat{\pa^{\al}f}+i\bar{k}\cdot\bar{v}\widehat{\pa^{\al}f}
+L\widehat{\pa^{\al}f}=\CF_{\bar{x}}\{\pa^{\al}\Ga(f,f)\},
\end{equation}
with initial data
\begin{equation*}
\widehat{\pa^{\al}f}(0,x_1,\bar{k},v)=\widehat{\pa^{\al}f}_0(x_1,\bar{k},v),
\end{equation*}
and the inflow boundary condition 
\begin{equation}\label{0ifb}
\hat{f}(t,-1,\bar{k},v)|_{v_1>0}=\widehat{g_-}(t,\bar{k},v),\ \ \hat{f}(t,1,\bar{k},v)|_{v_1<0}=\widehat{g_+}(t,\bar{k},v).
\end{equation}
Particularly, it follows from \eqref{0ifb} that if $|\al|=1$, then for $\pa^\al=\pa_{\bar{x}}$, one has
\begin{equation}\label{payifb}
\widehat{\pa_{\bar{x}}f}(t,-1,\bar{k},v)|_{v_1>0}=\widehat{\pa_{\bar{x}}g}_-(t,\bar{k},v),\ \
\widehat{\pa_{\bar{x}}f}(t,1,\bar{k},v)|_{v_1<0}=\widehat{\pa_{\bar{x}}g}_+(t,\bar{k},v),
\end{equation}
while for $\pa^\al=\pa_{x_1}$, one has by using  \eqref{paLLeq} with $\alpha =0$ and Remark \ref{property.remark} that
\begin{equation}\label{paxifb}
\begin{split}
\widehat{\pa_{x_1}f}(t,-1,\bar{k},v)\bigg|_{v_1>0}=&-\frac{1}{v_1}\left\{\widehat{\pa_tg}_-(t,\bar{k},v)
+\bar{v}\cdot\widehat{\pa_{\bar{x}}g}_-+L\widehat{g_-}-\CF_{\bar{x}}\{\Ga(g_-,g_-)\}\right\},\\
\widehat{\pa_{x_1}f}(t,1,\bar{k},v)\bigg|_{v_1<0}=&-\frac{1}{v_1}\left\{\widehat{\pa_tg}_+(t,\bar{k},v)
+\bar{v}\cdot\widehat{\pa_{\bar{x}}g}_++L\widehat{g_+}-\CF_{\bar{x}}\{\Ga(g_+,g_+)\}\right\}.
\end{split}
\end{equation}
In the rest of this section the complex inner product $( \cdot, \cdot ) = ( \cdot, \cdot )_{L^2_{x_1, v}}$ is in the space ${L^2_{x_1, v}}$.
Now taking the  inner product of \eqref{paLLeq} and $\widehat{\pa^{\al}f}$
with respect to $(x_1,v)$:
\begin{align}
(\pa_t\widehat{\pa^{\al}f},\widehat{\pa^{\al}f})+(v_1\pa_{x_1}\widehat{\pa^{\al}f},\widehat{\pa^{\al}f})
&+(i\bar{k}\cdot\bar{v}\widehat{\pa^{\al}f},\widehat{\pa^{\al}f})\notag \\
&+(L\widehat{\pa^{\al}f},\widehat{\pa^{\al}f})=\left(\CF_{\bar{x}}\{\pa^{\al}\Ga(f,f)\},\widehat{\pa^{\al}f}\right).
\label{paLLeqip}
\end{align}
In the rest of this proof we use the following brief notation  $\|\cdot\|=\|\cdot\|_{L^2_{x_1,v}}$ for the norm with $x_1\in I$ and $v\in \R^3$.
Then, by taking the real part of \eqref{paLLeqip} and integrating the resultant identity with respect to $t$ over $[0,T]$, we obtain that
\begin{multline}
\sup\limits_{0\le t \le T}\|\widehat{\pa^{\al}f}\|^2
+2\de_0\int_0^T\|\{\FI-\FP\}\widehat{\pa^{\al}f}\|^2_{D}dt+|\Upsilon_T^+ (\widehat{\pa^{\al}f})|^2\\
\leq 2\|\widehat{\pa^{\al}f}_0\|^2+|\Upsilon_T^- (\widehat{\pa^{\al}f})|^2
+2\int_0^T\left|\mathscr{R}\left(\CF_{\bar{x}}\{\pa^{\al}\Ga(f,f)\},\{\FI-\FP\}\widehat{\pa^{\al}f}\right)\right|dt,
\label{eng1}
\end{multline}
for any $\bar{k}\in \Z^2$, where $\Upsilon_T^+$ and $\Upsilon_T^-$ are defined in \eqref{def.ga+} and \eqref{def.ga-}, respectively,
and the dissipation norm $\|\cdot\|_D$ is defined in \eqref{def.fc.disDL} or \eqref{def.fc.disDB}  for the Landau or non-cutoff Boltzmann equation, respectively.  Furthermore, taking the square root on both sides of \eqref{eng1} and then  integrating in $\bar{k}\in \Z^2$ we obtain
\begin{align}
&\int_{\Z^2}\sup\limits_{0\le t \le T}\|\widehat{\pa^{\al}f}\|d\Si (\bar{k})
+\int_{\Z^2}\left(\int_0^T\|\{\FI-\FP\}\widehat{\pa^{\al}f}\|^2_{D}dt\right)^{1/2}d\Si (\bar{k})
\notag\\
&\qquad
+\int_{\Z^2}|\Upsilon_T^+ (\widehat{\pa^{\al}f})| d\Si (\bar{k})
\notag\\
&\lesssim \int_{\Z^2}\|\widehat{\pa^{\al}f}_0\|d\Si (\bar{k})+\int_{\Z^2}|\Upsilon_T^- (\widehat{\pa^{\al}f})|d\Si (\bar{k})\notag\\
&\qquad+\int_{\Z^2}\left(\int_0^T\left|\mathscr{R}\left(\CF_{\bar{x}}\{\pa^{\al}\Ga(f,f)\},\{\FI-\FP\}\widehat{\pa^{\al}f}\right)\right|dt\right)^{1/2}d\Si (\bar{k}).
\label{eng1ad}
\end{align}
By using \eqref{0ifb},
\eqref{payifb} and \eqref{paxifb} together with Lemma \ref{keynp.es}, it  follows from \eqref{eng1ad} that
\begin{multline}\label{eng2}
\int_{\Z^2}\sup\limits_{0\le t\le T}\|\widehat{\pa^{\al}f}\|d\Si (\bar{k})
+\int_{\Z^2}\left(\int_0^T\|\{\FI-\FP\}\widehat{\pa^{\al}f}\|^2_{D}dt\right)^{1/2}d\Si (\bar{k})
\\
+\int_{\Z^2}|\Upsilon_T^+ (\widehat{\pa^{\al}f})| d\Si (\bar{k})
\lesssim\|\pa^\al f_0\|_{L^1_{\bar{k}} L^2_{x_1,v}}
+E(\widehat{g_{\pm}})\\
+C_\eta\left\|\pa^{\al} f\right\|_{L^1_{\bar{k}}L^\infty_TL^2_{x_1,v}}
\left\|f\right\|_{L^1_{\bar{k}}L^2_TH_{x_1}^1L^2_{v,D}} \\
+C_\eta\left\|\pa^{\al} f\right\|_{L^1_{\bar{k}}L^2_TL_{x_1}^2L^2_{v,D}}
\left\|f\right\|_{L^1_{\bar{k}}L^\infty_TH^1_{x_1}L^2_{v}}
+\eta\left\|\{\FI-\FP\}\pa^{\al}f\right\|_{L^1_{\bar{k}}L^2_TL_{x_1}^2L^2_{v,D}},
\end{multline}
where we have used \eqref{paxifb} and the norm $E(\cdot)$ defined in \eqref{def.Enorm} to control the boundary term on the right hand-side. Consequently, a suitable linear combination of  the above estimate \eqref{eng2} and the macroscopic dissipation estimate \eqref{abcif.es} gives rise to
\begin{equation}\label{eng3}
\CE_T(f)+\CD_T(f)\lesssim
\sum_{|\al|\leq 1}\|\pa^\al f_0\|_{L^1_{\bar{k}} L^2_{x_1,v}}
+\CE_T(f)\CD_T(f)+E(\widehat{g_{\pm}}),
\end{equation}
where $\CE_T(f)$ and $\CD_T(f)$ are defined as in \eqref{def.et} and \eqref{def.dt}, respectively.
Performing  similar calculations to those used in obtaining \eqref{eng3}, recalling also the methods used to obtain \eqref{a priori torus weighted},  then one can also show the following velocity weighted estimate:
\begin{multline}\label{eng4}
\CE_{T,w}(f)+\CD_{T,w}(f)\lesssim
\sum_{|\al|\leq 1}\|w_{q,\vth}\pa^\al f_0\|_{L^1_{\bar{k}} L^2_{x_1,v}}
\\
+\CE_{T,w}(f)\CD_{T,w}(f)+E(w_{q,\vth}\widehat{g_{\pm}}).
\end{multline}
Here we recall that the velocity weight is useful to obtain the sub-exponential time decay only in case of {\it soft} potentials. In fact, according to Lemma \ref{wgesL}, one has
\begin{equation*}
\mathscr{R}(L\widehat{\pa^{\al}f},w^2_{q,\vth}\widehat{\pa^{\al}f})_{L^2_{x_1,v}}\geq \de_0\|w_{q,\vth}\widehat{\pa^{\al}f}\|^2_D-C\|\widehat{\pa^{\al}f}\|^2_D.
\end{equation*}
Applying this inequality, we are able to obtain an estimate similar to \eqref{eng1ad} with its right-hand term containing an extra term
$$
\int_{\Z^2}\left(\int_0^T\|\widehat{\pa^{\al}f}\|^2_Ddt\right)^{1/2}d\Si(\bar{k}),
$$
where the above term can be controlled by taking a linear combination with the estimate \eqref{eng3}. Hence, similar to how we derived \eqref{eng2} from \eqref{eng3}, we then obtain \eqref{eng4} analogously.  This concludes the proof in the case of the inflow boundary condition.


\medskip
\noindent{\it Case 2. Specular reflection boundary condition.} Compared with Case 1 on the inflow boundary condition \eqref{0ifb},
the difference in this case stems from the boundary term, i.e.,~we have to control
\begin{multline}
\int_{\Z^2}\Upsilon^-_{T,w}(\widehat{\pa^{\al}f})d\Si(\bar{k})=
\int_{\Z^2}\left( -\int_0^T\int_{v_1<0}v_1w^2|\widehat{\pa^{\al}f}(1)|^2dvdt\right.\\
\left.+\int_0^T\int_{v_1>0}v_1w^2|\widehat{\pa^{\al}f}(-1)|^2dvdt\right)^{1/2}d\Si(\bar{k}),
\label{srb2}
\end{multline}
with
\begin{equation}\label{0srb}
\hat{f}(-1,\bar{k},v_1,\bar{v})|_{v_1>0}=\hat{f}(-1,\bar{k},-v_1,\bar{v}),\ \
\hat{f}(1,\bar{k},v_1,\bar{v})|_{v_1<0}=\hat{f}(1,\bar{k},-v_1,\bar{v}).
\end{equation}
Note that above and below we use the notation $w^2 = w^2_{q,\vth}$ for brevity.
It is straightforward to see that if $\al=0$ or $\pa^{\al}=\pa_{\bar{x}}$, then by the change of variable $v_1 \to -v_1$ and \eqref{0srb},
\eqref{srb2} can be rewritten as
\begin{multline*}
\int_{\Z^2}\left(\int_0^T\int_{v_1>0}v_1w^2|\widehat{\pa^{\al}f}(1)|^2dvdt
-\int_0^T\int_{v_1<0}v_1w^2|\widehat{\pa^{\al}f}(-1)|^2dvdt\right)^{1/2}d\Si(\bar{k})\\
=\int_{\Z^2}\Upsilon^+_{T,w}(\pa^{\al}f)d\Si(\bar{k}),
\end{multline*}
which is just the corresponding boundary term on the left-hand side of \eqref{eng1} and hence they can be cancelled.
If $\pa^{\al}=\pa_{x_1}$, we first have from the equation \eqref{paLLeq} with $\al=0$ and the boundary condition \eqref{0srb} and Remark \ref{property.remark} that
\begin{align}
&v_1\pa_{x_1}\hat{f}(1,\bar{k},v_1) \notag\\
&=-\pa_t\hat{f}(1,\bar{k},v_1)-i\bar{k}\cdot\bar{v}\hat{f}(1,\bar{k},v_1)
-L\hat{f}(1,\bar{k},v_1)+\widehat{\Ga(f,f)}(1,\bar{k},v_1) \notag\\
&=-\pa_t\hat{f}(-1,\bar{k},-v_1)-i\bar{k}\cdot\bar{v}\hat{f}(-1,\bar{k},-v_1)
-L\hat{f}(-1,\bar{k},-v_1)+\widehat{\Ga(f,f)}(-1,\bar{k},-v_1) \notag\\
&=-\pa_t\hat{f}(-1,\bar{k},v_1)-i\bar{k}\cdot\bar{v}\hat{f}(-1,\bar{k},v_1)
-L\hat{f}(-1,\bar{k},v_1)+\widehat{\Ga(f,f)}(-1,\bar{k},v_1) \notag\\
&=v_1\pa_{x_1}\hat{f}(-1,\bar{k},v_1),\label{srb.sym}
\end{align}
where the second identity is valid due to the fact that we look for solutions to \eqref{paLLeq} satisfying the symmetric property $\hat{f}(t,-x_1,\bar{k},-v_1)=\hat{f}(t,x_1,\bar{k},v_1)$.  After that we used the boundary condition \eqref{0srb}. Note that in \eqref{srb.sym} we have suppressed the time variable for brevity.
Then \eqref{srb.sym} implies that the boundary terms on both sides of \eqref{eng1} are equal and thus they can also cancel each other.
Therefore, similar to how we obtained \eqref{eng4}, it follows that we have
\begin{equation}\label{engsrb}
\CE_{T,w}(f)+\CD_{T,w}(f)\lesssim
\sum_{|\al|\leq 1}\|w_{q,\vth}\pa^\al f_0\|_{L^1_{\bar{k}} L^2_{x_1,v}}
+\CE_{T,w}(f)\CD_{T,w}(f).
\end{equation}
Once \eqref{eng4} and \eqref{engsrb} are obtained, then \eqref{ifenges} and \eqref{srenges}
follow from the standard continuity argument, cf.~\cite{DLX-2016}, \cite{Guo-L}. This concludes the proof of the global existence of mild solutions.

\medskip
\noindent \underline{\it Time decay rates.} The proof is quite similar to the torus case. We shall only show \eqref{ifdec} for the inflow boundary value problem, since \eqref{srdec} for the specular reflection boundary value problem can be obtained in the same way. For the same reason as in the torus case, we only focus on the soft potentials.
Let $\hat{h}=e^{\la t^p}\hat{f}$ with $\la>0$ and  $0<p<1$ to be determined later. Then $\hat{h}$ satisfies
\begin{equation*}
\pa_t\widehat{\pa^{\al}h}+v_1\pa_{x_1}\widehat{\pa^{\al}h}+i\bar{k}\cdot\bar{v}\widehat{\pa^{\al}h}
+L\widehat{\pa^{\al}h}=e^{-\la t^p}\CF_{\bar{x}}\{\pa^{\al}\Ga(h,h)\}+\la pt^{p-1}\hat{h},
\end{equation*}
with initial data
\begin{equation*}
\widehat{\pa^{\al}h}(0,x_1,\bar{k},v)=\widehat{\pa^{\al}f}_0(x_1,\bar{k},v),
\end{equation*}
and for the inflow boundary condition we have
\begin{equation*}
\hat{h}(t,-1,\bar{k},v)|_{v_1>0}=e^{\la t^p}\widehat{g_-}(t,\bar{k},v),\ \ \hat{h}(t,1,\bar{k},v)|_{v_1<0}=e^{\la t^p}\widehat{g_+}(t,\bar{k},v).
\end{equation*}
Next, performing similar calculations as to how we obtained \eqref{eng3}, \eqref{eng4} and \eqref{engsrb}, one has
\begin{align}
&\sum_{|\al|\leq 1}\int_{\Z^2}\sup\limits_{0\leq t\leq T}\left\|\widehat{\pa^{\al}h}(t,\bar{k})\right\|d\Si(\bar{k})
+\sum_{|\al|\leq 1}\int_{\Z^2}\left(\int_0^T\left\|\widehat{\pa^{\al}h}\right\|^2_{D}dt\right)^{1/2}d\Si(\bar{k})
\notag\\
&\lesssim \sum_{|\al|\leq 1}\int_{\Z^2}\left\|\widehat{\pa^{\al}f}_0\right\|d\Si(\bar{k})+\sup_{0<t<T}E(e^{\la t^p}g_\pm)\notag\\
&\qquad+\sqrt{\la p}\sum_{|\al|\leq 1}\int_{\Z^2}\left(\int_0^Tt^{p-1}\left\|\widehat{\pa^{\al}h}\right\|^2dt\right)^{1/2}d\Si(\bar{k}).\label{engh}
\end{align}
We recall that above equation and in the below equation we use the notation $\| \cdot \| = \| \cdot \|_{L^2_{x_1, v}}$ and $\| \cdot \|_D = \| \cdot \|_{L^2_{x_1, v,D}}$
Further the last term on the right-hand can be estimated by the time-velocity splitting technique in the completely same way as for we treated \eqref{thmt.p1}  in the torus case. Thus, from \eqref{engh} we have
\begin{multline*}
\sum_{|\al|\leq 1}\int_{\Z^2}\sup\limits_{0\leq t\leq T}\left\|\widehat{\pa^{\al}h}(t,\bar{k})\right\|d\Si(\bar{k})
+\sum_{|\al|\leq 1}\int_{\Z^2}\left(\int_0^T\left\|\widehat{\pa^{\al}h}\right\|^2_{D}dt\right)^{1/2}d\Si(\bar{k})\\
\lesssim \sum_{|\al|\leq 1}\int_{\Z^2}\left\|w_{q,\vartheta}\widehat{\pa^{\al}f_0}\right\|d\Si(\bar{k})+E(w_{q,\vartheta}\widehat{g_\pm})+\sup_{0\leq t\leq T}E(e^{\la t^p}\widehat{g_\pm}),
\end{multline*}
which proves \eqref{ifdec}. This also concludes the proof of the time-decay rates of solutions.

\medskip
\noindent \underline{\it Positivity and uniqueness.} The uniqueness of the initial
boundary value problem \eqref{LLeq}, \eqref{idf}
\eqref{ifb} or \eqref{srb} can be proved by applying the similar method as the previous ``energy estimates" part
and which is now quite standard.  Also the local solution that we extend here from Section \ref{sec8} is unique.  Therefore we omit these analogous details.
Noticing that on the boundary
$$
F(t,\pm1,\bar{x},v)=\mu+\mu^{\frac{1}{2}}g_\pm(t,\bar{x},v)\geq0,
$$
the positivity of the solution to the Landau equation
is guaranteed by the maximum principle, cf. \cite{Guo-L}.  For the non-cutoff Boltzmann case, the positivity of the solution can be
also proved by using the same argument as in \cite[page 833]{GS}.
This completes the proofs of Theorems \ref{mthif} and \ref{mthsr}.
\end{proof}

\begin{proof}[Proof of Theorem \ref{regp.th}]
We shall show that under the assumptions \eqref{thm.cp.as1} or \eqref{thm.cp.as2}, the regularity of the initial data and the boundary data can propagate from the boundary
into the interior of the channel along the tangential direction. In fact, let $|\al|\leq 1$, then we may derive the following trilinear estimates with an extra Fourier multiplier $\langle \bar{k}\rangle^{2m}$:
\begin{align}
&\int_{\Z^2}\left(\int_0^T|(
\CF_{\bar{x}}\{\Gamma(\pa^{\al} f,g)\}, \langle\bar{k}\rangle^{2m}w^2_{q,\vartheta}\widehat{h})|dt\right)^{1/2}~d\Si(\bar{k})\notag\\
&\leq C_\eta\bigg(\left\|w_{q,\vth}\pa^{\al} f\right\|_{L^1_{\bar{k},m}L^\infty_TL^2_{x_1,v}}
\left\|w_{q,\vth} g\right\|_{L^1_{\bar{k},m}L^2_TH^1_{x_1}L^2_{v,D}}
\notag\\
&\qquad\qquad+\left\|w_{q,\vth}\pa^{\al} f\right\|_{L^1_{\bar{k},m}L^2_TL^2_{x_1}L^2_{v,D}}
\left\|w_{q,\vth} g\right\|_{L^1_{\bar{k},m}L^\infty_TH^1_{x_1}L^2_{v}}
\bigg)
\notag\\
&\qquad\qquad\qquad\qquad
+\eta\left\|w_{q,\vth}h\right\|_{L^1_{\bar{k},m}L^2_TL^2_{x_1}L^2_{v,D}},\label{trim.es1}
\end{align}
and
\begin{align}
&\int_{\Z^2}\left(\int_0^T|(
\CF_{\bar{x}}\{\Gamma(f,\pa^{\al}g)\}, \langle\bar{k}\rangle^{2m}w^2_{q,\vartheta}\widehat{h})|dt\right)^{1/2}~d\Si(\bar{k})\notag \\
&\leq C_\eta\bigg(\left\|w_{q,\vth}\pa^{\al} g\right\|_{L^1_{\bar{k},m}L^\infty_TL^2_{x_1,v}}
\left\|w_{q,\vth} f\right\|_{L^1_{\bar{k},m}L^2_TH^1_{x_1}L^2_{v,D}}
\notag \\
&\qquad\qquad+\left\|w_{q,\vth}\pa^{\al} g\right\|_{L^1_{\bar{k},m}L^2_TL^2_{x_1}L^2_{v,D}}
\left\|w_{q,\vth} f\right\|_{L^1_{\bar{k},m}L^\infty_TH^1_{x_1}L^2_{v}}
\bigg)
\notag\\
&\qquad\qquad\qquad\qquad
+\eta\left\|w_{q,\vth}h\right\|_{L^1_{\bar{k},m}L^2_TL^2_{x_1}L^2_{v,D}}.\label{trim.es2}
\end{align}
Moreover, regarding the macroscopic dissipation,
for the inflow boundary condition \eqref{ifb}, analogous to  it holds for $|\al|\leq 1$ that
\begin{align}
\|\pa^{\al}&\left[a,b,c\right]\|_{L^1_{\bar{k},m}L^2_TL^2_{x_1,v}}\notag \\
\lesssim& \sum\limits_{|\al|\leq1}\left(\left\|\{\FI-\FP\}\pa^{\al} f\right\|_{L^1_{\bar{k},m}L^2_TL^2_{x_1}L^2_{v,D}}
+
\|\pa^{\al}f\|_{L^1_{\bar{k}}L^\infty_TL^2_{x_1,v}}
+
\|\pa^{\al}f_0\|_{L^1_{\bar{k},m}L^2_{x_1,v}}\right)
\notag \\
&
+E(\langle\bar{k}\rangle^m\widehat{g}_\pm)
+\sum\limits_{|\al|\leq1}\int_{\Z^2}\left(\int_0^T\left\|(\langle\bar{k}\rangle^{m}\widehat{\pa^\al H},\mu^{1/4})_{L_v^2}\right\|_{L_{x_1}^2}^2dt\right)^{1/2}d\Si(\bar{k})\notag \\
&+
\sum\limits_{|\al|\leq1}\int_{\Z^2} \langle\bar{k}\rangle^{m}\Upsilon^+_T (\widehat{\pa^\al f})d\Si(\bar{k})
+\int_{\Z^2} \langle\bar{k}\rangle^{m}\Upsilon^-_T (\frac{\hat{H}}{|v_1|})d\Si(\bar{k}).\label{abcifm.es}
\end{align}
For the specular reflection boundary condition \eqref{srb}, for $|\al|\leq 1$, it similarly holds that
\begin{align}
\|\pa^{\al}&\left[a,b,c\right]\|_{L^1_{\bar{k},m}L^2_TL^2_{x_1,v}}\notag
\\
&\lesssim
\sum\limits_{|\al|\leq1}\left\|\{\FI-\FP\}\pa^{\al} f\right\|_{L^1_{\bar{k},m}L^2_TL^2_{x_1}L^2_{v,D}}
+\sum\limits_{|\al|\leq1}\|\pa^{\al}f\|_{L^1_{\bar{k},m}L^\infty_TL^2_{x_1,v}}
\notag\\
&\quad+\sum\limits_{|\al|\leq1}\|\pa^{\al}f_0\|_{L^1_{\bar{k},m}L^2_{x_1,v}}
\notag\\
&\quad+\sum\limits_{|\al|\leq1}\int_{\Z^2}\left(\int_0^T\left\|(\langle\bar{k}\rangle^{m}\widehat{\pa^\al H},\mu^{1/4})_{L_v^2}\right\|_{L_{x_1}^2}^2dt\right)^{1/2}d\Si(\bar{k}).
\label{abcsrm.es}
\end{align}
We point out that the proofs of these last two macroscopic estimates follow directly as in the proofs of Theorem \ref{abc.th}.

In what follows, we will only explain the regularity propagation properties for the inflow boundary condition. The corresponding results for the specular reflection boundary condition can be obtained similarly to how we obtained \eqref{engsrb} since the mode multiplier $\langle\bar{k}\rangle^{2m}$ doesn't influence the symmetry of $f(x_1,\bar{x},v_1,\bar{v})$.
Indeed, let $|\al|\leq 1$, then by taking the complex inner product of \eqref{paLLeq} and $\langle\bar{k}\rangle^{2m}w_{q,\vth}\widehat{\pa^{\al}f}$
with respect to $(x_1,v)$, we obtain
\begin{multline*}
(\pa_t\widehat{\pa^{\al}f},w_{q,\vth}\langle\bar{k}\rangle^{2m}\widehat{\pa^{\al}f})
+(v_1\pa_{x_1}\widehat{\pa^{\al}f},w_{q,\vth}\langle\bar{k}\rangle^{2m}\widehat{\pa^{\al}f})
+(i\bar{k}\cdot\bar{v}\widehat{\pa^{\al}f},w_{q,\vth}\langle\bar{k}\rangle^{2m}\widehat{\pa^{\al}f})
\\
+(L\widehat{\pa^{\al}f},w_{q,\vth}\langle\bar{k}\rangle^{2m}\widehat{\pa^{\al}f})
=(\CF_{\bar{x}}\{\pa^{\al}\Ga(f,f)\},w_{q,\vth}\langle\bar{k}\rangle^{2m}\widehat{\pa^{\al}f}),
\end{multline*}
for any $0\leq t\leq T$ and $\bar{k}\in \Z^2$.
We use Lemma \ref{wgesL} to observe that  the equation above further implies  that
that
\begin{align}
&\int_{\Z^2}\sup\limits_{0 \le t \le T}\|w_{q,\vth}\langle\bar{k}\rangle^{m}\widehat{\pa^{\al}f}(t,\bar{k})\|d\Si(\bar{k})
\notag
\\
&\qquad+\sqrt{2\de_0}\int_{\Z^2}\left(\int_0^T\|w_{q,\vth}\langle\bar{k}\rangle^{m}
\widehat{\pa^{\al}f}\|^2_{D}dt\right)^{1/2}d\Si(\bar{k})
\notag \\
&\qquad
-C\int_{\Z^2}\left(\int_0^T\|\langle\bar{k}\rangle^{m}\widehat{\pa^{\al}f}\|^2_Ddt\right)^{1/2}d\Si(\bar{k})\notag\\
&
\qquad+\int_{\Z^2}\langle\bar{k}\rangle^{m}\Upsilon^+_{T,w}(\widehat{\pa^{\al}f})d\Si(\bar{k})-\int_{\Z^2}\langle\bar{k}\rangle^{m}\Upsilon^-_{T,w}(\widehat{\pa^{\al}f})d\Si(\bar{k})\notag\\
&
\lesssim
\int_{\Z^2}\|w_{q,\vth}\widehat{\pa^{\al}f}_0\|d\Si(\bar{k})
\notag
\\
& \qquad +\int_{\Z^2}
\left(\int_0^T\left|(\CF_{\bar{x}}\{\pa^{\al}\Ga(f,f)\},w^2_{q,\vth}\langle\bar{k}\rangle^{2m}\widehat{\pa^{\al}f})\right|
dt\right)^{1/2}d\Si(\bar{k}).\label{eng2m}
\end{align}
Similarly, using Lemmas \ref{esLL} and \ref{esBL}, it also holds without any velocity weight that
\begin{align}
&\int_{\Z^2}\sup\limits_{0 \le t \le T}\|\langle\bar{k}\rangle^{m}\widehat{\pa^{\al}f}(t,\bar{k})\|d\Si(\bar{k})
\notag
\\
& \qquad
+
\sqrt{2\de_0}\int_{\Z^2}\left(\int_0^T\|\langle\bar{k}\rangle^{m}\{\FI-\FP\}
\widehat{\pa^{\al}f}\|^2_{D}dt\right)^{1/2}d\Si(\bar{k})
\notag
\\
&
\qquad+\int_{\Z^2}\langle\bar{k}\rangle^{m}\Upsilon^+_{T}(\widehat{\pa^{\al}f})d\Si(\bar{k})-\int_{\Z^2}\langle\bar{k}\rangle^{m}\Upsilon^-_{T}(\widehat{\pa^{\al}f})d\Si(\bar{k})\notag\\
&\lesssim \int_{\Z^2}\|\widehat{\pa^{\al}f}_0\|d\Si(\bar{k})
\notag\\
& \qquad
+\int_{\Z^2}
\left(\int_0^T\left|(\CF_{\bar{x}}\{\pa^{\al}\Ga(f,f)\},w^2_{q,\vth}\langle\bar{k}\rangle^{2m}\widehat{\pa^{\al}f})\right|
dt\right)^{1/2}d\Si(\bar{k}).\label{weng2m}
\end{align}
In \eqref{eng2m} and \eqref{weng2m} we notice that, as proved previously, the boundary terms vanish for the specular reflection boundary condition.   Further the boundary term with negative signs for incoming boundary velocities can be bounded by $E(w_{q,\vth}\langle\bar{k}\rangle^{m}\widehat{g_{\pm}})$.  The trilinear terms can be further bounded as in \eqref{trim.es1} and \eqref{trim.es2}. Therefore, applying all of those estimates, a linear combination of \eqref{eng2m} and
\eqref{weng2m} together with \eqref{abcifm.es} or \eqref{abcsrm.es}  implies that
\begin{multline*}
\sum_{|\al|\leq 1}\int_{\Z^2}\sup\limits_{0 \le t \le T}\|w_{q,\vth}\langle\bar{k}\rangle^{m}\widehat{\pa^{\al}f}(t,\bar{k})\|d\Si(\bar{k})
\\
+\sum_{|\al|\leq 1}\int_{\Z^2}\left(\int_0^T\|w_{q,\vth}\langle\bar{k}\rangle^{m}\widehat{\pa^{\al}f}\|^2_{D}dt\right)^{1/2}d\Si(\bar{k})
\\
\lesssim \sum_{|\al|\leq 1}\int_{\Z^2}\|w_{q,\vth}\langle\bar{k}\rangle^{m}\widehat{\pa^{\al}f}_0\|d\Si(\bar{k})+E(w_{q,\vth}\langle\bar{k}\rangle^{m}\widehat{g}_{\pm}),
\end{multline*}
provided that $\eps_0>0$ in \eqref{thm.cp.as1}  is suitably small.  Note that in the above estimate for the specular reflection case there is no boundary term $E(w_{q,\vth}\langle\bar{k}\rangle^{m}\widehat{g}_{\pm})$ on the right-hand side, so in that case the smallness assumption \eqref{thm.cp.as1} is actually just \eqref{thm.cp.as2}. Then the above estimate implies both of the estimates \eqref{reg.ib} and \eqref{reg.sr} for solutions corresponding to the inflow and specular reflection boundary conditions, respectively. This then completes the proof of Theorem \ref{regp.th}.
\end{proof}

\section{Local-in-time existence}\label{sec8}

In this section, we are concerned with the local-in-time existence of solutions with mild regularity to the problem {\bf (PT)} in the torus case and to the problem {\bf (PC)} in the finite channel case. For brevity of presentation, we give the full details of the proof only for the non-cutoff Boltzmann equation with the specular reflection boundary condition in the finite channel. The main idea is motivated by \cite[Theorem 4.2, page 541]{AMUXY-2011-CMP} and \cite[Lemma 5.1, page 4098]{MS-2016-JDE}.  Our approach used here can  also  be adopted to treat the local in time existence for the other cases mentioned in this paper.  Then the analogous local existence theorem to Theorem \ref{le.th} is true in the other cases that arise in this paper as in Theorems \ref{Torus Existence} - \ref{regp.th}.  For brevity of the presentation we omit these details.  However we remark that in the case of  the torus domain, it is possible to carry out a straightforward proof of local in time existence and uniqueness based on the standard approximation argument in terms of the known existence results for regular initial data.

\begin{theorem}[Local existence]\label{le.th}
Let all the conditions of Theorem \ref{mthsr} be satisfied, then there are $\eps_0>0$, $T_0>0$ and $C_0>0$ such that if $F_0(x_1,\bar{x},v)=\mu+\mu^{\frac{1}{2}}f_0(x_1,\bar{x},v)\geq0$
and
\begin{equation*}
\sum\limits_{|\al|\leq1}\|\pa^\al f_{0}\|_{L_{\bar{k}}^1L^2_{x_1,v}}\leq\eps_0,
\end{equation*}
then the specular reflection boundary problem for the non-cutoff Boltzmann equation in the finite channel  \eqref{LLeq},
\eqref{idf} with \eqref{srb}  admits a unique solution
$$
f(t,x,v), \ 0\leq t\leq T_0,\ x\in \Omega=I\times \T^2,\ v\in \R^3,
$$
with
\begin{equation}
\label{le.th.space}
f,\ \na_xf\in L^1_{\bar{k}}L^\infty_{T_0}L^2_{x_1,v}\cap L^1_{\bar{k}}L^2_{T_0}L^2_{x_1}L^2_{v,D},
\end{equation}
satisfying
$$
F(t,x_1,\bar{x},v)=\mu+{\mu}^{1/2}f(t,x_1,\bar{x},v)\geq0,\quad
f(t,-x_1,\bar{x},-v_1,\bar{v})=f(t,x_1,\bar{x},v_1,\bar{v}),
$$
and the uniform estimate
\begin{equation}\label{loces}
\sum_{|\al|\leq 1} \left(\|\pa^\al f\|_{L^1_{\bar{k}}L^\infty_{T_0}L^2_{x_1,v}}+\|\pa^\al f\|_{L^1_{\bar{k}}L^2_{T_0}L^2_{x_1}L^2_{v,D}}\right)
\leq C_{0}
\sum\limits_{|\al|\leq1}\|\pa^{\al}f_0\|_{L_{\bar{k}}^1L^2_{x_1,v}}.
\end{equation}
\end{theorem}

To prove Theorem \ref{le.th}, we start from the following linear inhomogeneous problem
\begin{eqnarray}\label{lc.lb}
\left\{\begin{array}{rll}
&\pa_t g+v_1\pa_{x_1}g+\bar{v}\cdot\na_{\bar{x}}g+\mathscr{L}_1g-\Ga(\bh,g)=-\mathscr{L}_2\bh,\\[2mm]
&g(0,x,v)=g_0(x,v),\\[2mm]
&g(-1,\bar{x},v_1,\bar{v})|_{v_1>0}=g(-1,\bar{x},-v_1,\bar{v}),\\[2mm]
& g(1,\bar{x},v_1,\bar{v})|_{v_1<0}=g(1,\bar{x},-v_1,\bar{v}),
\end{array}\right.
\end{eqnarray}
for a given function $\bh=\bh(t,x,v)$, 
where we have denoted the linear operators
\begin{equation*}
\mathscr{L}_1f=-\mu^{-\frac{1}{2}}Q(\mu,\mu^{\frac{1}{2}}f),\quad \mathscr{L}_2f=-\mu^{-\frac{1}{2}}Q(\mu^{\frac{1}{2}}f,\mu).
\end{equation*}
For the later use, we recall that thanks to \cite[Lemma 2.15 and Proposition 2.16, page 937--939]{AMUXY-2012-JFA}, it holds that
\begin{equation}\label{L1L2.es}
\left\langle \mathscr{L}_1g, g\right\rangle
\gtrsim
|g|^2_{D}-C_1\|\langle v\rangle^{\ga/2}g\|_{L^2_v}^2,\ \
\left|\left\langle \mathscr{L}_2g, h\right\rangle\right|\lesssim
\left\|\mu^{1/10^3}g\right\|_{L^2_v}\left\|\mu^{1/10^3}h\right\|_{L^2_v},
\end{equation}
where $C_1>0$ 
is a universal constant. The solvability of \eqref{lc.lb} is guaranteed by the following lemma.

\begin{lemma}\label{lb.loc}
There are $\eps_0>0$, $T_1>0$ and $C_1>0$ such that if
$$
g_0,\na_x g_0\in L^1_{\bar{k}}L^2_{x_1,v},\quad
\bh,\na_x\bh \in L^1_{\bar{k}}L^\infty_{T_0}L^2_{x_1,v}\cap L^1_{\bar{k}}L^2_{T_0}L^2_{x_1}L^2_{v,D},
$$
for $T_0\in (0,T_1]$,  and it holds that
\begin{equation}
\label{lb.loc.a1}
g_0(x_1,\bar{x},v_1,\bar{v})=g_0(-x_1,\bar{x},-v_1,\bar{v}),\quad
\bh(t,x_1,\bar{x},v_1,\bar{v})=\bh(t,-x_1,\bar{x},-v_1,\bar{v}),
\end{equation}
and
\begin{equation}
\label{lb.loc.a2}
\sum\limits_{|\al|\leq 1}\left\{\|\pa^{\al}\bh\|_{ L^1_{\bar{k}}L^\infty_{T_0}L^2_{x_1,v}}+\left\|\pa^{\al}\bh\right\|_{L^1_{\bar{k}}L^2_{T_0}L^2_{x_1}L^2_{v,D}}\right\}\leq \eps_0,
\end{equation}
then the initial boundary value problem \eqref{lc.lb} admits a unique weak solution
$$
g(t,x,v), \ 0\leq t\leq T_0,\ x\in \Omega=I\times \T^2,\ v\in \R^3,
$$
satisfying
\begin{equation}
\label{lg.symm}
g(t,x_1,\bar{x},v_1,\bar{v})=g(t,-x_1,\bar{x},-v_1,\bar{v})
\end{equation}
and
\begin{multline}\label{lg.es}
\sum\limits_{|\al|\leq 1}\|\pa^{\al}g\|_{L^1_{\bar{k}}L^\infty_{T_0}L^2_{x_1,v}}+\sum\limits_{|\al|\leq 1}\|\pa^{\al}g\|_{ L^1_{\bar{k}}L^2_{T_0}L^2_{x_1}L^2_{v,D}}\\
\leq C_0\left(\sum\limits_{|\al|\leq 1}\|\pa^{\al}g_0\|_{L^1_{\bar{k}}L^2_{x_1,v}}+\sqrt{T_1}\sum\limits_{|\al|\leq 1}\|\pa^{\al}\bh\|_{ L^1_{\bar{k}}L^2_{T_0}L^2_{x_1}L^2_{v,D}}\right).
\end{multline}
\end{lemma}

\begin{proof}
We will divide the proof into four steps as in the following.

\medskip
\noindent{\it Step 1.} First of all, we remark that since the operators $\CL_1$, $\CL_2$ and $\Ga(\bh,\cdot)$ involve a singular kernel and non-local derivatives, then it may be too hard  to prove the local in time existence of solutions to the linear problem \eqref{lc.lb} directly by applying the standard ODE theory.
However, due to
$$
L^1_{\bar{k}}L^\infty_{T_0}H^1_{x_1}L^2_v\subset L^\infty((0,T_0)\times\T^2;H^1_{x_1}L^2_v)\subset L^\infty_{T_0}L^2_{\bar{x}}H^1_{x_1}L^2_v,
$$
we may expect to achieve the proof through an approximation procedure by first constructing approximate solutions in the function space
$$
L^\infty(0,T_0;L^2_{\bar{x}}H^1_{x_1}L^2_v).
$$
The advantage of using the above function space in this proof is that, compared to $L^1_{\bar{k}}L^\infty_{T_0}H^1_{x_1}L^2_v$, it is more convenient to work with the weak formulation of \eqref{lc.lb} in the $L^2$ framework because $L^2_{\bar{x}}H^1_{x_1}L^2_v$ is a Hilbert space.
For this purpose, we will smooth out the data $\bh$ and $g_0$ in \eqref{lc.lb} with respect to spatial variable as
\begin{equation}
\label{def.hg0}
\bh_\vps=\bh\ast_x\chi_{\vps},\quad g_{0,\vps}=g_{0}\ast_x\chi_{\vps}.
\end{equation}
Here $*_x$ is the standard convolution with respect to $x$ and $\chi_{\vps}$ with $\vps>0$ suitably small is a smooth mollifier for $x\in \Omega=I\times \T^2$ such that $0\leq\chi_\vps\leq1$, and
\begin{eqnarray*}
\chi_{\vps}(x_1,\bar{x})=\left\{\begin{array}{rll}&1, \ \ \text{for }|x_1|<1-\vps,\\[2mm]
&0, \ \ \text{for }1-\frac{\vps}{2}<|x_1|<1,
\end{array}\right.
\end{eqnarray*}
and $\chi_\vps$ is even in $x_1\in (-1,1)$. As $\bh$ and $g_0$ are spatially periodic in $\bar{x}$, so are $\bh_\vps$ and $g_{0,\vps}$. Moreover, it is obvious to see that the symmetric property \eqref{lb.loc.a1} also holds true for $\bh_\vps$ and $g_{0,\vps}$.
Now, in order to solve \eqref{lc.lb}, let us first consider the corresponding linear inhomogeneous problem with $\bh$ and $g_0$ replaced by  data $\bh_\vps$ and $g_{0,\vps}$ defined in \eqref{def.hg0}, and set $g_\vps=g_\vps(t,x,v)$ to be the corresponding solution, namely, $g_\vps$ satisfies
\begin{eqnarray}\label{lc.lb.vps}
\left\{\begin{array}{rll}
&\pa_t g_{\vps}+v\cdot\na_{x}g_{\vps}+\mathscr{L}_1g_{\vps}-\Ga(\bh_\vps,g_{\vps})=-\mathscr{L}_2\bh_\vps,\\[2mm]
&g_{\vps}(0,x,v)=g_{0,\vps}(x,v),\\[2mm]
&g_{\vps}(-1,\bar{x},v_1,\bar{v})|_{v_1>0}=g_{\vps}(-1,\bar{x},-v_1,\bar{v}),\\[2mm]
& g_{\vps}(1,\bar{x},v_1,\bar{v})|_{v_1<0}=g_{\vps}(1,\bar{x},-v_1,\bar{v}).
\end{array}\right.
\end{eqnarray}

\medskip
\noindent{\it Step 2.}
Next, our goal is to solve \eqref{lc.lb.vps}. We define a linear operator $\CG$ as
\begin{equation}
\label{def.gselfa}
\CG=-\pa_t+(v_1\pa_{x_1}+\bar{v}\cdot\na_{\bar{x}}+\CL_1-\Ga(\bh_\vps,\cdot))^{*},
\end{equation}
where the adjoint operator $(\cdot)^*$ is taken with respect to the complex inner product in $L^2_{x,v}$. Note that $\CL_1$ is self-adjoint and the adjoint of $\Ga(\bh,\cdot)$ can be indeed defined as $\Ga^\ast(\bh,\cdot)$ which is given by \cite[page 843]{GS}.  We then {\it claim} that  $\CG$ is injective over the following function space
\begin{equation}\label{def.W1}
\begin{split}
\W_1=&\big\{g:\,g\in H^1(0,T_0; \CS(\Om\times\R^3_v))\text{ such that } g(T_0,x,v)\equiv 0,\\
&\qquad g(t,x_1,\bar{x},v_1,\bar{v})=g(t,-x_1,\bar{x},-v_1,\bar{v}), \text{ and} \\
&\qquad g(t,\pm1,\bar{x},v_1,\bar{v})|_{v_1\neq0}=g(t,\pm1,\bar{x},-v_1,\bar{v})\big\}.
\end{split}
\end{equation}
Here $\CS(\Om\times\R^3_v))$ is the standard Schwartz class by which we mean that $g\in \CS(\Om\times\R^3_v))$ is $C^\infty$ in all variables on the interior and decays faster than any polynomial in the $\R^3_v$ variables and is periodic in $\mathbb{T}^2$.
To show this {\it claim}, we take $g_{\vps}\in\W_1$ and  compute for $\al=(\al_1,\overline{\al})$ with $\al_1\leq1$ and $|\bar{\al}|\leq6$:
\begin{equation}\label{gip}
\begin{split}
\CR(\pa^{\al}\CG g_{\vps},\pa^\al g_{\vps})
=&
-\frac{1}{2}\frac{d}{dt}\|\pa^\al g_{\vps}\|^2
+(\CL_1\pa^\al g_{\vps},\pa^\al g_{\vps})+\CR(v_1\pa_{x_1}\pa^\al g_{\vps},\pa^\al g_{\vps})
\\&
-\sum\limits_{\al'\leq\al}C_{\alpha'}^{\alpha} \CR(\Ga^*(\pa^{\al'}\bh_\vps,\pa^{\al-\al'} g_{\vps}),\pa^\al g_{\vps}).
\end{split}
\end{equation}
The constant $C_{\alpha'}^{\alpha}$ is the multinomial coefficient.  
Here and below we have used the notation $\|\cdot\|=\|\cdot\|_{L^2_{x,v}}$ to be the $L^2$ norm in $x$ and $v$, and similarly the inner product is $(\cdot,\cdot)= (\cdot,\cdot)_{L^2_{x,v}}$.  We note that the analogous property to \eqref{srb.sym}  for $g_\vps$ can be used to show that the boundary term satisfies
$$
\CR(v_1\pa_{x_1}\pa^\al g_{\vps},\pa^\al g_{\vps})=0.  
$$
 We will perform our estimates by  iterating upon $|\alpha|$.

From this identity with $|\alpha|=0$, we apply \eqref{L1L2.es} to obtain for $0<t< T_0$ that 
\begin{multline}
\| g_{\vps}(t)\|^2
+
\la
\int_t^{T_0}\|g_{\vps}\|^2_{D}d\tau
\lesssim 
\int_{t}^{T_0} \left| \CR(\CG g_{\vps}, g_{\vps}) \right| ~ d\tau
+
\int_{t}^{T_0}  \left| \CR( g_{\vps},\Ga(\bh_\vps, g_{\vps}))  \right| ~ d\tau
\\
+
T_0\sup\limits_{t\leq \tau\leq T_0}
\| g_{\vps}(\tau)\|^2.
\notag
\end{multline}
We also split into $\langle v \rangle \ge R$ and  $\langle v \rangle \le R$ for some large $R\ge 1$ to handle the lower bound in \eqref{L1L2.es} which results in the last term above.  It further follows from Lemma \ref{bnp.es} with $w_{q,\vartheta}=1$   that 
\begin{multline}
\| g_{\vps}(t)\|^2+\la \int_t^{T_0}\|g_{\vps}\|^2_{D}d\tau
\lesssim 
\left(\sup\limits_{t\leq \tau\leq T_0}
\left\| g_{\vps}(\tau)\right\| \right)
\int_{t}^{T_0}
\left\|\CG g_{\vps}(\tau)\right\|d\tau
\\
+C_\eta
\int_{t}^{T_0}\left\| g_{\vps}\right\|^2_{D}
\left\|\bh_\vps\right\|^2_{L^\infty_xL^2_{v}}
d\tau
+
C_\eta
\int_{t}^{T_0}\left\| g_{\vps}\right\|^2 \left\|\bh_\vps\right\|^2_{L^\infty_xL^2_{v,D}}d\tau
\\
+
\eta
\int_t^{T_0}
\left\| g_{\vps}\right\|_D^2d\tau
+T_0\sup\limits_{t\leq \tau\leq T_0}
\left\| g_{\vps}(\tau)\right\|^2,
\notag
\end{multline}
for any $0<t<T_0$.  Above and in the rest of this section we use the additional norm notation $\left\|\cdot\right\|_D^2=\left\|\cdot\right\|_{L^2_{x,v,D}}^2$.
At the same time, with the Sobolev embedding $H^1_{x_1} \supset L^\infty(I)$, we note that using \eqref{lb.loc.a2} we have 
\begin{equation}\label{est1.epsilon0}
\sup\limits_{t\leq \tau\leq T_0}\left\|\bh_\vps\right\|^2_{L^\infty_xL^2_{v}}
\lesssim
 \left\|\bh_\vps\right\|^2_{ L^1_{\bar{k}}L^\infty_{T_0}H^1_{x_1}L^2_v}
 \lesssim
 \left\|\bh \right\|^2_{ L^1_{\bar{k}}L^\infty_{T_0}H^1_{x_1}L^2_v}
 \le C\epsilon_0
\end{equation}
and 
\begin{equation}\label{est2.epsilon0}
\int_{t}^{T_0}\left\|\bh_\vps\right\|^2_{L^\infty_xL^2_{v,D}}d\tau
\lesssim
\left\|\bh_\vps\right\|^2_{L^1_{\bar{k}}L^2_{T_0}H^1_{x_1}L^2_{v,D}}
\lesssim
\left\|\bh\right\|^2_{L^1_{\bar{k}}L^2_{T_0}H^1_{x_1}L^2_{v,D}}
 \le C\epsilon_0,
\end{equation}
where the constant $C >0$ above is uniform and it does not depend on $\vps>0$.

Therefore, for $\epsilon_0>0$ small enough in \eqref{lb.loc.a2} and $T_0>0$ sufficiently small we have
\begin{equation}
\| g_{\vps}(t)\|^2+
\la \int_t^{T_0}\|g_{\vps}\|^2_{D}d\tau 
\lesssim
\int_{t}^{T_0}\left\|\CG g_{\vps}\right\|d\tau,
\label{cg.ij.noD}
\end{equation}
for any $0<t<T_0$.   Then
 it is direct that the estimate \eqref{cg.ij.noD} implies that $\CG$ defined in \eqref{def.gselfa} is injective over $\W_1$. Hence,  the {\it claim} is proved.

For the time integral term on the right-hand side of \eqref{cg.ij.noD}, we should note that for $g_\vps\in \W_1$, and for any $\alpha$, it holds that
\begin{multline} \label{CG}
\int_{t}^{T_0}\left\|\pa^{\al}\CG g_{\vps}\right\|_{L^2_{x,v}}d\tau
\\
\lesssim
\int_{t}^{T_0}
\left(
\|\pa_\tau\pa^{\al}g_\vps\|+\|v\cdot \pa^{\al}\na_xg_\vps\|
+
\|\langle v \rangle^{(\gamma/2+s)^+}\pa^{\al} g_\vps\|_{L^2_xH^2_{v}}
\right)d\tau
\\
+
\int_{t}^{T_0}
\left(
\sum_{\al' \le \al}
\|\langle v \rangle^{(\gamma/2+s)^+}\pa^{\al-\al'} g_\vps\|_{L^2_xH^2_{v}}
\|\pa^{\al'}\bh_\vps\|_{L_x^\infty L^2_v}
\right)
d\tau
\le C_\vps
<\infty.
\end{multline}
Here we have used \cite[Proposition 6.10, page 4107]{MS-2016-JDE} to estimate $\CL_1\pa^{\al-\al'} g_\vps$ and $\Ga(\pa^{\al'}\bh_\vps,\pa^{\al-\al'}g_\vps)$.

In particular to obtain \eqref{CG} we used that for any $\alpha'$ we have
\begin{equation}\label{est1.epsilon}
\sup\limits_{t\leq \tau\leq T_0}\left\|\pa^{\al'}\bh_\vps\right\|^2_{L^\infty_xL^2_{v}}\leq C_\vps \left\|\bh\right\|^2_{ L^1_{\bar{k}}L^\infty_{T_0}H^1_{x_1}L^2_v},
\end{equation}
and for later use we additionally have that
\begin{equation}\label{est2.epsilon}
\int_{t}^{T_0}\left\|\pa^{\al'}\bh_\vps\right\|^2_{L^\infty_xL^2_{v,D}}d\tau\leq C_\vps
\left\|\bh\right\|^2_{L^1_{\bar{k}}L^2_{T_0}H^1_{x_1}L^2_{v,D}},
\end{equation}
where $C_\vps >0$ above may depend on $\vps$.  The terms in the upper bound above are finite due to \eqref{lb.loc.a2}.

Now we will perform an estimate similar to \eqref{cg.ij.noD} but including derivatives: $\partial^\alpha$.  A key point is that the estimates \eqref{est1.epsilon} and \eqref{est2.epsilon} depend crucially upon $\vps>0$.  However we need a fixed uniform smallness of $\epsilon_0>0$ in \eqref{lb.loc.a2} and we can not allow $\epsilon_0\to 0$ as $\vps\to0$.   Therefore we will carefully iterate the derivatives such that we only use  \eqref{lb.loc.a2} when we use the uniform in $\vps>0$ estimates \eqref{est1.epsilon0} and \eqref{est2.epsilon0} with no derivatives on $\bh_\vps$.  When we have derivatives on $\bh_\vps$ such as in \eqref{est1.epsilon} and \eqref{est2.epsilon} then we instead use the previous steps in the iteration to close our estimates.

In the case $|\alpha| =1$, using \eqref{gip} and taking one derivative we can again apply \eqref{L1L2.es} to obtain for $0<t< T_0$ that 
\begin{multline}\notag 
\|\pa^\al g_{\vps}(t)\|^2+\la \int_t^{T_0}\|\pa^{\al}g_{\vps}\|^2_{D}d\tau
\lesssim 
\int_{t}^{T_0} \left| \CR(\pa^{\al}\CG g_{\vps},\pa^\al g_{\vps}) \right| ~ d\tau
\\
+
\int_{t}^{T_0}  \left| \CR(\pa^{\al} g_{\vps},\Ga(\bh_\vps,\pa^\al g_{\vps}))  \right| ~ d\tau
+
\int_{t}^{T_0}  \left| \CR(g_{\vps},\Ga(\pa^{\al} \bh_\vps,\pa^\al g_{\vps}))  \right| ~ d\tau
\\
+
T_0\sup\limits_{t\leq \tau\leq T_0}\|\pa^\al g_{\vps}(\tau)\|^2.
\end{multline}
It further follows from Lemma \ref{bnp.es} with $w_{q,\vartheta}=1$ and Sobolev embedding  that 
\begin{multline}\notag 
\|\pa^\al g_{\vps}(t)\|^2+\la \int_t^{T_0}\|\pa^{\al}g_{\vps}\|^2_{D}d\tau
\lesssim 
\sup\limits_{t\leq \tau\leq T_0}
\left\|\pa^{\al} g_{\vps}(\tau)\right\|
\int_{t}^{T_0}
\left\|\pa^{\al}\CG g_{\vps}(\tau)\right\|d\tau
\\
+
C_\eta
\int_{t}^{T_0}\left\|\pa^{\al} g_{\vps}\right\|^2_{D}
\left\|\bh_\vps\right\|^2_{L^\infty_xL^2_{v}}
d\tau
+
C_\eta
\int_{t}^{T_0}\left\|\pa^{\al} g_{\vps}\right\|^2 \left\|\bh_\vps\right\|^2_{L^\infty_xL^2_{v,D}}d\tau
\\
+
C_\eta
\int_{t}^{T_0}\left\| g_{\vps}\right\|^2_{D}
\left\|\pa^{\al}\bh_\vps\right\|^2_{L^\infty_xL^2_{v}}
d\tau
+
C_\eta
\int_{t}^{T_0}\left\| g_{\vps}\right\|^2 \left\|\pa^{\al}\bh_\vps\right\|^2_{L^\infty_xL^2_{v,D}}d\tau
\\
+\eta\int_t^{T_0}
\left\|\pa^{\al} g_{\vps}\right\|_D^2d\tau
+
T_0\sup\limits_{t\leq \tau\leq T_0}\|\pa^\al g_{\vps}(\tau)\|^2,
\end{multline}
which holds for any $0<t<T_0$.    We conclude from \eqref{lb.loc.a2}, \eqref{est1.epsilon0} and \eqref{est2.epsilon0}  and the uniform smallness of $\epsilon_0>0$ that
\begin{multline}\notag 
\|\pa^\al g_{\vps}(t)\|^2+\la \int_t^{T_0}\|\pa^{\al}g_{\vps}\|^2_{D}d\tau
\lesssim 
\int_{t}^{T_0}
\left\|\pa^{\al}\CG g_{\vps}(\tau)\right\|d\tau
\\
+
\int_{t}^{T_0}\left\| g_{\vps}\right\|^2_{D}
\left\|\pa^{\al}\bh_\vps\right\|^2_{L^\infty_xL^2_{v}}
d\tau
+
\int_{t}^{T_0}\left\| g_{\vps}\right\|^2 \left\|\pa^{\al}\bh_\vps\right\|^2_{L^\infty_xL^2_{v,D}}d\tau.
\end{multline}
Next using \eqref{est1.epsilon} and \eqref{est2.epsilon} (in this next step we are not using the smallness of $\epsilon_0>0$) we take a suitable linear combination of the last estimate with \eqref{cg.ij.noD}  to obtain that
\begin{equation}\label{cg.ij.1D}
\sum_{|\alpha| \le 1}
\|\pa^\al g_{\vps}(t)\|^2
+
\la \sum_{|\alpha| \le 1}\int_t^{T_0}\|\pa^{\al}g_{\vps}\|^2_{D}d\tau
\lesssim 
\sum_{|\alpha| \le 1}
\int_{t}^{T_0}
\left\|\pa^{\al}\CG g_{\vps}(\tau)\right\|d\tau.
\end{equation}
In particular we multiply \eqref{cg.ij.noD}  by a suitable large constant and then add it to the previous line to obtain \eqref{cg.ij.1D}.

In general, for $\al = (\al_1, \bar{\al})$ and $\bar{\al}= (\bar{\al}_2, \bar{\al}_3)$ we will prove that the following
\begin{multline}\label{cg.ij}
\sum\limits_{\al_1\leq1,|\bar{\al}|\leq6\atop{|\al|  \leq m}}
\|\pa^\al g_{\vps}(t)\|^2
+
\la \sum\limits_{\al_1\leq1,|\bar{\al}|\leq6\atop{|\al|  \leq m}} \int_t^{T_0}\|\pa^{\al}g_{\vps}\|^2_{D}d\tau
\\
\lesssim 
\sum\limits_{\al_1\leq1,|\bar{\al}|\leq6\atop{|\al|  \leq m}}
\int_{t}^{T_0}
\left\|\pa^{\al}\CG g_{\vps}(\tau)\right\|d\tau.
\end{multline}
holds for $m\in \{2,\ldots,7\}$.  The $m=0$ case was shown in \eqref{cg.ij.noD} and the $m=1$ case was shown in \eqref{cg.ij.1D}.

We assume that \eqref{cg.ij} holds for $|\alpha|=m$ and then we prove that it further holds for $|\alpha|=m+1$.   
Consider any fixed $\alpha$ with $|\alpha| =m+1$, again using \eqref{gip} and  \eqref{L1L2.es} we obtain for $0<t< T_0$ that 
\begin{multline}\notag 
\|\pa^\al g_{\vps}(t)\|^2+\la \int_t^{T_0}\|\pa^{\al}g_{\vps}\|^2_{D}d\tau
\lesssim 
\int_{t}^{T_0} \left| \CR(\pa^{\al}\CG g_{\vps},\pa^\al g_{\vps}) \right| ~ d\tau
\\
+
\int_{t}^{T_0}  \left| \CR(\pa^{\al} g_{\vps},\Ga(\bh_\vps,\pa^\al g_{\vps}))  \right| ~ d\tau
+
\int_{t}^{T_0}  
\sum\limits_{\al' < \al}
\left| \CR(\pa^{\al'} g_{\vps},\Ga(\pa^{\al-\al'} \bh_\vps,\pa^\al g_{\vps}))  \right| ~ d\tau
\\
+
T_0\sup\limits_{t\leq \tau\leq T_0}\|\pa^\al g_{\vps}(\tau)\|^2.
\end{multline}
It similarly follows from Lemma \ref{bnp.es}  and the Sobolev embedding  that 
\begin{multline}\notag 
\|\pa^\al g_{\vps}(t)\|^2+\la \int_t^{T_0}\|\pa^{\al}g_{\vps}\|^2_{D}d\tau
\lesssim 
\sup\limits_{t\leq \tau\leq T_0}
\left\|\pa^{\al} g_{\vps}(\tau)\right\|
\int_{t}^{T_0}
\left\|\pa^{\al}\CG g_{\vps}(\tau)\right\|d\tau
\\
+
C_\eta
\int_{t}^{T_0}\left\|\pa^{\al} g_{\vps}\right\|^2_{D}
\left\|\bh_\vps\right\|^2_{L^\infty_xL^2_{v}}
d\tau
+
C_\eta
\int_{t}^{T_0}\left\|\pa^{\al} g_{\vps}\right\|^2 \left\|\bh_\vps\right\|^2_{L^\infty_xL^2_{v,D}}d\tau
\\
+
C_\eta
\int_{t}^{T_0}
\sum\limits_{\al' < \al}
\left( \left\| \pa^{\al'}g_{\vps}\right\|^2 \left\|\pa^{\al-\al'}\bh_\vps\right\|^2_{L^\infty_xL^2_{v,D}}
+
\left\| \pa^{\al'}g_{\vps}\right\|^2_{D}
\left\|\pa^{\al-\al'}\bh_\vps\right\|^2_{L^\infty_xL^2_{v}}
\right) 
d\tau
\\
+\eta\int_t^{T_0}
\left\|\pa^{\al} g_{\vps}\right\|_D^2d\tau
+
T_0\sup\limits_{t\leq \tau\leq T_0}\|\pa^\al g_{\vps}(\tau)\|^2,
\end{multline}
which holds for any $0<t<T_0$.   Again using \eqref{lb.loc.a2}, \eqref{est1.epsilon0} and \eqref{est2.epsilon0}  and the uniform smallness of $\epsilon_0>0$ we obtain
\begin{multline}\notag 
\|\pa^\al g_{\vps}(t)\|^2+\la \int_t^{T_0}\|\pa^{\al}g_{\vps}\|^2_{D}d\tau
\lesssim 
\int_{t}^{T_0}
\left\|\pa^{\al}\CG g_{\vps}(\tau)\right\|d\tau
\\
+
\int_{t}^{T_0}
\sum\limits_{\al' < \al}
\left( \left\| \pa^{\al'}g_{\vps}\right\|^2 \left\|\pa^{\al-\al'}\bh_\vps\right\|^2_{L^\infty_xL^2_{v,D}}
+
\left\| \pa^{\al'}g_{\vps}\right\|^2_{D}
\left\|\pa^{\al-\al'}\bh_\vps\right\|^2_{L^\infty_xL^2_{v}}
\right) 
d\tau.
\end{multline}
Next using \eqref{est1.epsilon} and \eqref{est2.epsilon} (and not relying on the smallness of $\epsilon_0>0$), since $\al' < \al$ we take a suitable linear combination of above last estimate with \eqref{cg.ij} for $|\alpha|=m$  to obtain that \eqref{cg.ij} also holds for $|\alpha|=m+1$.  In particular again we multiply \eqref{cg.ij} for $|\alpha|=m$   by a suitable large constant and then add it to the previous line to obtain \eqref{cg.ij} for $|\alpha|=m+1$, which proves \eqref{cg.ij}.

Furthermore, we define the codomain of the mapping $\CG$ on $\W_1$ as
\begin{equation*}
\begin{split}
\W_2=&\big\{w:w=\CG g, \ g\in \W_1\big\}\subset L^1(0,T_0; L^2(\Om\times\R^3)),
\end{split}
\end{equation*}
and we define the  functional
\begin{equation*}
\begin{split}
&\CM:\ \ \W_2\rightarrow \C,\\
& w_\vps=\CG h_{\vps}\mapsto (h_{\vps}(0),g_{0,\vps})
-\int_{0}^{T_0}(\CL_2\bh_{\vps},h_{\vps})dt
=\CM(w_\vps),
\end{split}
\end{equation*}
where $h_{\vps}\in\W_1$ is uniquely determined by $w_\vps\in\W_2$ as $\CG:\W_1\to\W_2$ is bijective.
We note using \eqref{cg.ij.noD}  we have that
\begin{align}
|\CM(w_{\vps})|&\leq \|h_{\vps}(0)\|\|g_{0,\vps}\|
+C\|\bh_{\vps}\|_{L^\infty(0,T_0;L^2(\Om\times\R^3))}\|h_{\vps}\|_{L^1(0,T_0;L^2(\Om\times\R^3))}\notag\\
&\leq C_{T_0}(\vps)\left(\|g_{0,\vps}\|
+\|\bh_{\vps}\|_{L^\infty(0,T_0;L^2(\Om\times\R^3))}\right)\|\CG h_{\vps}\|_{L^1(0,T_0;L^2(\Om\times\R^3))}.\label{CMbd0}
\end{align}
This implies that $\CM:\W_2\to\C$ can be extended to be  a bounded linear functional on  $L^1(0,T_0;L^2(\Om\times\R^3))$.
So,  using the Hahn-Banach Theorem, there is $g_{\vps}\in L^\infty(0,T_0;L^2(\Om\times\R^3))$ such that
\begin{equation}
\label{lite.p1}
\CM(w_{\vps})=\int_0^{T_0}( g_{\vps}(t), w_{\vps}(t))dt,\quad \forall\,w_{\vps}\in L^1(0,T_0;L^2(\Om\times\R^3),
\end{equation}
and it holds that
$$
\|g_{\vps}\|_{L^\infty(0,T_0;L^2(\Om\times\R^3))}\leq C_{T_0}(\vps)\left(\|g_{0,\vps}\|
+\|\bh_{\vps}\|_{L^\infty(0,T_0;L^2(\Om\times\R^3))}\right).
$$
Recall \eqref{def.W1}. In terms of \eqref{lite.p1}, for any $h_\vps\in \W_1$, it  follows that
\begin{equation}\label{wf.g}
\begin{split}
\CM(\CG  h_{\vps})=&\int_0^{T_0}(g_{\vps}(t), \CG  h_{\vps})dt=(h_{\vps}(0),g_{0,\vps})
-\int_{0}^{T_0}(\CL_2\bh_{\vps},h_{\vps})dt,
\end{split}
\end{equation}
which implies that $g_\vps\in L^\infty(0,T_0;L^2(\Om\times\R^3))$ is a weak solution to \eqref{lc.lb.vps}. Moreover, from the energy estimate \eqref{cg.ij.noD},  it is straightforward to verify $g_\vps\in L^2_{T_0}L_x^2L^2_{v,D}$.  

\medskip
\noindent{\it Step 3.}
In this step, we are going to show that the obtained weak solution $g_\vps$ possesses higher regularity such as 
\begin{equation}
\label{wdal}
\pa^{\al}g_\vps\in L^\infty_{T_0}L_x^2 L^2_v\cap L^2_{T_0}L_x^2L^2_{v,D}
\end{equation}
for $\al\in\CA$ with 
$$
\CA=\{\al:\al=(\al_1,\bar{\al}), \al_1\leq1, |\bar{\al}|\leq6\}.
$$
The proof uses an induction on $n=|\al|$ with the help of the weak formulation in \eqref{lc.lb.vps}.  The main issue to be dealt with arises from the fact that $\pa^\al$ does not commute with $\CG$ since the linear operator $\Ga(\bh_\vps,\cdot)$ depends on the function $\bh_\vps(t,x,v)$. 

Indeed, for $|\al|=1$, thanks to \eqref{CG}, we can still obtain a
bounded linear functional on $L^1(0,T_0;L^2(\Om\times\R^3))$,
given by
\begin{equation*}
\begin{split}
&\CM_\al:\ \ \W_2\rightarrow \C,\\
& w_\vps=\CG h_{\vps}\mapsto -(\pa^{\al}h_{\vps}(0),g_{0,\vps})-
\int_{0}^{T_0}(\CL_2\pa^{\al}\bh_{\vps},h_{\vps})dt+\int_{0}^{T_0}(\Ga(\pa^{\al}\bh_\vps,g_{\vps}),h_{\vps})dt.
\end{split}
\end{equation*}
The above uses $g_{\vps}$ from \eqref{lite.p1}.   Note further that for $|\alpha| =1$ we have
\begin{equation}\label{CM.comparison}
\CM(\pa^{\al} \CG h_{\vps}) = \CM( \CG \pa^{\al} h_{\vps}) - 
\int_{0}^{T_0}(g_{\vps}, \Ga^*(\pa^{\al}\bh_\vps, h_{\vps})) dt
=
- \CM_\al (\CG h_{\vps}).
\end{equation}
Moreover, similar to obtaining \eqref{CMbd0}, after integrating by parts on the initial data term and using \eqref{cg.ij.noD}, one has
\begin{align*}
|\CM_\al(w_{\vps})|&\leq \|h_{\vps}(0)\|\|\pa^{\al} g_{0,\vps}\|
+C\|\pa^{\al}\bh_{\vps}\|_{L^\infty(0,T_0;L^2(\Om\times\R^3))}\|h_{\vps}\|_{L^1(0,T_0;L^2(\Om\times\R^3))}\notag\\
&\qquad+C\|\pa^{\al}\bh_{\vps}\|_{L^\infty(0,T_0;H^2_xL^2_v)}\|g_{\vps}\|_{L^2_{T_0}L_x^2L^2_{v,D}(\R^3)}
\|h_{\vps}\|_{L^2_{T_0}L_x^2L^2_{v,D}(\R^3)}\notag\\
&\qquad+C\|g_{\vps}\|_{L^\infty(0,T_0;L^2(\Om\times\R^3))}\|\pa^{\al}\bh_{\vps}\|_{L^2_{T_0}H_x^2L^2_{v,D}(\R^3)}
\|h_{\vps}\|_{L^2_{T_0}L_x^2L^2_{v,D}(\R^3)}\notag\\
&\leq C_{T_0}
\| \CG h_{\vps}\|_{L^1_{T_0}L^2_{x,v}} \Big(\|\pa^{\al} g_{0,\vps}\|
+\|\pa^{\al}\bh_{\vps}\|_{L^\infty(0,T_0;L^2(\Om\times\R^3))}
\notag\\
&\qquad\qquad+\|\pa^{\al}\bh_{\vps}\|_{L^\infty(0,T_0;H^2_xL^2_v)}\|g_{\vps}\|_{L^2_{T_0}L_x^2L^2_{v,D}(\R^3)}\notag\\
&\qquad\qquad
+\|g_{\vps}\|_{L^\infty(0,T_0;L^2(\Om\times\R^3))}\|\pa^{\al}\bh_{\vps}\|_{L^2_{T_0}H_x^2L^2_{v,D}(\R^3)}\Big)
.
\end{align*}
Therefore, as before, there is $g^{\al}_\vps\in L^\infty(0,T_0;L^2(\Om\times\R^3))$ such that
\begin{equation}\label{lite.p2}
\begin{split}
\CM_\al (w_{\vps})
=
\int_0^{T_0}( g^{\al}_{\vps}(t), w_{\vps}(t))dt,\quad 
\forall\,w_{\vps}\in L^1(0,T_0;L^2(\Om\times\R^3)),
\end{split}
\end{equation}
and similarly to before
\begin{multline*}
\|g^{\al}_{\vps}\|_{L^\infty(0,T_0;L^2(\Om\times\R^3))}
\leq C_{T_0}\Big( \|\pa^{\al} g_{0,\vps}\|
+\|\pa^{\al}\bh_{\vps}\|_{L^\infty(0,T_0;L^2(\Om\times\R^3))}\\
+\|\pa^{\al}\bh_{\vps}\|_{L^\infty(0,T_0;H^2_xL^2_v)}\|g_{\vps}\|_{L^2_{T_0}L_x^2L^2_{v,D}(\R^3)}
\\
+\|g_{\vps}\|_{L^\infty(0,T_0;L^2(\Om\times\R^3))}\|\pa^{\al}\bh_{\vps}\|_{L^2_{T_0}H_x^2L^2_{v,D}(\R^3)}\Big).
\end{multline*}
Note that \eqref{lite.p2} then reads as 
\begin{align}
-\int_0^{T_0}&(\pa_th_{\vps}, g^{\al}_{\vps})+\int_0^{T_0}(v\cdot\na_{x}h_{\vps},g^{\al}_{\vps})
+\int_0^{T_0}(\mathscr{L}_1h_{\vps},g^{\al}_{\vps})-\int_0^{T_0}(\Ga(\bh_\vps,g^{\al}_{\vps}),h_{\vps})\notag\\
&=-(\pa^{\al}h_{\vps}(0),g_{0,\vps})-
\int_0^{T_0}(\mathscr{L}_2\pa^{\al}\bh_\vps,g^{\al}_{\vps})+\int_0^{T_0}(\Ga(\pa^{\al}\bh_\vps,g_{\vps}),g^{\al}_{\vps}),\label{lite.p3}
\end{align}
meaning that $g^{\al}_{\vps}$ is a weak solution to the following problem:
\begin{eqnarray*}
\left\{\begin{array}{rll}
&\pa_t g^{\al}_{\vps}+v\cdot\na_{x}g^{\al}_{\vps}+\mathscr{L}_1g^{\al}_{\vps}-\Ga(\bh_\vps,g^{\al}_{\vps})=
-\mathscr{L}_2\pa^{\al}\bh_\vps+\Ga(\pa^{\al}\bh_\vps,g_{\vps}),\\
&g^{\al}_{\vps}(0,x,v)=\pa^{\al}g_{0,\vps}(x,v),\\
&g_{\vps}(t,-1,\bar{x},v_1,\bar{v})|_{v_1>0}=g_{\vps}(t,-1,\bar{x},-v_1,\bar{v}),\\
&g_{\vps}(t,1,\bar{x},v_1,\overline{v})|_{v_1<0}=g_{\vps}(t,1,\bar{x},-v_1,\bar{v}).
\end{array}\right.
\end{eqnarray*}
Then using \eqref{lite.p1} with \eqref{CM.comparison} and \eqref{lite.p2}, 
we observe that
\begin{equation*}
\begin{split}
\int_0^{T_0}(g^{\al}_{\vps},\CG h_\vps)dt=-\int_0^{T_0}(g_{\vps},\pa^{\al}\CG h_\vps)dt,
\end{split}
\end{equation*}
for any $h_\vps$, so that $g^{\al}_{\vps}=\pa^{\al}g_\vps$ with $|\al|=1$ in the weak sense. This proves \eqref{wdal} with $|\al|=1$.

Now we assume that \eqref{wdal} is true
for $\al\in\CA$ with $|\al|=m\geq1$.
Letting  $\al\in\CA$ with $|\al|=m+1$, as shown before,  one sees that
\begin{equation*}
\begin{split}
&\CM_{\al}:\ \ \W_2\rightarrow \C,\\
& w_\vps=\CG h_{\vps}\mapsto (-1)^{|\al|}(\pa^{\al}h_{\vps}(0),g_{0,\vps})-
\int_{0}^{T_0}(\CL_2\pa^{\al}\bh_{\vps},h_{\vps})dt\\&\qquad\qquad\qquad\qquad+
\sum\limits_{\al>\al'\geq0}\tilde{C}^{\alpha}_{\alpha'}
\int_{0}^{T_0}(\Ga(\pa^{\al-\al'}\bh_\vps,\pa^{\al'}g_{\vps}),h_{\vps})dt,
\end{split}
\end{equation*}
can be extended to be a bounded linear functional on $L^\infty(0,T_0;L^2(\Om\times\R^3))$. Here, since $|\al'|\leq m$, the induction assumption has been used.  Then we also have
\begin{multline}\label{CM.comparison.H}
\CM(\pa^{\al} \CG h_{\vps}) = \CM( \CG \pa^{\al} h_{\vps}) - 
\sum\limits_{\al>\al'\geq0} C^{\alpha}_{\alpha'}
\int_{0}^{T_0}(g_{\vps}, \Ga^*(\pa^{\al-\al'}\bh_\vps, \pa^{\al'}h_{\vps})) dt
\\
=
(-1)^{|\al|} \CM_\al (\CG h_{\vps}).
\end{multline}
In \eqref{CM.comparison.H} and above $C^{\alpha}_{\alpha'}$ and $\tilde{C}^{\alpha}_{\alpha'}$ are the constants that arise from the multinomial formula and they depend only on $\alpha$ and $\alpha'$.

Then, similarly, there is $g^{\al}_\vps\in L^\infty(0,T_0;L^2(\Om\times\R^3))$ such that
\begin{equation*}
\begin{split}
\CM_\al (w_{\vps})=\int_0^{T_0}( g^{\al}_{\vps}(t), w_{\vps}(t))dt,\quad \forall\,w_{\vps}\in L^1(0,T_0;L^2(\Om\times\R^3)),
\end{split}
\end{equation*}
and again  we have
\begin{equation*}
\begin{split}
&\|g^{\al}_{\vps}\|_{L^\infty(0,T_0;L^2(\Om\times\R^3)}\\
&\leq C_{T_0}\Big(\|\partial^{\al} g_{0,\vps}\|
+\|\pa^{\al}\bh_{\vps}\|_{L^\infty(0,T_0;L^2(\Om\times\R^3))}
\notag\\
&\qquad\qquad+
\sum\limits_{\al>\al'\geq0}\|\pa^{\al-\al'}\bh_{\vps}\|_{L^\infty(0,T_0;H^2_xL^2_v)}
\|\pa^{\al'}g_{\vps}\|_{L^2_{T_0}L_x^2L^2_{v,D}(\R^3)}
\\&
\qquad\qquad
+\sum\limits_{\al>\al'\geq0}\|\pa^{\al'}g_{\vps}\|_{L^\infty(0,T_0;L^2(\Om\times\R^3))}\|\pa^{\al-\al'}\bh_{\vps}\|_{L^2_{T_0}H_x^2L^2_{v,D}(\R^3)}\Big).
\end{split}
\end{equation*}
Consequently, similar to the previous arguments, one sees that $g^{\al}_{\vps}$ with $\al\in \CA$ and $|\al|=m+1$ is a weak solution to the following problem
\begin{eqnarray*}
\left\{\begin{array}{rll}
&\pa_t g^{\al}_{\vps}+v\cdot\na_{x}g^{\al}_{\vps}+\mathscr{L}_1g^{\al}_{\vps}-\Ga(\bh_\vps,g^{\al}_{\vps})=
-\mathscr{L}_2\pa^{\al}\bh_\vps+\sum\limits_{\al>\al'\geq0}
\Ga(\pa^{\al-\al'}\bh_\vps,\pa^{\al'}g_{\vps}),\\
&g^{\al}_{\vps}(0,x,v)=\pa^{\al}g_{0,\vps}(x,v),\\
&g_{\vps}(-1,\bar{x},v_1,\bar{v})|_{v_1>0}=g_{\vps}(-1,\bar{x},-v_1,\bar{v}),\\
&g_{\vps}(1,\bar{x},v_1,\bar{v})|_{v_1<0}=g_{\vps}(1,\bar{x},-v_1,\bar{v}),
\end{array}\right.
\end{eqnarray*}
and furthermore as in \eqref{CM.comparison.H},  $g^{\al}_{\vps}=\pa^{\al}g_\vps$ holds true in the weak sense as
\begin{equation*}
\begin{split}
\int_0^{T_0}(g^{\al}_{\vps},\CG h_\vps)dt=(-1)^{m+1}\int_0^{T_0}(g_{\vps},\pa^{\al}\CG h_\vps)dt,
\end{split}
\end{equation*}
for any $h_\vps$.
Hence, one has that $\pa^{\al}g_\vps\in L^\infty(0,T_0;L^2(\Om\times\R^3))$, and it is also straightforward to  further verify that $\pa^{\al}g_\vps\in L^2_{T_0}L_x^2L^2_{v,D}$. This justifies  \eqref{wdal}
for $\al\in\CA$ with $|\al|=m+1$. Thus, by induction,  \eqref{wdal} is proved.  Note that in the above steps many of the estimates depended crucially on fixed $\vps>0$.

\medskip
\noindent{\it Step 4.}
We are going to show that there is $g=g(t,x,v)$, $0<t<T_0$, $x\in\Omega$, $v\in \R^3$ with $g(t,x_1,\bar{x},v_1,\bar{v})=g(t,-x_1,\bar{x},-v_1,\bar{v})$ such that
\begin{equation*}
g,\ \na_x g\in L^1_{\bar{k}}L^\infty_{T_0}L^2_{x_1}L^2_v\cap L^1_{\bar{k}}L^2_{T_0}L^2_{x_1}L^2_{v,D},
\end{equation*}
and
for any $|\al|\leq 1$, $\pa^{\al}g_\vps$ strongly converges to $\pa^\al g$ in 
the above space
as $\vps\rightarrow0^+$.

Firstly, starting from the approximate solution $g_\vps$ with the estimate \eqref{wdal} in the regular Sobolev space, we need to prove that for any $|\al|\leq 1$, we have that
\begin{equation*}
\pa^{\al}g_\vps \in L^1_{\bar{k}}L^\infty_{T_0}L^2_{x_1}L^2_v\cap L^1_{\bar{k}}L^2_{T_0}L^2_{x_1}L^2_{v,D}.
\end{equation*}
Indeed, letting $|\al|\leq 1$, to prove the estimates on $\pa^\al g_\vps$ in $L^1_{\bar{k}}L^\infty_{T_0}L^2_{x_1}L^2_v$,  we split the whole Fourier space $\Z^2_{\bar{k}}$ as $\{\bar{k}\neq 0\}\cup\{\bar{k}=0\}$, and write
\begin{align}
\|\pa^{\al}g_\vps\|_{L^1_{\bar{k}}L^\infty_{T_0}L^2_{x_1}L^2_v}&=\int_{\Z^2}\sup\limits_{0\leq t\leq T_0}\|\widehat{\pa^{\al}g}_\vps\|d\Si(\bar{k})\notag\\
&=
\int_{\bar{k}\neq0}\sup\limits_{0\leq t\leq T_0}\|\widehat{\pa^{\al}g_\vps}\|d\Si(\bar{k})
+\sup\limits_{0\leq t\leq T_0}\left\|\int_{\T^2}\pa^{\al}g_\vps d\bar{x}\right\|.
\label{lite.p4}
\end{align}
Above, and in the rest of  {\it Step 4}, we use the notation $\| \cdot \| = \| \cdot \|_{L^2_{x_1}L^2_{v}}$.

\medskip
\noindent{\it Estimate on the first term on the right-hand side of \eqref{lite.p4}}: To treat the sup norm in time we may rewrite it as
\begin{equation}
\label{lite.p40}
\sup\limits_{0\leq t\leq T_0}|\cdot|=\lim\limits_{p\to\infty} \|\cdot\|_{L^p_{T_0}}=\liminf\limits_{p\to\infty} \|\cdot\|_{L^p_{T_0}}.
\end{equation}
Let $\widehat{\Delta^{b}_{\bar{x}}g} := |\bar{k} |^b \hat{g}$ for $b\in \R$.  Then it holds that
\begin{align}
\int_{\bar{k}\neq0}\sup\limits_{0\leq t\leq T_0}\|\widehat{\pa^{\al}g_\vps}\|d\Si(\bar{k}) &=\int_{\bar{k}\neq0}\frac{1}{|\bar{k}|^{5/2}}\sup\limits_{0\leq t\leq T_0}\|\widehat{\Delta^{5/2}_{\bar{x}}\pa^{\al}g_\vps}\|d\Si(\bar{k})\notag\\
&=\int_{\bar{k}\neq0}\frac{1}{|\bar{k}|^{5/2}}\liminf\limits_{p\rightarrow\infty}
\left\|\|\widehat{\Delta^{5/2}_{\bar{x}}\pa^{\al}g_\vps}\|\right\|_{L^p_{T_0}}d\Si(\bar{k})\notag\\
&\leq \liminf\limits_{p\rightarrow\infty}\int_{\bar{k}\neq0}\frac{1}{|\bar{k}|^{5/2}}
\left\|\|\widehat{\Delta^{5/2}_{\bar{x}}\pa^{\al}g_\vps}\|\right\|_{L^p_{T_0}}d\Si(\bar{k}),\label{lite.p5}
\end{align}
where Fatou's lemma has been used in the last inequality. Let $p\geq 2$ with the conjugate $p'$ such that $1/p+1/p'=1$, then we can derive from H\"older's inequality that
\begin{align*}
&\int_{\bar{k}\neq0}\frac{1}{|\bar{k}|^{5/2}}
\left\|\|\widehat{\Delta^{5/2}_{\bar{x}}\pa^{\al}g_\vps}\|\right\|_{L^p_{T_0}}d\Si(\bar{k})\\
&\leq \left(\int_{\bar{k}\neq0}\frac{1}{|\bar{k}|^{5p'/2}}d\Si(\bar{k})\right)^{1/p'}
\left(\int_{\bar{k}\neq0}
\left\|\|\widehat{\Delta^{5/2}_{\bar{x}}\pa^{\al}g_\vps}\|\right\|^p_{L^p_{T_0}}d\Si(\bar{k})\right)^{1/p}\\
&\leq C_{p'}\cdot T_0^{1/p} \sup_{0\leq t\leq T_0}\left(\int_{\Z^2}
\|\widehat{\Delta^{5/2}_{\bar{x}}\pa^{\al}g_\vps}\|^pd\Si(\bar{k})\right)^{1/p}.
\end{align*}
Applying the inequality $\|\cdot\|_{\ell^p(\Z^2)}\leq \|\cdot\|_{\ell^2(\Z^2)}$ for any $2\leq p\leq \infty$, we see that
\begin{equation*}
\left(\int_{\Z^2}
\|\widehat{\Delta^{5/2}_{\bar{x}}\pa^{\al}g_\vps}\|^pd\Si(\bar{k})\right)^{1/p}\leq \left(\int_{\Z^2}
\|\widehat{\Delta^{5/2}_{\bar{x}}\pa^{\al}g_\vps}\|^2d\Si(\bar{k})\right)^{1/2},
\end{equation*}
which by Plancherel's identity the last upper bound is further bounded as
\begin{equation*}
\|\Delta^{5/2}_{\bar{x}}\pa^{\al}g_\vps\|_{L^2_{x,v}}\leq \sum\limits_{\al_1\leq1,|\bar{\al}|\leq6}\left\|\pa^{\al} g_{\vps}\right\|_{L^2_{x,v}}.
\end{equation*}
Plugging those estimates above into \eqref{lite.p5} and taking the $\liminf$ as $p\to\infty$ gives 
\begin{equation}
\label{lite.p50}
\int_{\bar{k}\neq0}\sup\limits_{0\leq t\leq T_0}\|\widehat{\pa^{\al}g_\vps}\|d\Si(\bar{k})\leq C \sum\limits_{\al_1\leq1,|\bar{\al}|\leq6} \sup_{0\leq t\leq T_0}\left\|\pa^{\al} g_{\vps}\right\|_{L^2_{x,v}}\leq C_\vps<\infty.
\end{equation}

\medskip
\noindent
{\it Estimate on the second term on the right-hand side of \eqref{lite.p4}}: We first apply Minkowski's inequality to obtain
\begin{equation*}
\sup\limits_{0\leq t\leq T_0}\left\|\int_{\T^2}\pa^{\al}g_\vps d\bar{x}\right\|\leq\sup\limits_{0\leq t\leq T_0}\int_{\T^2}\|\pa^{\al}g_\vps\| d\bar{x} \leq \int_{\T^2}\sup\limits_{0\leq t\leq T_0}\|\pa^{\al}g_\vps\| d\bar{x}.
\end{equation*}
Further by \eqref{lite.p40} and Fatou's lemma, one has
\begin{align}
\int_{\T^2}\sup\limits_{0\leq t\leq T_0}\|\pa^{\al}g_\vps\| d\bar{x}= \int_{\T^2}\liminf\limits_{p\rightarrow\infty}
\Big\|\|\pa^{\al}g_\vps\|\Big\|_{L^p_{T_0}}d\bar{x}
\leq \liminf\limits_{p\rightarrow\infty} \int_{\T^2}
\Big\|\|\pa^{\al}g_\vps\|\Big\|_{L^p_{T_0}}d\bar{x}.\label{lite.p6}
\end{align}
Similarly to before, for any $p\geq 2$ with the conjugate $p'$,  by H\"older's inequality it holds that
\begin{align*}
\int_{\T^2}
\Big\|\|\pa^{\al}g_\vps\|\Big\|_{L^p_{T_0}}d\bar{x} &\leq \left(\int_{\T^2}1^{p'}d\bar{x}\right)^{1/p'}
\left(\int_{\T^2}\Big\|\|\pa^{\al}g_\vps\|\Big\|^p_{L^p_{T_0}}d\bar{x}\right)^{1/p}\\&=C_{p'}\left(\int_{\T^2}\int_0^{T_0}\left\|\pa^{\al}g_\vps\right\|^pdtd\bar{x}\right)^{1/p},
\end{align*}
which implies that
\begin{align*}
\liminf\limits_{p\rightarrow\infty} \int_{\T^2}
\Big\|\|\pa^{\al}g_\vps\|\Big\|_{L^p_{T_0}}d\bar{x} &\leq \liminf\limits_{p\rightarrow\infty}  C_{p'}\left(\int_{\T^2}\int_0^{T_0}\left\|\pa^{\al}g_\vps\right\|^pdtd\bar{x}\right)^{1/p}
\\&=C\sup\limits_{0\leq t\leq T_0}\sup_{\bar{x}\in \T^2}\left\|\pa^{\al}g_\vps\right\|.
\end{align*}
By the embedding $H^2(\T_{\bar{x}}^2)\subset L^\infty(\T_{\bar{x}}^2)$, plugging the above estimate into \eqref{lite.p6}, we have
\begin{equation}
\label{lite.p7}
\int_{\T^2}\sup\limits_{0\leq t\leq T_0}\|\pa^{\al}g_\vps\| d\bar{x} \leq C \sum\limits_{\al_1\leq1,|\bar{\al}|\leq6} \sup\limits_{0\leq t\leq T_0}\left\|\pa^{\al} g_{\vps}\right\|_{L^2_{x,v}}\leq C_\vps<\infty.
\end{equation}
Therefore, by \eqref{lite.p7} and \eqref{lite.p50}, it follows from \eqref{lite.p4} that  $\pa^{\al}g_\vps \in L^1_{\bar{k}}L^\infty_{T_0}L^2_{x_1}L^2_v$ for $|\al|\leq 1$.  We have the following estimate:
\begin{align*}
&\sum\limits_{|\al|\leq1}\|\pa^{\al}g_\vps\|_{L^1_{\bar{k}}L^2_{T_0}L^2_{x_1}L^2_v}=\sum\limits_{|\al|\leq1}\int_{\Z^2}
\left(\int_0^{T_0}\|\widehat{\pa^{\al}g}_\vps\|^2dt\right)^{1/2}d\Si(\bar{k})\notag \\
&=
\sum\limits_{|\al|\leq1}\int_{\bar{k}\neq0}\frac{1}{|\bar{k}|^{3/2}}
\left(\int_0^{T_0}\|\widehat{\Delta^{3/2}_{\bar{x}}\pa^{\al}g_\vps}\|^2dt\right)^{1/2}
d\Si(\bar{k})
\\&
\quad
+\sum\limits_{|\al|\leq1}\left(\int_0^{T_0}\left\|\int_{\T^2} \pa^{\al}g_\vps d\bar{x}\right\|^2dt\right)^{1/2}
\notag \\
& \lesssim \left(\int_0^{T_0}\sum\limits_{\al_1\leq1,|\bar{\al}|\leq6}\left\|\pa^{\al} g_{\vps}(t)\right\|_{L^2_{x,v}}^2dt\right)^{1/2}\leq C_\vps<\infty.
\end{align*}
We have shown that $\pa^{\al}g_\vps \in L^1_{\bar{k}}L^2_{T_0}L^2_{x_1}L^2_{v,D}$ for $|\al|\leq 1$ where the bounds above depend upon $\vps>0$.

Secondly, we will now show that $g_\vps$ and $\na_x g_\vps$ are bounded   in $L^1_{\bar{k}}L^\infty_{T_0}L^2_{x_1}L^2_v\cap L^1_{\bar{k}}L^2_{T_0}L^2_{x_1}L^2_{v,D}$ uniformly for any $\vps>0$. In fact, turning to \eqref{lc.lb.vps}, we may perform the similar energy estimates as for obtaining \eqref{eng2} using \eqref{L1L2.es} to deduce that 
\begin{align*}
&\sum\limits_{|\al|\leq1}\int_{\Z^2}\sup\limits_{0\leq t\leq T_0}\|\widehat{\pa^{\al}g_\vps}(t,\bar{k})\|d\Si(\bar{k})
+\sqrt{2\de_0}\sum\limits_{|\al|\leq1}\int_{\Z^2}\left(\int_0^{T_0}\|\widehat{\pa^{\al}g_\vps}\|^2_{D}dt\right)^{1/2}d\Si(\bar{k})\notag 
\\
&
\lesssim \sum\limits_{|\al|\leq1}\int_{\Z^2}\|\widehat{\pa^{\al}g}_{0,\vps}\|d\Si(\bar{k})
+
\sum\limits_{|\al|\leq1}\int_{\Z^2}\left(\int_0^{T_0}\left|(\widehat{\pa^{\al}\Ga(\bh_\vps,g_\vps)},\widehat{\pa^{\al}g_\vps})\right|dt\right)^{1/2}d\Si(\bar{k})
\notag \\
&\qquad+\sum\limits_{|\al|\leq1}\int_{\Z^2}\left(\int_{0}^{T_0}|(\CL_2\widehat{\pa^{\al}\bh}_{\vps},\widehat{\pa^{\al}g_\vps})|dt\right)^{1/2}d\Si(\bar{k})
\notag \\
&\qquad
+
T_0\sum\limits_{|\al|\leq1}\int_{\Z^2}\sup\limits_{0\leq t\leq T_0}\|\widehat{\pa^{\al}g_\vps}(t,\bar{k})\|d\Si(\bar{k}).
\notag 
\end{align*}
The above upper bound is then further bounded by
\begin{align*}
&\lesssim \sum\limits_{|\al|\leq1}\int_{\Z^2}\|\pa^{\al}g_{0,\vps}\|d\Si(\bar{k})
+C_\eta
\sum\limits_{|\al|\leq1}\left\|\pa^{\al}g_\vps\right\|_{L^1_{\bar{k}}L^\infty_{T_0}L_{x_1}^2L^2_{v}}
\sum\limits_{|\al|\leq1}\left\|\pa^{\al}\bh_{\vps}\right\|_{L^1_{\bar{k}}L^2_{T_0}L_{x_1}^2L^2_{v,D}}
\notag \\
&\qquad+
C_\eta\sum\limits_{|\al|\leq1}\left\|\pa^{\al}g_\vps\right\|_{L^1_{\bar{k}}L^2_{T_0}L_{x_1}^2L^2_{v,D}}
\sum\limits_{|\al|\leq1}\left\|\pa^{\al}\bh_{\vps}\right\|_{L^1_{\bar{k}}L^\infty_{T_0}L_{x_1}^2L^2_{v}}
\notag \\
&\qquad+
C_\eta\sum\limits_{|\al|\leq1}\left\|\pa^{\al}\bh_{\vps}\right\|_{L^1_{\bar{k}}L^2_{T_0}L_{x_1}^2L^2_{v,D}}
+\eta\sum\limits_{|\al|\leq1}\left\|\pa^{\al}g_\vps\right\|_{L^1_{\bar{k}}L^2_{T_0}L_{x_1}^2L^2_{v,D}}
\notag \\
&\qquad
+
T_0\sum\limits_{|\al|\leq1}\int_{\Z^2}\sup\limits_{0\leq t\leq T_0}\|\widehat{\pa^{\al}g_\vps}(t,\bar{k})\|d\Si(\bar{k}).
\end{align*}
On the other hand, 
it is straightforward to check that
\begin{equation*}
\begin{split}
\sum\limits_{|\al|\leq1}&\left\|\pa^{\al}\bh_{\vps}\right\|_{L^1_{\bar{k}}L^\infty_{T_0}L^2_{x_1,v}}
+\sum\limits_{|\al|\leq1}\left\|\pa^{\al}\bh_{\vps}\right\|_{L^1_{\bar{k}}L^2_{T_0}L_{x_1}^2L^2_{v,D}}\\
\lesssim
& \sum\limits_{|\al|\leq1}\left\|\pa^{\al}\bh\right\|_{L^1_{\bar{k}}L^\infty_{T_0}L^2_{x_1,v}}
+\sum\limits_{|\al|\leq1}\left\|\pa^{\al}\bh\right\|_{L^1_{\bar{k}}L^2_{T_0}L_{x_1}^2L^2_{v,D}}
\leq C\eps_0,
\end{split}
\end{equation*}
and similarly
\begin{equation*}
\begin{split}
\sum\limits_{|\al|\leq1}\int_{\Z^2}\|\widehat{\pa^{\al}g}_{0,\vps}\|d\Si(\bar{k})\leq C\sum\limits_{|\al|\leq1}\int_{\Z^2}\|\widehat{\pa^{\al}g}_{0}\|d\Si(\bar{k}).
\end{split}
\end{equation*}
These two bounds above are independent of $\vps>0$.  
Consequently, we obtain the following uniform-in-$\vps$ estimate
\begin{multline}\label{gvps.eng2}
\sum\limits_{|\al|\leq1}\int_{\Z^2}\sup\limits_{0\leq t\leq T}\|\widehat{\pa^{\al}g_\vps}(t,\bar{k})\|d\Si(\bar{k})
+\sum\limits_{|\al|\leq1}\int_{\Z^2}\left(\int_0^T\|\widehat{\pa^{\al}g_\vps}\|^2_{D}dt\right)^{1/2}d\Si(\bar{k})\\
\lesssim \sum\limits_{|\al|\leq1}\int_{\Z^2}\|\widehat{\pa^{\al}g}_{0}\|d\Si(\bar{k})
+\sum\limits_{|\al|\leq1}\left\|\pa^{\al}\bh\right\|_{L^1_{\bar{k}}L^2_{T_0}L_{x_1}^2L^2_{v,D}}.
\end{multline}
This holds on a short time interval $T_0>0$.

Lastly, we prove that for $|\al|\leq 1$, the sequence $(\pa^{\al}g_{\vps})_{\vps>0}$ is Cauchy in the space $L^1_{\bar{k}}L^\infty_{T_0}L^2_{x_1}L^2_v\cap L^1_{\bar{k}}L^2_{T_0}L^2_{x_1}L^2_{v,D}$ as $\vps \to 0^+$.
To do so, it suffices to show that
\begin{multline}
\sum\limits_{|\al|\leq1}\int_{\Z^2}\sup\limits_{0\leq t\leq T_0}\|\widehat{\pa^{\al}g_{\vps_1}}-\widehat{\pa^{\al}g_{\vps_2}}\|d\Si(\bar{k})
\\
+\sum\limits_{|\al|\leq1}\int_{\Z^2}\left(\int_0^{T_0}\|\widehat{\pa^{\al}g_{\vps_1}}-\widehat{\pa^{\al}g_{\vps_2}}\|^2_{D}dt\right)^{1/2}d\Si(\bar{k})
\\
\lesssim \sum\limits_{|\al|\leq1}\int_{\Z^2}\|\widehat{\pa^{\al}g}_{0,\vps_1}-\widehat{\pa^{\al}g}_{0,\vps_2}\|d\Si(\bar{k})
+\sum\limits_{|\al|\leq1}\left\|\pa^{\al}\bh_{\vps_1}-\pa^{\al}\bh_{\vps_2}\right\|_{L^1_{\bar{k}}L^2_{T_0}L_{x_1}^2L^2_{v,D}}.\label{gvps.eng3}
\end{multline}
Above the implicit constant is uniform in $\vps>0$.  
Indeed, in terms of
\begin{eqnarray*}
\left\{\begin{array}{rll}
&\pa_t g_{\vps}+v\cdot\na_{x}g_{\vps}+\mathscr{L}_1g_{\vps}-\Ga(\bh_{\vps},g_{\vps})=-\mathscr{L}_2\bh_{\vps},\\
&g_{\vps}(0,x,v)=g_{0,\vps}(x,v),\\
&g_{\vps}(t,-1,\bar{x},v_1,\bar{v})|_{v_1>0}=g_{\vps}(t,-1,\bar{x},-v_1,\bar{v}),\\
&g_{\vps}(t,1,\bar{x},v_1,\bar{v})|_{v_1<0}=g_{\vps}(t,1,\bar{x},-v_1,\bar{v}),
\end{array}\right.
\end{eqnarray*}
the estimate \eqref{gvps.eng3} can be derived by the energy estimates  similar to obtaining \eqref{gvps.eng2}, and we omit the details for brevity. Therefore, there is a function $g$ with $g,\na_x g\in L^1_{\bar{k}}L^\infty_{T_0}L^2_{x_1}L^2_v\cap L^1_{\bar{k}}L^2_{T_0}L^2_{x_1}L^2_{v,D}$ such that $\pa^{\al}g_\vps\rightarrow \pa^{\al}g$ $(|\al|\leq1)$ in $L^1_{\bar{k}}L^\infty_{T_0}L^2_{x_1}L^2_v\cap L^1_{\bar{k}}L^2_{T_0}L^2_{x_1}L^2_{v,D}$ as $\vps\rightarrow0^+$. Moreover, it is direct using similar estimates to verify that the limit function $g$  is a unique solution to
\eqref{lc.lb} satisfying the symmetric property \eqref{lg.symm} and the estimate \eqref{lg.es}. This completes the proof of Lemma \ref{lb.loc}.
\end{proof}


\begin{proof}[Proof of Theorem \ref{le.th}.] We now construct the approximate solution sequence which is denoted by $(f^n(t,x,v))_{n=0}^\infty$ for the problem
\eqref{LLeq}, \eqref{idf} and \eqref{srb} using the following iterative scheme:
\begin{equation*}
\left\{\begin{aligned}
&\pa_t f^{n+1}+v_1\pa_{x_1}f^{n+1}+\bar{v}\cdot\na_{\bar{x}}f^{n+1}+\mathscr{L}_1f^{n+1}-\Ga(f^{n},f^{n+1})=-\mathscr{L}_2f^{n},\\
&f^{n+1}(0,x,v)=f_0(x,v),\\ 
&f^{n+1}(-1,\bar{x},v_1,\bar{v})|_{v_1>0}=f^{n+1}(-1,\bar{x},-v_1,\bar{v}),\\
&f^{n+1}(1,\bar{x},v_1,\bar{v})|_{v_1<0}=f^{n+1}(1,\bar{x},-v_1,\bar{v}),
\end{aligned}\right.
\end{equation*}
for $n=0,1,2,\cdots$,
where we have set $f^0(t,x,v)\equiv f_0(x,v)$.
With Lemma \ref{lb.loc} in hand, it is a standard  procedure to apply the induction argument to show that there are $\eps_0>0$ and $T_0>0$ such that if
\begin{equation*}
\|f_0\|_{L^1_{\bar{k}} L^2_{x_1,v}}+\|\na_x f_0\|_{L^1_{\bar{k}} L^2_{x_1,v}}\leq \eps_0,
\end{equation*}
then the approximate solution sequence $(f^n(t,x,v))_{n=0}^\infty$ is well-defined in the space $L^1_{\bar{k}}L^\infty_{T_0}L^2_{x_1}L^2_v\cap L^1_{\bar{k}}L^2_{T_0}L^2_{x_1}L^2_{v,D}$, it is also a Cauchy sequence in this function space.   Then the limit function $f(t,x,v)$ such that \eqref{le.th.space} holds true by similar estimates is indeed a unique solution of  \eqref{LLeq},
\eqref{idf} and \eqref{srb} satisfying the desired symmetric property and the estimate \eqref{loces}. For the proof of positivity, we can use the argument from \cite[page 833]{GS}; the details are omitted for brevity. The proof of Theorem \ref{le.th} is complete.
\end{proof}

\appendix
\section{Basic estimates}\label{sec.app}
\setcounter{equation}{0}

In this appendix, we collect some known basic estimates for the linearized Landau operator or Boltzmann operator $L$ as in \eqref{Ldef}.
We will use the following estimates which were proven in \cite[Corollary 1 and Lemma 5, page 400]{Guo-L}:

\begin{lemma}[Landau case]\label{esLL}
Let $L$ be defined as \eqref{Ldef}.
Suppose $\gamma\geq-3$, then there exist two generic constants
$\de_0,\ C>0$, such that
\begin{equation*}
\dis \de_0|\{{\bf I-P}\}g|_{D}^2\leq ( Lg,g)_{L^2_v},
\end{equation*}
\begin{equation*}
|g|_{D}^2\geq C\left\{\left|\langle v\rangle^{\frac{\gamma}{2}}\{{\bf
P}_v\partial_jg\}\right|_2^2
+\left|\langle v\rangle^{\frac{\gamma+2}{2}}\{({\bf I}-{\bf
P}_v)\partial_jg\}\right|_2^2
+\left|\langle v\rangle^{\frac{\gamma+2}{2}}g\right|_2^2\right\},
\end{equation*}
where ${\bf P}_v$ is the projection defined as
$$
{\bf P}_vh_j=\sum_{m=1}^{3} {h_mv_m}\frac{v_j}{|v|^2}, \ \   j\in \{1,2,3\},
$$
for any vector-valued function $h(v)=[h_1(v),h_2(v),h_3(v)]$.
\end{lemma}

For the Boltzmann case \eqref{bBop} with \eqref{Ldef}, recalling also Remark \ref{norm.remark}, we have the following linearized estimate from \cite[Equation (2.13), page 784]{GS}.

\begin{lemma}[Boltzmann case]\label{esBL}
Let $L$ be defined as \eqref{Ldef}.  For $0<s<1$ and $\ga>-3$.  Then there exists a uniform quantitative constant $C_0>0$ such that
\begin{equation*}
\dis \frac{1}{C_0}|\{{\bf I-P}\}g|_{D}^2\leq ( Lg,g)_{L^2_v}\leq C_0 |\{{\bf I-P}\}g|_{D}^2,
\end{equation*}
\end{lemma}

This above Lemma above is similarly  also independently contained in \cite{AMUXY-2012-JFA}.  Note that from \cite{AMUXY-2012-JFA} and also similarly in \cite[Equation (2.15), page 784]{GS}, recalling again Remark \ref{norm.remark}, one has that
\begin{equation}\notag
\dis \de_0\left\{|g|^2_{H^s_{\ga/2}}+\left|\langle v\rangle^{\frac{\gamma+2s}{2}}g\right|^2_2\right\}\leq |g|_{D}^2\leq C |g|^2_{H^s_{s+\ga/2}},
\end{equation}
where $\delta_0, C>0$.  Here the weighted fractional Sobolev norm $|g(v)|_{H^s_\ell}^2=\left|\langle v\rangle^{\ell} g(v)\right|_{H^s}^2$
is given by
\begin{equation*}
\left|g\right|^2_{H^{s}_{\ell}}
=
\left|
\langle v\rangle^{\ell} g\right|_{L^2_{v}}^2+\int_{{\R}^3}dv\int_{{\R}^3}du\,
\frac{\left[\langle v\rangle^{\ell}g(v)-\langle
u\rangle^{\ell}g(u)\right]^2}{|v-u|^{3+2s}}{\chi}_{|v-u|\leq1}.
\end{equation*}
This norm above turns out to be equivalent with
\begin{equation*}
\left|g\right|^2_{H^{s}_{\ell}}=\int_{{\R}^3}dv\, \left|\left(1-\Delta_v\right)^{\frac{s}{2}}\left(w^\ell(v)
g(v)\right)\right|^2.
\end{equation*}
We refer to \cite{AMUXY-2012-JFA} for additional details.

Lastly, we give the following uniform weighted estimates for both the linear Landau and Boltzmann operator as follows.

\begin{lemma}[Weighted estimates]\label{wgesL}
Let $L$ be given by \eqref{Ldef}, we have the estimate
\begin{equation*}
\left( Lg,w^{2}_{q,\vth}g\right)_{L^2_v}\geq
\delta_q\left|w_{q,\vth}g\right|^2_{D}
-C\left|g\right|_{L^2({B_R})}^2,
\end{equation*}
where $\de_q, C>0$, and  $B_R$ denotes the closed ball in $\R^3_v$ with center zero and radius a constant $R>0$.  Here $(q,\vth)$ are given as in {\bf (H)} in \eqref{q} for the Boltzmann case and the Landau case respectively.
\end{lemma}

The lemma above is proven in \cite[Lemma 9, page 323]{SG-08-ARMA} for the Landau case.  For the Boltzmann case, such an estimate is proven in \cite[Lemma 2.6, page 783, Equation (2.10)]{GS} with polynomial weights.  To prove the estimate with exponential weights as in \eqref{def.w} with \eqref{q}
it will follow directly from using the estimate \cite[Lemma 2.6, page 783, Equation (2.10)]{GS} or \cite{AMUXY-2012-JFA}, combined with the techniques for handling the exponential weight in the proof of \cite[Lemma 9, page 323]{SG-08-ARMA} following the restrictions in  \eqref{q}.    See also \cite{DLYZ-VMB}.


\medskip
\noindent {\bf Acknowledgements:}
RJD was supported by the General Research Fund (Project No.~14302716) from RGC of Hong Kong and a Direct Grant from CUHK. SQL was supported by grants from the National Natural Science Foundation of China (contracts: 11471142, 11571063 and 11731008). SS was supported by JSPS Kakenhi Grant (No.~18J00285). RJD and SS would like to thank Professor Yoshinori Morimoto for stimulating discussions on the topic when RJD visited the Kyoto University in March 2018. SQL and SS would also like to thank Department of Mathematics, CUHK for the host of their visit.  RMS was partially supported by the NSF grants DMS-1500916 and DMS-1764177 of the USA.  RMS would further like to thank RJD and the Chinese University of Hong Kong for their hospitality during his visit in July 2018.



\bibliographystyle{habbrv}  




\end{document}